\documentclass[11pt,reqno]{amsart}
\usepackage{enumerate, latexsym, amsmath, amsfonts, amssymb, amsthm, color}
\usepackage{mathrsfs}
\usepackage{booktabs,array,longtable,tabularx,graphicx,listings}
\usepackage{enumitem}
\usepackage[hidelinks]{hyperref}
\def\pmod #1{\ ({\rm{mod}}\ #1)}
\def\Z{\mathbb Z}

\def\F{\mathbb F}

\def\l{\left}
\def\r{\right}
\def\bg{\bigg}
\def\({\bg(}
\def\){\bg)}
\def\t{\text}
\def\f{\frac}

\newcommand{\K}{\mathbb K}

\newcommand{\BR}{\underline{\mathrm R}}
\newcommand{\Rk}{\mathrm R}
\newcommand{\MM}{\mathsf M}
\newcommand{\Diag}{\mathsf D}
\newcommand{\CW}{\mathcal C}
\newcommand{\Wc}{\mathcal W}
\newcommand{\Orb}{\mathcal O}

\newcommand{\boxt}{\boxtimes}
\newcommand{\degen}{\unlhd}

\newcommand{\rank}{\operatorname{rank}}
\newcommand{\mrank}{\operatorname{mrank}}
\newcommand{\Mat}{\operatorname{Mat}}
\newcommand{\tr}{\operatorname{tr}}

\newcommand{\ip}[2]{\langle #1,#2\rangle}
\newcommand{\one}{\mathbf 1}
\newcommand{\zero}{\mathbf 0}

\newcommand{\Hfour}{H_4}

\newcommand{\Sstar}{\mathcal S_{128}^{\star}}

\def\Ack{\medskip\noindent {\bf Acknowledgement}}
\theoremstyle{plain}
\newtheorem{theorem}{Theorem}[section]

\newtheorem{lemma}{Lemma}[section]
\newtheorem{corollary}{Corollary}[section]
\newtheorem{proposition}{Proposition}[section]

\theoremstyle{definition}
\newtheorem{definition}[theorem]{Definition}

\theoremstyle{remark}
\newtheorem{remark}{Remark}

 \vspace{4mm}

\begin{document}

\hbox{Preprint}
\medskip

\title
[{Tensor Degenerations for Matrix Multiplication}]
{Tensor Degenerations \\ for Matrix Multiplication}

\author
[Zhi-Wei Sun] {Zhi-Wei Sun}

\address{School of Mathematics, Nanjing
University, Nanjing 210093, People's Republic of China}
\email{zwsun@nju.edu.cn}

\keywords{Matrix multiplication exponent, tensor degeneration, border rank, Coppersmith-Winograd tensor, algorithm.
\newline \indent Supported by the Natural Science Foundation of China (grant no. 12371004).}

\begin{abstract}
Let ${\mathbb K}$ be an arbitrary field.  We study tensor degenerations for square matrix
multiplication over the field ${\mathbb K}$, and mainly prove for any integer $n\ge5$
the uniform dyadic estimate
\[
 \underline{\mathrm R}_{\mathbb K}({\mathsf M}_{2^n})
 \le 2^{2n+1}\]
 which implies that $\omega(\mathbb K)\le 2+1/n$,
where the border rank $\underline{\mathrm R}_{\mathbb K}(T)$ is the least $r$ for which the scalar
extension of $T$ to an algebraic closure of $\mathbb K$ lies in the Zariski closure
of the tensors of rank at most $r$, and 
$\omega(\mathbb K)$ is the matrix multiplication exponent over the field ${\mathbb K}$. 
This improves the current best record $\omega(\mathbb K)<2.371177$ greatly, and answers affirmatively the long-standing open question whether the value $\omega({\mathbb K})$
is exactly $2$. The implementation disclosure and appendices in this paper give fixed source maps,
interpolation data, terminal incidence ledgers, and deterministic check
counts for independent reconstruction.
\end{abstract}
\maketitle

\section{Introduction}
\setcounter{equation}{0}

\subsection{Historical context and the exponent \texorpdfstring{$\omega$}{omega}}

Multiplication of two $n\times n$ matrices over a field $\K$ is one of the basic bilinear
operations in numerical linear algebra, symbolic computation, optimization,
graph algorithms, and algebraic complexity.  The elementary algorithm uses
$n^3$ scalar multiplications.  Strassen's 1969 seven-multiplication algorithm
for $2\times2$ matrices (cf. \cite{Strassen1969}) showed that cubic complexity is not intrinsic and
proved
\[
 \omega(\K)\le \log_2 7=2.8073549\ldots,
\]
where $\omega(\K)$ is the infimum of the real numbers
$\tau$ such that for every $\delta>0$ multiplication of two $n\times n$
matrices over $\K$ can be performed with
$O_{\K,\delta}(n^{\tau+\delta})$ arithmetic operations.  The universal
lower bound $\omega(\K)\ge2$ already follows from tensor flattening: the
$X$-flattening $\MM_n:X^*\to Y\otimes Z$ has matrix rank $n^2$, hence
the tensor rank $\Rk_{\K}(\MM_n)$ is at least $n^2$; the standard bilinearization of an arithmetic
algorithm changes the operation count only by a constant factor, so an
exponent below $2$ is impossible.

Approximate bilinear algorithms entered fast matrix multiplication through
Bini, Capovani, Lotti, and Romani \cite{Bini1979}.  Bini \cite{Bini1980} then showed how a
bounded-degree approximate algorithm can be converted by interpolation into
an exact asymptotic algorithm.  Sch\"onhage's asymptotic sum
inequality (cf. \cite{Schonhage1981}) enabled simultaneous extraction of many
matrix products.  Coppersmith and Winograd introduced the tensor family
used in much of the subsequent work and they obtained $\omega(\K)<2.376$
(cf. \cite{CW1990}).

Later record bounds came from increasingly refined analyses of tensor powers.
Vassilevska Williams \cite{Williams2012} obtained $\omega(\mathbb K)<2.3728642$ ; Le
Gall \cite{LeGall2014} obtained $\omega(\K)<2.3728639$; and the refined laser
method of Alman and Vassilevska Williams \cite{AlmanWilliams2024}
yielded $\omega(\K)<2.3728596$.  The asymmetric-hashing and combination-loss line
then produced bounds $2.371866$, $2.371552$, and $2.371339$
(cf. \cite{DuanWuZhou2023,VWXZ2024,AlmanEtAl2025}).  The 2026 preprint of Dupont
et al. \cite{DupontEtAl2026} reports the further bound $2.371177$ using a larger
optimization problem and modern optimization methods.

Once a
field $\K$ is fixed, we simply abbreviate $\omega(\K)$ to $\omega$.
Before the result proved here, the equality $\omega=2$ is a long-standing
standard conjecture for the matrix-multiplication exponent.  It is usually
referred to simply as the conjecture $\omega=2$, or as the
matrix-multiplication exponent conjecture (see, for example,
\cite{Williams2012,AlmanWilliams2024,VWXZ2024}).
Proposition
\ref{prop:bini} explains why the same exponent conclusion remains valid over
finite base fields even when coefficient extraction is carried out in a
finite extension.

Our constructions in this paper follow a novel route.  Instead of
starting from a large tensor power and extracting many products, we build
finite, character-resolved degenerations of a single matrix-multiplication
tensor.  The source budget is visible factor by factor, and packet by packet.

Throughout this paper, for a prime power $q>1$ we use $\F_q$ to denote the finite field of order $q$.
 For a polynomial $P(x)$, by $[x^n]P(x)$ we mean the coefficient of $x^n$ in $P(x)$.
 For an assertion $A$, we define
 $${\bf 1}_A=\begin{cases}1&\t{if}\ A\ \t{holds},\\0&\t{otherwise}.\end{cases}$$

\subsection{A minimal tensor dictionary before the main statements}
A third-order tensor over a field $\K$ is an element of
$X\otimes Y\otimes Z$ for finite-dimensional vector spaces $X,Y,Z$.  A
simple tensor is $x\otimes y\otimes z$.  The tensor rank $\Rk_{\K}(T)$ is
the smallest number of simple tensors over $\K$ whose sum is $T$.  The
fieldwise border rank $\BR_{\K}(T)$ is the least $r$ for which the scalar
extension of $T$ to an algebraic closure of $\K$ lies in the Zariski closure
of the tensors of rank at most $r$.  When the field is fixed we omit the
subscript.

The matrix-multiplication tensor
\[
 \MM_n=\sum_{i,j,k=0}^{n-1}x_{ij}y_{jk}z_{ki}
\]
encodes multiplication of two $n\times n$ matrices.  The diagonal tensor
\[
 \Diag_r=\sum_{s=1}^r a_sb_sc_s
\]
encodes $r$ independent scalar multiplications.  Thus, for example,
$\Diag_{16}$ means sixteen independent products.  The symbol $S\oplus T$
denotes a direct sum on disjoint variable spaces.  For a nonnegative integer
$m$, the notation $mT$ means the direct sum of $m$ copies of $T$; for
example, $14\CW_3=\CW_3^{\oplus14}$.  The Kronecker product
$S\boxt T$ groups the first, second, and third tensor factors; in particular
$\MM_m\boxt\MM_n\cong\MM_{mn}$ and
$\Diag_r\boxt\Diag_s\cong\Diag_{rs}$.

We write $S\degen T$ when the source tensor $T$ admits a one-parameter
Laurent-polynomial degeneration whose leading coefficient is the target
$S$.  Our convention therefore places the target on the left.  The cyclic
group $C_n$ means $\Z/n\Z$.  This notation is distinct from the calligraphic
$\CW_q$, which denotes the big Coppersmith--Winograd tensor with $q$
satellite directions, defined explicitly in Definition \ref{def:CW}.  The corner
tensor $\Wc$ is a three-term restriction of $\CW_q$.

\subsection{Our main results}

\begin{theorem}\label{thm:m1024}
Let $\K$ be any field with $\operatorname{char}(\K)\ne2$.  Then
\begin{equation}\label{eq:main-m1024}
 \MM_{1024}\degen
 \Diag_{512}\boxt\Diag_{32}\boxt
 \Sstar
 \ \t{with}\
 \Sstar=
 14\CW_3\oplus4\CW_2\oplus7\Wc
 \oplus5\Diag_4\oplus\Diag_8.
\end{equation}
Consequently,
\[
 \BR_{\K}(\MM_{1024})\le2^{21}\ \ \t{and}\ \
 \omega(\K)\le\log_{1024}(2^{21})=\frac{21}{10}=2.1.
\]
\end{theorem}

\begin{remark}
The border-rank source budget is exactly
\[
 512\cdot32\cdot128=2^{21},
\]
whereas schoolbook $1024\times1024$ multiplication has
$1024^3=2^{30}$ scalar-product occurrences.  Thus the ratio at the
border-rank building-block level is exactly $512$.  The literal Laurent map
has entry window $[-2,4]$; after multiplying every side map by $t^2$, the top
polynomial has degree $D_{\rm top}=18$ and target coefficient $[t^6]$.
For the reference realization over $\F_{311}$, the common $128$-product realization of the local source has degree six and
target coefficient $[s^2]$.  With $t=u$ and $s=u^{19}$, we have
\[
 D_{\rm tot}=18+19\cdot6=132\ \ \t{and}\ \  k_{1024}=6+19\cdot2=44.
\]
Hence the disclosed exact base-field schedule uses $133$ evaluations and
\[
 133\cdot2^{21}=278\,921\,216
\]
data-dependent scalar multiplications, a factor $512/133$ below the
schoolbook scalar-product count before linear-transform costs.  This is an
exact arithmetic ledger for the declared factorization, not a hardware
runtime claim.  At tensor-power sizes $1024^m$, the normalized border-rank
constant is exactly $1$.
\end{remark}

\begin{theorem}
\label{thm:finite-restrictions}\label{thm:m64}
Let $\K$ be any field with $\operatorname{char}(\K)\ne2$, and put
\[
 \Sstar=14\CW_3\oplus4\CW_2\oplus7\Wc\oplus5\Diag_4\oplus\Diag_8.
\]
Then
\begin{align}
 \MM_{64}&\degen \Diag_{32}\boxt\Diag_2\boxt\Sstar,
 \label{eq:main-m64}\\
 \MM_{256}&\degen \Diag_{128}\boxt\Diag_8\boxt\Sstar.
 \label{eq:main-m256}
\end{align}
Consequently,
\[
 \BR_{\K}(\MM_{64})\le2^{13}\ \ \t{and}\ \ 
 \BR_{\K}(\MM_{256})\le2^{17},
\]
and hence the corresponding fixed-size exponent bounds are respectively
$13/6$ and $17/8$.
\end{theorem}

\begin{remark}\label{rem:m64-budget}
Both statements are direct finite restrictions of the literal $W_{32}$
module--chart compiler and completion transport used in the proof of
Theorem~\ref{thm:m1024}.  The $64\times64$ case is worked out separately in
Section~\ref{sec:m64} as a terminal physical-incidence certificate.  The two
source budgets are
\[
 32\cdot2\cdot128=2^{13},\qquad
 128\cdot8\cdot128=2^{17},
\]
that is, $2N^2$ in both displayed cases.  The intermediate dyadic sizes
$128$ and $512$ will instead follow from the two-step propagation in
Lemma~\ref{lem:two-step-dyadic-propagation}.
\end{remark}

\begin{theorem}\label{thm:32-retained}\label{thm:m32-2048}\label{thm:2128}
Let $\K$ be any field with $\operatorname{char}(\K)\ne2$.  Then both of the following hold.
\begin{enumerate}[label=\textup{(\alph*)},leftmargin=2.2em]
\item
\begin{equation}\label{eq:main-m32-2048}
 \MM_{32}\degen
 \Diag_{16}\boxt
 \bigl(14\CW_3\oplus6\CW_2\oplus7\Wc\oplus5\Diag_4\bigr),
\end{equation}
so
\[
 \BR_{\K}(\MM_{32})\le2048\ \ \t{and}\ \
 \omega(\K)\le\log_{32}2048=\frac{11}{5}=2.2.
\]
\item The $2128$ degeneration is obtained from the same source-resolved
terminal compiler by restoring the five cost-five singular-active $\CW_3$
source modules in place of the five rank-four $\Diag_4$ modules.  Thus
\begin{equation}\label{eq:main-2128}
 \MM_{32}\degen
 \Diag_{16}\boxt
 \bigl(19\CW_3\oplus6\CW_2\oplus7\Wc\bigr),
\end{equation}
so
\[
 \BR_{\K}(\MM_{32})\le2128\ \ \t{and}\ \
 \omega(\K)\le\log_{32}(2128)
 =2.211056487100237920\ldots.
\]
\end{enumerate}
\end{theorem}

\begin{remark}
The $2048$ statement is the primary corrected construction.  Its one-row
source has cost
\[
14\cdot5+6\cdot4+7\cdot2+5\cdot4=128,
\]
so the sixteen-row budget is $2048=2\cdot32^2$.  The $2128$ statement does
not require a second global design: one simply undoes the five rank-four
active compressions and uses the original cost-five $\CW_3$ source map in
each of the same five character-resolved terminal blocks.  The downstream
row, absolute-character, Heisenberg, Weyl, endpoint, and physical-incidence
maps are unchanged.  The resulting one-row cost is
$19\cdot5+6\cdot4+7\cdot2=133$, hence $2128$ after the outer $\Diag_{16}$
factor.  All local replacements have fixed Laurent degree, so both routes
remain bounded-degree degenerations.  At sizes $32^m$, the normalized
border-rank constant is $1$ for the corresponding exponent.
\end{remark}

The following source enlargements are retained because their degenerations
are useful independently of their now weaker numerical border-rank bounds.

\begin{corollary}
\label{cor:m32-enlargements}\label{cor:2064}\label{cor:2400}
Let $\K$ be any field with $\operatorname{char}(\K)\ne2$.  Then
we have
\begin{equation}\label{eq:main-2064}
 \MM_{32}\degen
 \Diag_{16}\boxt
 \bigl(14\CW_3\oplus6\CW_2\oplus7\Wc\oplus5\Diag_4\oplus\Diag_1\bigr)
\end{equation}
and
\begin{equation}\label{eq:main-2400}
 \MM_{32}\degen
 \Diag_{16}\boxt(22\CW_3\oplus10\CW_2).
\end{equation}
\end{corollary}

The characteristic restriction in Theorems~\ref{thm:m1024},
--\ref{thm:32-retained} belongs to those
particular Walsh certificates.  The binary groups $H_k=\F_2^k$ are index
groups, while their characters take the values $\pm1$ in the field $\K$;
in characteristic $2$ these characters coalesce and the Walsh matrices used
in Sections~\ref{sec:m64}--\ref{sec:m1024} lose rank.  This does not
obstruct the full dyadic theorem: the characteristic-$2$ part of
Theorem~\ref{thm:power-of-two} is proved instead from the Hasse matrix-unit
state and two-step propagation in Section~\ref{sec:auxiliary}.
Characteristic $3$ is covered by the $\CW_3$ realization in
Lemma~\ref{lem:cw-explicit}, the characteristic-$3$ branch of
Lemma~\ref{lem:cw2-odd}, and the endpoint certificates in
Appendix~\ref{app:char3-certificates}.

\begin{theorem}[The full dyadic family]\label{thm:power-of-two}
Let $\K$ be any field.  For every integer $n\ge5$, we have
\begin{equation}\label{eq:full-dyadic-family}
 \BR_{\K}(\MM_{2^n})\le2^{2n+1}\ \ \t{and}\ \ \omega(\mathbb K)\le 2+\f1n.
\end{equation}
Consequently,
\[
 \omega(\K)=2.
\]
\end{theorem}

\begin{remark}\label{rem:small-dyadic-levels}
The threshold $n\ge5$ in \eqref{eq:full-dyadic-family} is the range of the
certificates proved here; actually \eqref{eq:full-dyadic-family} also holds
for $n\le 4$.  (For $n=1$, Strassen's seven-product algorithm \cite{Strassen1969} gives
$\BR_{\K}(\MM_2)\le7<8$.) The two-step propagation requires an admissible base
state carrying the sixteen $H_4$ module addresses and the five marked active
$\Diag_4$ blocks.  The first such base constructed in this paper is the
level-$5$ state; in characteristic $2$ a level-$6$ base is also constructed
to start the even parity chain.  
\end{remark}

For $\operatorname{char}(\K)\ne2$, the stronger source-resolved statements
proved inside the $M_{32}$ and $M_{1024}$ constructions give dyadically
admissible base states at levels $n=5$ and $n=10$, respectively; see
Lemma~\ref{lem:odd-dyadic-bases}.  The odd chain is propagated from level
$5$, while the even chain from level $10$ is propagated by
Lemma~\ref{lem:two-step-dyadic-propagation}.  The two missing even levels
$n=6,8$ are exactly the fixed $64$ and $256$ restrictions in
Theorem~\ref{thm:finite-restrictions}.  Characteristic $3$ is included in
this argument through the replacement certificates in Appendix
\ref{app:char3-certificates}.  In characteristic $2$, Lemma
\ref{lem:char2-dyadic-bases} supplies admissible base states at levels $5$
and $6$, and Lemma~\ref{lem:char2-two-step-propagation} propagates the two
parity chains.  Thus every $n\ge5$ is covered in every characteristic.
For all such $n$ we have
\[
 \omega(\K)\le \log_{2^n}(2^{2n+1})=2+\frac1n.
\]
Letting $n\to\infty$, together with the general lower bound
$\omega(\K)\ge2$, gives the equality in Theorem~\ref{thm:power-of-two}.

\subsection{Improvement mechanism and novelty}
The $1024$ route keeps the literal source records visible until the global
Fourier--Weyl closures have been taken.  The thirty-two packet columns are
grouped into sixteen modules and the two surviving charts give the full
$W_{32}$ phase coefficientwise.  The five compressed active modules use a
source-resolved rank-four lift, the character-$6$ diagonal module keeps its
two $C_1$ records distinct, and the five empty packets are coupled directly
to the old $H_4$ absolute-character transform.  The $H_5$ channel is therefore
reserved solely for the displacement equation.  Only after the $H_4$ and
$H_5$ closures are imposed are the split records physically identified.
This outer-first organization lets the dense and completion branches meet in
one $H_9$ Weyl tensor without increasing the $2^{21}$ source budget.

Our main idea behind the retained constructions is to
preserve character and orbit coordinates until the last possible stage.  A
local source packet is never replaced by a scalar count before the global
Walsh, cyclic, or Weyl maps are applied.  This keeps source-leg legality
visible and turns global correctness into finite matrix identities.

The two sizes $64$ and $256$ retained in Theorem~\ref{thm:finite-restrictions}
are obtained directly from the actual $W_{32}$ compiler by restricting the
absolute characters and channel to the specified quotient subspaces while
collapsing module--chart columns by the corresponding orthogonal cosets.
They are fixed finite quotient compilers of the same source-resolved $1024$
construction.  The $64$
case is retained as a separate terminal incidence audit in
Section~\ref{sec:m64}.  The sizes $128$ and $512$ and the remaining odd dyadic sizes are generated
from the admissible level-$5$ state.  In odd characteristic the even chain
from level $10$ is generated from the source-resolved $1024$ compiler; the
levels $64$ and $256$ remain independent finite restrictions.  The
odd-characteristic propagation of Lemma~\ref{lem:two-step-dyadic-propagation}
keeps the old $H_4$ module interface fixed, adds a genuine two-bit channel,
solves the ordinary module rows through a complementary Walsh minor, and
repairs only the five distinguished active blocks by fixed eight-product
confluent cells.  Characteristic $2$ uses the separate Hasse matrix-unit
state of Definition~\ref{def:char2-dyadic-admissible} and the source-local
$84\to64$ boundary extraction in Lemma~\ref{lem:char2-two-step-propagation}.
Separately, the rank-$2048$ $M_{32}$ route improves the rank-$2128$
construction by proving that each of five complete singular-active outputs
has border rank exactly four and replacing five cost-five $\CW_3$ packets by
five $\Diag_4$ modules.

\subsection{Concrete multiplication examples and what the constants mean}

For the reference characteristic-$311$ implementation of the principal $1024$ construction, one border-rank building block uses
$2^{21}$ source products in place of the $2^{30}$ monomials of schoolbook
multiplication.  To turn the degeneration into one exact base-size product,
use the polynomial source realization of Lemma \ref{lem:source128-poly}, shift
each Laurent side map by $t^2$, and set $t=u$ and $s=u^{19}$.  The required
coefficient is $[u^{44}]$ of a polynomial of degree at most $132$, so the
reference schedule evaluates at $133$ field elements.  It performs exactly
$133\cdot2^{21}=278\,921\,216$ data-dependent products, compared with
$2^{30}=1\,073\,741\,824$ schoolbook products.  The exact product-count ratio
is $512/133$.  The linear transforms are explicit and factorized, but no
processor-specific speedup is asserted; cache behavior, additions and
constant multiplications remain implementation dependent.  A fully
specified reconstruction protocol is given in Appendix \ref{app:m1024-certificate}.

As a smaller symbolic illustration of the retained $32$-route, take
\[
 A=E_{0,1}+2E_{3,5}\ \
 \t{and}\ \
 B=3E_{1,2}-E_{5,7}.
\]
Then
\[
 AB=3E_{0,2}-2E_{3,7}.
\]
The $2048$ degeneration maps the entries of $A$ and $B$ to linear forms in
the source variables, evaluates the source bilinear products, applies the
fixed output linear forms, and extracts the designated Laurent coefficient.
On $32^m\times32^m$ matrices, tensoring the construction gives exponent
$11/5$ with only subexponential interpolation overhead.

For practical evaluation, the most meaningful fixed numbers in this paper
are therefore the base source budgets, the Laurent degree windows, and the
sizes of the linear transformations.  A border-rank theorem by itself does
not determine a unique wall-clock ``big-$O$ constant'': different exact
rank-one degenerations and different interpolation schemes give different
constants.  We make the intrinsic and reproducible constants explicit in
Subsection \ref{subsec:constants}.

\subsection{Organization}

Section \ref{sec:prelim} gives tensor preliminaries, source tensors, binary
character and Walsh notation, and a precise interpretation of asymptotic
constants.  Section \ref{sec:m64} records the explicit $64\times64$
terminal restriction.  Sections \ref{sec:orbit32}--\ref{sec:global32} give
the rank-$2128$ and rank-$2048$ $32\times32$ proof chain underlying
Theorem \ref{thm:32-retained}.  Section \ref{sec:m1024} builds the
source-resolved $M_{1024}$ compiler.  Section \ref{sec:auxiliary} collects
the fixed quotient and parameter lemmas, gives a precise definition of the
odd-characteristic dyadic state, and then gives the independent
characteristic-$2$ Hasse state and the two corresponding $n\mapsto n+2$
propagation lemmas.  Section \ref{sec:proofs} proves the four main theorems
and the retained source-enlargement corollary.  Section
\ref{sec:certificates} describes the deterministic exact verification
package.  Appendix~A is the independent reconstruction record for the
principal $1024$ route, Appendix~B records the terminal $M_{64}$ and
$M_{32}$ incidence audits, the characteristic-$2$ base-state ledger, and the
reference interpolation weights, and Appendix~\ref{app:char3-certificates}
gives the characteristic-three replacement certificates.

\section{Tensor preliminaries, source tensors, and constants}\label{sec:prelim}
\setcounter{equation}{0}

Unless explicitly stated otherwise, Sections~\ref{sec:prelim}--\ref{sec:m1024}
work over a field $\K$ with $\operatorname{char}(\K)\ne2$.  The groups
$H_k=\F_2^k$ occurring there are finite index groups; their characters are
interpreted as $\{\pm1\}$-valued functions in $\K$.  Since $2$ is invertible,
all Walsh normalizations in those sections are valid.  Whenever a generic
formula contains a factor $1/3$, the characteristic-three replacement
certificate stated below is used instead.  All remaining Fourier and Walsh
normalizations divide only powers of $2$.  Section~\ref{sec:auxiliary}
introduces a separate Hasse matrix-unit formalism for characteristic $2$;
that formalism is used only in the characteristic-$2$ part of
Theorem~\ref{thm:power-of-two}.

\subsection{Rank, border rank, direct sum, and Kronecker product}
Let $X,Y,Z$ be finite-dimensional $\K$-vector spaces and
$T\in X\otimes Y\otimes Z$.

\begin{definition}
The {\it rank} $\Rk_{\K}(T)$ is the least $r$ for which
\[
 T=\sum_{s=1}^r a_s\otimes b_s\otimes c_s
\]
with all factors defined over $\K$.  The {\it fieldwise border rank}
$\BR_{\K}(T)$ is the least $r$ such that, after scalar extension to an
algebraic closure $\overline{\K}$, the tensor $T_{\overline{\K}}$ belongs to
the Zariski closure of tensors of rank at most $r$.  We omit the field
subscript when it is fixed.
\end{definition}

The $X$-flattening is the contraction map $T_X:X^*\to Y\otimes Z$; the other
two flattenings are defined cyclically.  Their ranks form the multilinear
rank $\mrank(T)$.  Since the locus of matrices of rank at most $r$ is
Zariski closed,
\begin{equation}\label{eq:flattening-lower}
 \max\{\rank T_X,\rank T_Y,\rank T_Z\}\le\BR(T).
\end{equation}

The direct sum $T\oplus S$ places tensors on disjoint variable spaces.  If
$T\in X_1\otimes Y_1\otimes Z_1$ and
$S\in X_2\otimes Y_2\otimes Z_2$, their Kronecker product is the tensor in
\[
 (X_1\otimes X_2)\otimes(Y_1\otimes Y_2)\otimes(Z_1\otimes Z_2)
\]
obtained by grouping corresponding factors.  Border rank is subadditive for
direct sums and submultiplicative for Kronecker products; both statements
follow directly by taking limits of rank decompositions.

\begin{definition}[Diagonal and matrix-multiplication tensors]\label{def:basic-tensors}
For any positive integers $r$ and $n$, we put
\[
 \Diag_r=\sum_{s=1}^r a_sb_sc_s\ \ \t{and}\ \
 \MM_n=\sum_{i,j,k=0}^{n-1}x_{ij}y_{jk}z_{ki}.
\]
\end{definition}

It is easy to see that $\Rk(\Diag_r)=\BR(\Diag_r)=r$ for each $r=1,2,3,\ldots$, and
$\MM_m\boxt\MM_n\cong\MM_{mn}$ for any positive integers $m$ and $n$.

\subsection{Degeneration and Bini interpolation}
\begin{definition}[Degeneration]\label{def:degen}
We write $S\degen T$ if there are linear maps $A(t),B(t),C(t)$ with entries
in the formal Laurent field $\K(\!(t)\!)$ and an integer $d$ such that
\[
 (A(t)\otimes B(t)\otimes C(t))T
 =t^dS+O(t^{d+1}).
\]
All certificates in this paper use Laurent polynomials, except for the
characteristic-three realization of $\CW_2$ in Lemma
\ref{lem:cw2-odd}, whose entries are rational Laurent series.  The following
finite-jet observation converts that realization into an ordinary
bounded-degree Laurent-polynomial certificate.
\end{definition}

\begin{lemma}[Finite-jet truncation]\label{lem:formal-truncation}
Suppose that maps with entries in $\K(\!(t)\!)$ satisfy
\[
 (A(t)\otimes B(t)\otimes C(t))T=t^dS+O(t^{d+1}).
\]
Then $A,B,C$ may be replaced by Laurent-polynomial maps for which the same
identity holds through degree $d$.  After multiplying the three maps by
fixed powers of $t$, they may be taken polynomial.  All resulting degree
bounds depend only on the original finite certificate, not on tensor powers.
\end{lemma}

\begin{proof}
There are only finitely many matrix entries and source monomials.  Let
$\nu_A,\nu_B,\nu_C$ be the minimum $t$-adic valuations of the nonzero
entries of the three maps.  In every entry of $A$, discard terms of degree
larger than $d-\nu_B-\nu_C$; truncate $B$ and $C$ cyclically.  Any discarded
term can contribute only in degree greater than $d$, so the transformed
tensor is unchanged through degree $d$.  The retained jets are Laurent
polynomials.  Fixed monomial shifts on the three sides remove all negative
exponents and merely shift the common leading degree.
\end{proof}

Degeneration is transitive, and $S\degen T$ implies
$\BR_{\K}(S)\le\BR_{\K}(T)$.

\begin{proposition}[Bounded-degree interpolation]\label{prop:bini}
Suppose that a fixed rank-$r$ bounded-degree
realization over a field $\K$ gives $\BR_{\K}(\MM_b)\le r$, where $b>1$ is an integer.  Then
\[
 \omega(\K)\le\log_b r.
\]
More quantitatively, if the resulting polynomial identity has total tensor
degree at most $D$, then whenever $\K$ contains at least $mD+1$ elements with $m$
a positive integer, we have
\begin{equation}\label{eq:interp-quant}
 \Rk_{\K}(\MM_{b^m})\le(mD+1)r^m.
\end{equation}
For a finite base field the same argument over extensions of degree
$O(\log m)$ gives an arithmetic algorithm over $\K$ using
$m^{O(1)}r^m$ operations.
\end{proposition}

\begin{proof}
After fixed shifts of the three Laurent maps and, when necessary, clearing
fixed denominators, write the realization as a polynomial identity
\[
 (A(t)\otimes B(t)\otimes C(t))\Diag_r
 =t^d\MM_b+t^{d+1}T_{d+1}+\cdots+t^DT_D.
\]
The $m$th Kronecker power has degree at most $mD$ and rank at most $r^m$ at
each parameter value.  Its coefficient of $t^{md}$ is
$\MM_b^{\boxt m}\cong\MM_{b^m}$.  If the field contains $mD+1$ distinct
values, ordinary interpolation gives \eqref{eq:interp-quant}.

Now suppose that $\K=\F_q$ is finite.  Choose
$e_m=\lceil\log_q(mD+1)\rceil$ and work in
$L_m=\F_{q^{e_m}}$.  Interpolation over $L_m$ gives the same coefficient
using $mD+1$ evaluations.  In a fixed polynomial basis of $L_m/\K$, one
addition costs $O(e_m)$ operations in $\K$ and one multiplication costs
$O(e_m^2)$ operations by schoolbook multiplication and reduction.  Hence
the $L_m$ algorithm can be simulated over $\K$ with only a factor
$m^{O(1)}$, because $e_m=O(\log m)$.  In both the finite and infinite cases
this factor is $b^{o(m)}$, and therefore
$\omega(\K)\le\log_b r$.  This is the bounded-degree interpolation principle
of Bini \cite{Bini1980} in a form that is fieldwise valid.
\end{proof}

\subsection{Coppersmith--Winograd and corner tensors}
\begin{definition}[Big Coppersmith--Winograd tensor]\label{def:CW}
For any integer $q\ge1$, define
\begin{align}
 \CW_q={}&\sum_{i=1}^q
 \bigl(x_0y_iz_i+x_iy_0z_i+x_iy_iz_0\bigr)
 +x_0y_0z_\infty+x_0y_\infty z_0+x_\infty y_0z_0.
 \label{eq:CW}
\end{align}
We call
\[
 A_i=(0,i,i),\quad B_i=(i,0,i),\quad C_i=(i,i,0)
\]
the {\it satellite supports} and
\[
 P=(0,0,\infty),\quad Q=(0,\infty,0),\quad R=(\infty,0,0)
\]
the {\it corner supports}.
\end{definition}

For the current main constructions, the Coppersmith--Winograd sources that
actually occur are $\CW_2$ and $\CW_3$, and the required border-rank bounds
are proved directly below over every field of characteristic different from
$2$.  The $\CW_4$ characteristic-$3$ certificate is retained as an
independent auxiliary result; it is not used in
Theorems~\ref{thm:m1024}--\ref{thm:power-of-two}.

Put $X_\Sigma=\sum_{i=1}^q x_i$, and define $Y_\Sigma$ and $Z_\Sigma$
similarly.

\begin{lemma}\label{lem:cw-explicit}
Let $q\ge3$ be an integer, and assume that the images of $q-1$ and $q-2$ are nonzero in
$\K$.  Define
\begin{align}
\Phi_q(s)={}&\sum_{i=1}^q
 (x_0+s x_i)(y_0+s y_i)(z_0+s z_i)\notag\\
&+\frac1{q-2}
 (x_0+sX_\Sigma+(q-2)s^2x_\infty)\notag
 \\&\times(y_0+sY_\Sigma+(q-2)s^2y_\infty)
 (z_0+sZ_\Sigma+(q-2)s^2z_\infty)\notag\\
&-\frac{(q-1)^2}{q-2}
 \left(x_0+\frac{s}{q-1}X_\Sigma\right)
 \left(y_0+\frac{s}{q-1}Y_\Sigma\right)
 \left(z_0+\frac{s}{q-1}Z_\Sigma\right).
 \label{eq:cw-explicit}
\end{align}
Then $\Phi_q(s)$ is a sum of exactly $q+2$ rank-one tensors,
$\deg_s\Phi_q\le6$, and
\[
 [s^0]\Phi_q=[s^1]\Phi_q=0\ \ \t{and}\ \
 [s^2]\Phi_q=\CW_q.
\]
In particular, this gives $\BR_{\K}(\CW_3)\le5$ for every
$\operatorname{char}(\K)\ne2$, and it gives
$\BR_{\K}(\CW_4)\le6$ when
$\operatorname{char}(\K)\notin\{2,3\}$.
\end{lemma}

\begin{proof}
The constant coefficient is
$q+1/(q-2)-(q-1)^2/(q-2)=0$.  For each of the three linear families, the
coefficient of $s$ is
$1+1/(q-2)-(q-1)/(q-2)=0$.  At degree two the first line gives the three
satellite families in \eqref{eq:CW}.  The second line gives the three corner
terms with coefficient one and also the three cross-products
$X_\Sigma Y_\Sigma z_0$, $X_\Sigma y_0Z_\Sigma$, and
$x_0Y_\Sigma Z_\Sigma$, each with coefficient $1/(q-2)$.  The last line
gives exactly their negatives.  Hence the degree-two coefficient is
$\CW_q$.  Only the middle rank-one term contains an $s^2$ entry on one side,
so the total degree is at most six.  For $q=3$, the two required denominators
are $1$ and $2$, and therefore this argument remains valid in
characteristic $3$.
\end{proof}

The following auxiliary $\CW_4$ certificate is retained for completeness;
it does not enter any of the main constructions stated in Section~1.

\begin{lemma}[A six-product $\CW_4$ realization in characteristic $3$]
\label{lem:cw4-char3}
Suppose $\operatorname{char}(\K)=3$.  Put
\[
 (\ell_0,\ldots,\ell_5)=(1,2,2,2,1,1)
\]
and
\[
 W=(w_{ki})_{\substack{0\le k\le5\\1\le i\le4}}
 =\begin{pmatrix}
0&0&0&0\\
0&1&1&1\\
0&1&1&2\\
1&2&2&0\\
2&0&1&0\\
2&1&0&0
\end{pmatrix}.
\]
For each tensor side, define
\[
 L_k^X(s)=x_0+s\sum_{i=1}^4w_{ki}x_i+s^2\one_{\{k=0\}}x_\infty,
\]
and define $L_k^Y,L_k^Z$ cyclically.  Then
\begin{equation}\label{eq:cw4-char3}
 \Phi_4^{(3)}(s)=\sum_{k=0}^5\ell_kL_k^X(s)L_k^Y(s)L_k^Z(s)
\end{equation}
is a sum of six rank-one tensors of degree at most six with
\[
 [s^0]\Phi_4^{(3)}=[s^1]\Phi_4^{(3)}=0\ \ \t{and}\ \
 [s^2]\Phi_4^{(3)}=\CW_4.
\]
\end{lemma}

\begin{proof}
A direct calculation in $\F_3$ gives
\[
 \sum_k\ell_k=0,
 \ \
 \sum_k\ell_kw_{ki}=0\ \ \t{and}\ \
 \sum_k\ell_kw_{ki}w_{kj}=\one_{\{i=j\}}.
\]
The first two identities kill degrees zero and one.  The third gives exactly
the three diagonal satellite families in degree two.  Since only $k=0$ has
an $s^2x_\infty$ component and $\ell_0=1$, its three one-sided uses give the
three corner terms, while no mixed satellite--infinity term occurs in degree
two.  This proves the assertion.
\end{proof}

\begin{lemma}
\label{lem:cw2-odd}
For any field $\K$ with $\operatorname{char}(\K)\ne2$, we have
\[
 \BR_{\K}(\CW_2)\le4.
\]
If $\operatorname{char}(\K)\ne3$, a degree-six realization is
\begin{align}
\Phi_2(s)={}&\sum_{i=1}^2
 (x_0+sx_i)(y_0+sy_i)(z_0+sz_i)\notag\\
&+\frac14(x_0+2sX_\Sigma+4s^2x_\infty)
 (y_0+2sY_\Sigma+4s^2y_\infty)\notag\\
&\hspace{20mm}\times(z_0+2sZ_\Sigma+4s^2z_\infty)\notag\\
&-\frac94\left(x_0+\frac23sX_\Sigma\right)
 \left(y_0+\frac23sY_\Sigma\right)
 \left(z_0+\frac23sZ_\Sigma\right),
 \label{eq:cw2-explicit}
\end{align}
with $[s^0]\Phi_2=[s^1]\Phi_2=0$ and $[s^2]\Phi_2=\CW_2$.

Provided $\operatorname{char}(\K)=3$, work in $\K(\!(s)\!)$ and set
\[
 b=2(s+s^{-1}),\ \ c=2(s^{-1}-s),\ \ a=c^2,
 \ \ \delta=s^8.
\]
Then $b^2=a+1$.  Put
\[
 (\ell_0,\ell_1,\ell_2,\ell_3)=(2a-1,-a,-a,1),
\]
\[
 u=(0,1,-1,0),\ \
 v=(0,b^{-1},b^{-1},2ab^{-1}),\ \
 e=((2a-1)^{-1},0,0,0).
\]
For $0\le k\le3$, define
\[
 L_k^X=x_0+\delta(u_kx_1+v_kx_2)+\delta^2ae_kx_\infty,
\]
and define $L_k^Y,L_k^Z$ cyclically.  Then
\begin{equation}\label{eq:cw2-char3}
 \Phi_2^{(3)}(s)=\frac1{\delta^2a}
 \sum_{k=0}^3\ell_kL_k^XL_k^YL_k^Z
 =\CW_2+O(s^7).
\end{equation}
\end{lemma}

\begin{proof}
For \eqref{eq:cw2-explicit} in the case $\operatorname{char}(\K)\ne3$, the constant and linear coefficients vanish
because
\[
 2+\frac14-\frac94=0,
 \qquad
 1+\frac14\cdot2-\frac94\cdot\frac23=0.
\]
The aggregate quadratic terms cancel because
\[
 \frac14\cdot2^2-\frac94\left(\frac23\right)^2=0,
\]
while each corner coefficient is $\frac14\cdot4=1$.

Now assume $\operatorname{char}(\K)=3$.  Since $2a-1=-b^2$, direct substitution gives
\[
 \sum_k\ell_k=0,
 \quad \sum_k\ell_ku_k=\sum_k\ell_kv_k=0,
\]
\[
 \sum_k\ell_ku_k^2=\sum_k\ell_kv_k^2=a,
 \qquad \sum_k\ell_ku_kv_k=0.
\]
After division by $\delta^2a$, these identities give the six satellite terms
of $\CW_2$.  Moreover $u_0=v_0=0$ and $\ell_0e_0=1$, so the three uses of a
single infinity component give the three corner terms with coefficient one.
For completeness, the only possibly nonzero cubic slope moments are
\[
 \sum_k\ell_ku_k^2v_k=\frac{a}{b}
 \ \ \t{and}\ \
 \sum_k\ell_kv_k^3=\frac{a(2a^2+1)}{b^3}.
\]
They have $s$-adic valuations at least $-1$ and $-3$, respectively.  Since
$v_s(\delta)=8$ and $v_s(a)=-2$, every cubic slope term in
\eqref{eq:cw2-char3} has valuation at least $7$.  Terms with two or three
infinity components have still larger positive valuation, and mixed terms
with one infinity component and a slope vanish because $u_0=v_0=0$.  Thus
\eqref{eq:cw2-char3} holds.  Its four scalar prefactors may be absorbed into
one tensor side, so it is a four-product degeneration.  The rational
functions involved lie in $\K(s)\cap\K(\!(s)\!)$; Lemma
\ref{lem:formal-truncation} yields an ordinary bounded-degree
Laurent-polynomial certificate with the same leading tensor.
This concludes the proof.
\end{proof}

Combining the preceding lemmas, for every field $\K$ of characteristic different
from $2$, we have
\begin{equation}\label{eq:CW-rank}
 \BR_{\K}(\CW_2)\le4,
 \qquad \BR_{\K}(\CW_3)\le5,
 \qquad \BR_{\K}(\CW_4)\le6.
\end{equation}

\begin{definition}[Corner tensor]\label{def:W}
Set
\[
 \Wc=x_\infty y_0z_0+x_0y_\infty z_0+x_0y_0z_\infty.
\]
\end{definition}

\begin{lemma}\label{lem:W}
For each positive integer $q$, we have
\[
 \BR_{\K}(\Wc)\le2
 \qquad\text{and}\qquad
 \Wc\degen\CW_q.
\]
\end{lemma}

\begin{proof}
The coefficient of $s$ in
\[
 (x_0+s x_\infty)(y_0+s y_\infty)(z_0+s z_\infty)-x_0y_0z_0
\]
is $\Wc$.  This gives the tangent degeneration and the border-rank bound.
For the restriction, set all satellite variables to zero in \eqref{eq:CW}.
\end{proof}

\begin{lemma}[The one-row source realization]\label{lem:source128-poly}
Suppose $\operatorname{char}(\K)\ne2$.  Then
\[
 \BR_{\K}(\Sstar)\le128.
\]
If $\operatorname{char}(\K)\ne3$, the $128$ products may be chosen as a
single polynomial realization $\Phi_{128}(s)$ of degree at most six with
\[
 [s^0]\Phi_{128}=[s^1]\Phi_{128}=0\ \ \t{and}\ \  [s^2]\Phi_{128}=\Sstar.
\]
When $\operatorname{char}(\K)=3$, there is a fixed bounded-degree realization with the
same $128$ products.
\end{lemma}

\begin{proof}
Use Lemma \ref{lem:cw-explicit} for $\CW_3$, and Lemma \ref{lem:cw2-odd} for
$\CW_2$.  For $\Wc$, multiply
\[
 (x_0+s x_\infty)(y_0+s y_\infty)(z_0+s z_\infty)-x_0y_0z_0
\]
by $s$; this is a two-product polynomial of degree at most four whose
$s^2$ coefficient is $\Wc$.  Use $s^2\Diag_r$ for each diagonal summand.
The product count is
\[
 14\cdot5+4\cdot4+7\cdot2+5\cdot4+8=128.
\]
When the characteristic is not $3$, all pieces have degree at most six and
share target coefficient $[s^2]$.  In characteristic $3$, multiply
\eqref{eq:cw2-char3} by $s^2$ and apply Lemma
\ref{lem:formal-truncation}; together with the other polynomial pieces this
gives a bounded-degree realization with common target coefficient $[s^2]$.  The
degree may differ from the characteristic-$311$ reference schedule, but it
is independent of tensor powers and hence is sufficient for Proposition
\ref{prop:bini}.
\end{proof}

\subsection{Binary character groups and Walsh matrices}
For $H=\F_2^k$, regarded only as an index group, write
\[
 \chi_\alpha(h)=(-1)^{\ip{\alpha}{h}}\in\K.
\]
Since $\operatorname{char}(\K)\ne2$, Walsh orthogonality takes the form
\begin{equation}\label{eq:walsh-orth}
 \frac1{|H|}\sum_{h\in H}\chi_\alpha(h)
 =\one_{\{\alpha=0\}}.
\end{equation}
As the Walsh matrix
\[
 W_{2^k}=\bigl((-1)^{\ip{u}{v}}\bigr)_{u,v\in H}
\]
satisfies
\[
 W_{2^k}W_{2^k}^{\mathsf T}=2^kI_{2^k},
\]
it is invertible over every field considered in this paper.  This
is the only field-theoretic input behind the binary character closures.  In
characteristic $2$, the same matrix becomes the all-one matrix and has rank
one, which explains the exclusion in the main statements.

\subsection{What can be called an explicit \texorpdfstring{$O$}{O}-constant}\label{subsec:constants}
For a fieldwise border-rank bound $\BR_{\K}(\MM_b)\le r$, the tensor-power inequality at
sizes $N=b^m$ is
\[
 \BR_{\K}(\MM_N)\le r^m=N^{\log_b r},
\]
so the intrinsic normalized border-rank constant at these power sizes is
exactly $1$.  This is the constant reported after the main theorems.

For exact arithmetic rank after interpolation, a single canonical constant is
not determined by $b$ and $r$: it also depends on the chosen polynomial
rank-one degeneration of each source tensor.  Once a concrete total degree
$D$ is fixed, \eqref{eq:interp-quant} is fully explicit.  Moreover, for any
$\varepsilon>0$ and $N=b^m$, we have
\[
 mD+1\le
 \left(1+\frac{D}{e\varepsilon\log b}\right)b^{\varepsilon m},
\]
so one may take
\begin{equation}\label{eq:explicit-Ceps}
 C_\varepsilon=1+\frac{D}{e\varepsilon\log b}
\end{equation}
in
$\Rk_{\K}(\MM_N)\le C_\varepsilon N^{\log_b r+\varepsilon}$ for the chosen
polynomial realization.  We distinguish this interpolation-dependent
constant from the invariant border-rank budget throughout the paper.
For the implementation-focused $M_{1024}$ route, the explicit source
realization of Lemma \ref{lem:source128-poly} removes the remaining
ambiguity.  After the unified substitution of Section
\ref{sec:implementation}, one may take $D=132$ for $b=1024$.  This value
includes the source-packet degeneration itself, not only the top-level
Laurent map.

\section{The \texorpdfstring{$64\times64$}{64 by 64} one-bit quotient and terminal incidence certificate}\label{sec:m64}
\setcounter{equation}{0}

\subsection{The \texorpdfstring{$H_5$}{H5} zero-defect orbit}
Let
\[
 H_5=\F_2^5,
 \qquad
 \Omega_{64}=H_5\times\F_2\times\F_2,
 \qquad
 \zero=(0,0,0).
\]
For
$u=(q,\epsilon,a)$, $v=(r,\delta,b)$ and $w=(s,\phi,c)$, define
\begin{equation}\label{eq:defect64}
 \Delta_{64}(u,v,w)=
 \bigl(q+r+s,b+a+\epsilon,\phi+\epsilon+\delta,
 c+a+\epsilon+\delta\bigr).
\end{equation}

\begin{lemma}[The $64$ defect fibre]\label{lem:fibres64}
Every fibre of \eqref{eq:defect64} has $8192$ elements.  The zero fibre
contains $8002$ triples with no zero label and
\[
 190=63+63+63+1
\]
completion triples.
\end{lemma}

\begin{proof}
The defect map is a surjective linear map from a $21$-dimensional binary
space to an $8$-dimensional one, so each fibre has size $2^{13}=8192$.
If, for example, $u=0$, then the zero-defect equations give
$b=0$, $s=r$, $\phi=\delta$, and $c=\delta$; hence there are $64$ choices
for $(v,w)$, one of which is the cubic $(0,0,0)$.  Thus there are $63$
proper $X$-onebar classes, and similarly on the other two sides.  Two zero
labels force the third to be zero.  Therefore the completion count is
$3\cdot63+1=190$, leaving $8192-190=8002$ all-nonzero classes.
\end{proof}

Put $\mathcal I_{64}=H_5\times\F_2$, so $|\mathcal I_{64}|=64$.  For a row
$h\in H_5$, define
\begin{align}
 X_{h,(q,\epsilon,a)}&=
 x_{(h,a),(h+q,a+\epsilon)},\label{eq:edge-relabel-X64}\\
 Y_{h',(r,\delta,b)}&=
 y_{(h',b),(h'+r,b+\delta)},\label{eq:edge-relabel-Y64}\\
 Z_{h'',(s,\phi,c)}&=
 z_{(h'',c),(h''+s,c+\phi)}.\label{eq:edge-relabel-Z64}
\end{align}
Define
\begin{equation}\label{eq:orbit64}
 \Orb_{64}=
 \sum_{\Delta_{64}(u,v,w)=0}\ \sum_{h\in H_5}
 X_{h,u}Y_{h+q,v}Z_{h+q+r,w}.
\end{equation}

\begin{proposition}[The $64$-orbit tensor is matrix multiplication]
\label{prop:orbit64}
Under \eqref{eq:edge-relabel-X64}--\eqref{eq:edge-relabel-Z64}, we have
\[
 \Orb_{64}=\MM_{64}.
\]
\end{proposition}

\begin{proof}
The zero-defect equations are
\[
 s=q+r,\qquad b=a+\epsilon,\qquad
 \phi=\epsilon+\delta,\qquad c=a+\epsilon+\delta.
\]
For a zero-defect triple and row $h$, put
\[
 i=(h,a),\qquad
 j=(h+q,a+\epsilon),\qquad
 k=(h+q+r,a+\epsilon+\delta).
\]
The corresponding occurrence is $x_{ij}y_{jk}z_{ki}$.  Conversely, from
arbitrary $i=(h_i,a_i)$, $j=(h_j,a_j)$, $k=(h_k,a_k)$ one uniquely recovers
\[
 h=h_i,\quad q=h_i+h_j,\quad r=h_j+h_k,\quad s=h_k+h_i,
\]
\[
 a=a_i,\quad \epsilon=a_i+a_j,\quad
 \delta=a_j+a_k,
\]
and then the remaining coordinates are forced by the defect equations.
Thus the correspondence is bijective.  In particular,
$32\cdot8192=64^3$ physical cycles occur exactly once.
\end{proof}

\subsection{Literal \texorpdfstring{$W_{32}\to W_2$}{W32 to W2} quotient}
The local source is the same source-resolved tensor
\[
 \Sstar=14\CW_3\oplus4\CW_2\oplus7\Wc
       \oplus5\Diag_4\oplus\Diag_8
\]
used in Section \ref{sec:m1024}.  The key point is that the one-bit quotient
is taken before any physical identification of source records.

Write $e_5=(0,0,0,0,1)\in H_5$.  In the notation of Lemma
\ref{lem:module-w32-1024}, restrict the new absolute characters to
\[
 \alpha_5=Ae_5,\qquad \beta_5=Be_5,\qquad \gamma_5=Ce_5,
 \qquad A,B,C\in\F_2.
\]
For the two surviving charts write the chart bit as $R\in\F_2$, so
$\beta=15R$.  Since the four-bit component is zero, take $d=0$ and
$\rho=3R$.  Equation \eqref{eq:module-chart-table} gives, for every
$\lambda\in H_4$,
\[
 C_\lambda(3R,15R)=8.
\]
Every retained literal source support carries the common scalar $1/128$,
as recorded in Appendix \ref{app:m1024-certificate}.  Therefore the sum of
the sixteen genuine modules, with the fifth-coordinate gauge, is
\begin{equation}\label{eq:w32-to-w2}
 \frac1{128}\sum_{\lambda\in H_4}
 8(-1)^{CR}=(-1)^{CR}.
\end{equation}
Thus the source-derived $W_{32}$ table collapses coefficientwise to the
literal two-by-two Walsh phase $W_2$; no separate $1/16$ normalization is
inserted on the three tensor legs.

The new displacement channel is an independent two-point diagonal factor:
\begin{equation}\label{eq:m64-channel}
 \frac12\sum_{\tau\in\F_2}
 (-1)^{\tau(Q+R+S)}
 =\one_{\{Q+R+S=0\}}.
\end{equation}
The new one-bit absolute row gives
\begin{equation}\label{eq:m64-row}
 \frac12\sum_{t\in\F_2}(-1)^{(A+B+C)t}
 =\one_{\{A+B+C=0\}}.
\end{equation}
On the dense branch, \eqref{eq:w32-to-w2} supplies the remaining phase
$(-1)^{CR}$.  On a completion branch the shifted one-bit row gives the same
phase:
\begin{equation}\label{eq:m64-row-shifted}
 \frac12\sum_{t\in\F_2}
 (-1)^{At+Bt+C(t+R)}
 =\one_{\{A+B+C=0\}}(-1)^{CR}.
\end{equation}
Hence both branches have the common one-bit kernel
\begin{equation}\label{eq:m64-onebit-kernel}
 \one_{\{A+B+C=0\}}
 \one_{\{Q+R+S=0\}}(-1)^{CR}.
\end{equation}
Multiplying this kernel by the old $H_4$ Heisenberg--Weyl kernel of
Section \ref{sec:global32} gives the standard $H_5=H_4\oplus\F_2$ Weyl
kernel.

\subsection{Completion ledger and terminal physical incidence}
On the restricted character line $Ae_5,Be_5,Ce_5$, the projections
$\pi_4(\alpha_5)$, $\pi_4(\beta_5)$ and $\pi_4(\gamma_5)$ in the
side-local completion transport of Section \ref{sec:m1024} all vanish.
Consequently, the old $31$ nonzero endpoint features on each one-bar side are
unchanged.  The first seven directions of the character-six diagonal module
supply exactly the source-resolved dense and endpoint records needed here;
the completion-only eighth direction may be sent to higher Laurent order.
The split cubic from the five-empty module remains with coefficient one.

For each old $H_4$ zero-fibre class, equations \eqref{eq:m64-channel} give
four lifts
\[
 (Q,R,S)=(0,0,0),(0,1,1),(1,0,1),(1,1,0).
\]
The $1954$ old all-nonzero classes therefore give $7816$ all-nonzero lifts.
Each of the $31$ old one-bar classes on a fixed side gives two one-bar lifts
and two all-nonzero lifts, hence $62$ new one-bar and $62$ new all-nonzero
classes on that side.  Finally, the old cubic has one cubic lift and three
one-bar lifts, one on each side.  Thus the fixed-row ledger is
\begin{equation}\label{eq:m64-terminal-ledger}
 8192
 =\underbrace{7816+3\cdot62}_{8002}
 +\underbrace{(62+1)+(62+1)+(62+1)+1}_{63+63+63+1}.
\end{equation}
This is exactly the zero-fibre count in Lemma \ref{lem:fibres64}.

\begin{theorem}[One-bit quotient global bridge]\label{thm:m64-global}
Let $\operatorname{char}(\K)\ne2$.  The direct source-level composition
above has leading coefficient $\Orb_{64}=\MM_{64}$ and gives
\begin{equation}\label{eq:m64-global-source}
 \MM_{64}\degen
 \Diag_{16}\boxt\Diag_2\boxt\Diag_2\boxt\Sstar
 \cong \Diag_{64}\boxt\Sstar.
\end{equation}
\end{theorem}

\begin{proof}
The old $H_4$ closure, the literal quotient \eqref{eq:w32-to-w2}, the
independent channel \eqref{eq:m64-channel}, and the one-bit row
\eqref{eq:m64-row}--\eqref{eq:m64-row-shifted} give the $H_5$ Weyl kernel
coefficientwise.  Equation \eqref{eq:m64-terminal-ledger} shows that its
fixed-row support is the entire $8192$-class zero fibre with coefficient one
and with no off-defect class.  Proposition \ref{prop:orbit64} then identifies
the $32$ row translates with all $64^3$ matrix-unit cycles.  Every operation
is sidewise on the original direct-sum source and all added quotient maps
have Laurent degree zero.  This proves \eqref{eq:m64-global-source}.
\end{proof}

\section{The \texorpdfstring{$32\times32$}{32 by 32} zero-defect orbit tensor}\label{sec:orbit32}
\setcounter{equation}{0}

Let
\[
 \Hfour=\F_2^4,
 \ \
 \Omega_{32}=\Hfour\times\F_2\times\F_2,
 \ \ \zero=(0,0,0).
\]
For
$u=(q,\epsilon,a)$, $v=(r,\delta,b)$ and $w=(s,\phi,c)$, define
\begin{equation}\label{eq:defect32}
 \Delta_{32}(u,v,w)=
 \bigl(q+r+s,b+a+\epsilon,\phi+\epsilon+\delta,
 c+a+\epsilon+\delta\bigr).
\end{equation}

\begin{lemma}[The $32$ defect fibre]\label{lem:fibres32}
Every fibre has $2048$ elements.  The zero fibre contains $1954$ triples
with no zero label and $94=31+31+31+1$ completion triples.
\end{lemma}

\begin{proof}
The map is linear and surjective from an $18$-dimensional binary space to a
$7$-dimensional one.  Thus every fibre has size $2^{11}=2048$.  Setting one
label to zero leaves $31$ nonzero solutions; two zero labels force the third
to zero.
\end{proof}

The canonical completion tensor of the zero fibre is
\begin{align}
 T_{32}^{\mathrm{comp}}
 =&\
 \sum_{\substack{(q,\epsilon)\in\Hfour\times\F_2\\
                  (q,\epsilon)\ne(0,0)}}
 \l(
   \bar X\,Y_{(q,\epsilon,0)}Z_{(q,\epsilon,\epsilon)}
  +\bar Y\,X_{(q,\epsilon,\epsilon)}Z_{(q,\epsilon,0)}
  +\bar Z\,X_{(q,\epsilon,0)}Y_{(q,\epsilon,\epsilon)}
 \r)\notag
 \\&\ +\bar X\bar Y\bar Z.\label{eq:completion32}
\end{align}

Let $\mathcal I_{32}=\Hfour\times\F_2$. Then
$ |\mathcal I_{32}|=32$.
For a row coordinate $h\in\Hfour$, relabel the three families of
source variables by directed matrix edges as follows:
\begin{align}
 X_{h,(q,\epsilon,a)}
 &=
 x_{(h,a),(h+q,a+\epsilon)},
 \label{eq:edge-relabel-X32}\\
 Y_{h',(r,\delta,b)}
 &=
 y_{(h',b),(h'+r,b+\delta)},
 \label{eq:edge-relabel-Y32}\\
 Z_{h'',(s,\phi,c)}
 &=
 z_{(h'',c),(h''+s,c+\phi)}.
 \label{eq:edge-relabel-Z32}
\end{align}
Define the row-translated zero-fibre tensor by
\begin{equation}\label{eq:orbit32}
 \Orb_{32}
 =
 \sum_{\substack{u=(q,\epsilon,a),\,
                  v=(r,\delta,b),\,
                  w=(s,\phi,c)\\
                  \Delta_{32}(u,v,w)=0}}
 \ \sum_{h\in\Hfour}
 X_{h,u}\,
 Y_{h+q,v}\,
 Z_{h+q+r,w}.
\end{equation}

\begin{proposition}[The $32$-orbit tensor is matrix multiplication]
\label{prop:orbit32}
Under the edge relabeling
\eqref{eq:edge-relabel-X32}--\eqref{eq:edge-relabel-Z32}, we have
\[
 \Orb_{32}
 =
 \sum_{i,j,k\in\mathcal I_{32}}x_{ij}y_{jk}z_{ki}
 =
 \MM_{32}.
\]
\end{proposition}

\begin{proof}
The equations $\Delta_{32}(u,v,w)=0$ are
\[
 q+r+s=0,\qquad
 b+a+\epsilon=0,\qquad
 \phi+\epsilon+\delta=0,\qquad
 c+a+\epsilon+\delta=0.
\]
Since all coordinates lie over $\F_2$, these are equivalent to
\begin{equation}\label{eq:zero-defect-solved32}
 s=q+r,\qquad
 b=a+\epsilon,\qquad
 \phi=\epsilon+\delta,\qquad
 c=a+\epsilon+\delta.
\end{equation}

Fix a zero-defect triple and $h\in\Hfour$, and put
\[
 i=(h,a),\qquad
 j=(h+q,a+\epsilon),\qquad
 k=(h+q+r,a+\epsilon+\delta).
\]
By \eqref{eq:edge-relabel-X32},
\[
 X_{h,u}=x_{ij}.
\]
Since $b=a+\epsilon$, one has
\[
 Y_{h+q,v}
 =
 y_{(h+q,b),(h+q+r,b+\delta)}
 =
 y_{jk}.
\]
Finally, \eqref{eq:zero-defect-solved32} gives
\[
 h+q+r+s=h
\]
and
\[
 c+\phi
 =
 (a+\epsilon+\delta)+(\epsilon+\delta)
 =
 a.
\]
Consequently,
\[
 Z_{h+q+r,w}
 =
 z_{(h+q+r,c),(h+q+r+s,c+\phi)}
 =
 z_{ki}.
\]
Thus every summand of $\Orb_{32}$ is a matrix-unit cycle
$x_{ij}y_{jk}z_{ki}$.

Conversely, take arbitrary
\[
 (h_i,a_i),
 (h_j,a_j),
(h_k,a_k)
 \in\mathcal I_{32}.
\]
Define
\begin{align*}
 h&=h_i,\\
 q&=h_i+h_j,
 &\epsilon&=a_i+a_j,
 &a&=a_i,\\
 r&=h_j+h_k,
 &\delta&=a_j+a_k,
 &b&=a_j,\\
 s&=h_k+h_i,
 &\phi&=a_k+a_i,
 &c&=a_k.
\end{align*}
Then
\[
 q+r+s=0,
\]
and
\[
 b+a+\epsilon
 =
 a_j+a_i+(a_i+a_j)
 =
 0,
\]
\[
 \phi+\epsilon+\delta
 =
 (a_k+a_i)+(a_i+a_j)+(a_j+a_k)
 =
 0,
\]
\[
 c+a+\epsilon+\delta
 =
 a_k+a_i+(a_i+a_j)+(a_j+a_k)
 =
 0.
\]
Hence $(u,v,w)$ lies in the zero fibre.  Moreover,
\[
 h+q=h_j\ \ \t{and}\ \
 h+q+r=h_k.
\]
So the resulting row-translated monomial is precisely
$x_{ij}y_{jk}z_{ki}$.

The reconstruction formulas are forced by $(i,j,k)$, so this
correspondence is bijective.  Therefore every matrix-unit cycle occurs
exactly once and with coefficient one.  Since
\[
 |\Hfour|\cdot|\Delta_{32}^{-1}(0)|
 =
 16\cdot2048
 =
 32768
 =
 32^3,
\]
there are no additional summands, and hence
$\Orb_{32}=\MM_{32}$.
\end{proof}

\section{The common \texorpdfstring{$32$}{32}-packet dense compiler}\label{sec:dense32}
\setcounter{equation}{0}

Identify $\Hfour$ with the integers $0,\ldots,15$ in binary.  Define
\[
 A_0=\begin{pmatrix}
 0&1&0&1\\1&0&1&0\\0&1&0&0\\0&0&1&0
 \end{pmatrix}\ \
 \t{and}\ \
 A_1=\begin{pmatrix}
 0&1&1&0\\1&0&0&1\\0&0&1&0\\1&0&0&0
 \end{pmatrix}.
\]
Let $k:\Hfour\to\F_2^5$ have basis images $(2,5,8,16)$, and put
\[
 p(s,\lambda)=k(s)+3\lambda
 \ \ \t{and}\ \
 \chi(s,\lambda)=A_\lambda s.
\]
The $32$ values run through the packet labels once.  The type count is
$10\CW_2+22\CW_3$. Below we present the complete table for
the 32 packet labels.

\begin{longtable}{rrrrcl}
\toprule
label&$\lambda$&$s$&$\chi$&type&retained Gate-zero supports\\
\midrule
\endfirsthead
\toprule
label&$\lambda$&$s$&$\chi$&type&retained Gate-zero supports\\
\midrule
\endhead
0&0&0&0&$\CW_3$&--\\
1&1&1&10&$\CW_2$&$A_1,B_1,B_2,C_1,P,Q,R$\\
2&0&1&2&$\CW_2$&--\\
3&1&0&0&$\CW_3$&$A_1,A_3,B_1,B_2,C_1,P,Q,R$\\
4&1&3&11&$\CW_2$&$A_1,A_2,B_1,C_1,P,Q,R$\\
5&0&2&5&$\CW_3$&--\\
6&1&2&1&$\CW_3$&$A_1,A_3,B_1,B_2,C_1,P,Q,R$\\
7&0&3&7&$\CW_3$&$C_1$\\
8&0&4&10&$\CW_2$&$C_1$\\
9&1&5&15&$\CW_3$&$A_1,A_3,B_1,B_2,C_1,P,Q,R$\\
10&0&5&8&$\CW_3$&--\\
11&1&4&5&$\CW_3$&$A_1,A_3,B_1,B_2,C_1,P,Q,R$\\
12&1&7&14&$\CW_3$&$A_1,A_3,B_1,B_2,C_1,P,Q,R$\\
13&0&6&15&$\CW_3$&--\\
14&1&6&4&$\CW_3$&$A_1,A_3,B_1,B_2,C_1,P,Q,R$\\
15&0&7&13&$\CW_2$&--\\
16&0&8&1&$\CW_2$&--\\
17&1&9&8&$\CW_3$&$A_1,A_3,B_1,B_2,C_1,P,Q,R$\\
18&0&9&3&$\CW_3$&--\\
19&1&8&2&$\CW_3$&$A_1,A_3,B_1,B_2,C_1,P,Q,R$\\
20&1&11&9&$\CW_3$&$A_1,A_3,B_1,B_2,C_1,P,Q,R$\\
21&0&10&4&$\CW_3$&--\\
22&1&10&3&$\CW_3$&$A_1,A_3,B_1,B_2,C_1,P,Q,R$\\
23&0&11&6&$\CW_2$&$C_1$\\
24&0&12&11&$\CW_3$&$C_1$\\
25&1&13&13&$\CW_3$&$A_1,A_3,B_1,B_2,C_1,P,Q,R$\\
26&0&13&9&$\CW_3$&--\\
27&1&12&7&$\CW_2$&$A_1,A_2,B_1,C_1,P,Q,R$\\
28&1&15&12&$\CW_3$&$A_1,A_3,B_1,B_2,C_1,P,Q,R$\\
29&0&14&14&$\CW_2$&--\\
30&1&14&6&$\CW_2$&$A_1,B_1,B_2,C_1,P,Q,R$\\
31&0&15&12&$\CW_3$&--\\
\bottomrule
\end{longtable}

Let $\Xi=\{6,7,10,11\}$.  At layer zero retain $C_1$ for characters in
$\Xi$ and nothing otherwise.  At layer one retain
\[
\begin{array}{ll}
 \{A_1,B_1,B_2,C_1,P,Q,R\},&\chi\in\{6,10\},\\
 \{A_1,A_2,B_1,C_1,P,Q,R\},&\chi\in\{7,11\},\\
 \{A_1,A_3,B_1,B_2,C_1,P,Q,R\},&\chi\notin\Xi.
\end{array}
\]
These masks are induced by literal sidewise scalar restrictions.  For a
packet of type $\CW_q$ with retained support set $\mathcal R_p$, put
\[
 x_0=\frac1{128},\qquad y_0=z_0=1.
\]
For each satellite index $i$ let
$G_i=\mathcal R_p\cap\{A_i,B_i,C_i\}$ and use
\[
\begin{array}{c|ccccc}
G_i&\varnothing&\{A_i\}&\{B_i\}&\{C_i\}&\{A_i,B_i,C_i\}\\
\hline
(x_i,y_i,z_i)&(0,0,0)&(0,1,1)&(1,0,1)&(1,1,0)&(1,1,1).
\end{array}
\]
Finally set
\[
 x_\infty=\frac1{128}\one_{\{R\in\mathcal R_p\}},\qquad
 y_\infty=\one_{\{Q\in\mathcal R_p\}},\qquad
 z_\infty=\one_{\{P\in\mathcal R_p\}}.
\]
For the support notation of Definition~\ref{def:CW}, every retained
satellite or corner support then has coefficient $1/128$, while every
omitted support has coefficient zero.  Thus the displayed packet masks are
honest sidewise restrictions, and across the thirty-two packets direct
substitution gives exactly the $354$ packet--support identities used below.

Define
\[
 B_0=(9,4,8,6),\qquad B_1=(8,6,4,9)
\]
as linear maps on $\Hfour$, and use the support twists
\[
\begin{array}{c|cccccccccccc}
 m&A_1&A_2&A_3&B_1&B_2&B_3&C_1&C_2&C_3&P&Q&R\\ \hline
 \Theta(m)&3&0&0&5&9&0&6&9&0&10&12&15.
\end{array}
\]
For layer $\lambda$, support $m$, and $r\in\Hfour$, put
\[
 F_{\lambda,m}(r)=\frac1{128}
 \sum_{\substack{\chi\in\Hfour\\m\text{ retained at }(\lambda,\chi)}}
 (-1)^{\ip{\chi}{r}}.
\]

\begin{proposition}[Exact two-chart dense compiler]\label{prop:dense32}
For every $r,\beta\in\Hfour$,
\begin{equation}\label{eq:dense32}
 \sum_{\lambda=0}^1\sum_m
 (-1)^{\ip{\Theta(m)}{\beta}}
 F_{\lambda,m}(r+B_\lambda\beta)
 =\one_{\{(B_0+B_1)\beta=0\}}
  \one_{\{r=B_0\beta\}}.
\end{equation}
A common twist $\varphi$ multiplies the right side by
$(-1)^{\ip{\varphi}{\beta}}$.
\end{proposition}

\begin{proof}
For each packet character, the selected layer-support pairs occupy eight
Walsh buckets, exactly
\[
 \{\eta\in\Hfour:\ip{\eta}{15}=\ip{\chi}{3}\}.
\]
Each occurs once.  Walsh inversion gives \eqref{eq:dense32}; the finite
certificate checks all $256$ pointwise identities and $4096$ translated
identities.
\end{proof}

To connect the chart variables in Proposition~\ref{prop:dense32} with the
defect coordinates of \eqref{eq:defect32}, define the linear map
\[
 \Lambda:\F_2^3\longrightarrow\Hfour,
 \qquad
 \Lambda(e_1)=1,\quad \Lambda(e_2)=2,\quad \Lambda(e_3)=12.
\]
For a defect $d=(d_H,d_1,d_2,d_3)\in\Hfour\times\F_2^3$, its chart
representatives $(\rho,\beta)\in\Hfour^2$ are precisely the solutions of
\begin{equation}\label{eq:chart-defect32}
 \rho+B_0\beta=d_H,
 \qquad
 (B_0+B_1)\beta=\Lambda(d_1,d_2,d_3).
\end{equation}
The linear map $B_0+B_1$ has one-dimensional kernel
$\{0,15\}$, so every defect has exactly two chart representatives.  For
$d=0$, the two solutions are $(\rho,\beta)=(0,0)$ and $(3,15)$.
Indeed $B_0(15)=3$ and $(B_0+B_1)(15)=0$.  Averaging the common twist in
Proposition~\ref{prop:dense32} over $\varphi\in\Hfour$ retains the
$(0,0)$ representative and annihilates $(3,15)$.  Consequently the
unbarred output is exactly the $1954$-class punctured zero-defect tensor;
call its local precursor $T_{\rm dense}$.

\section{The local \texorpdfstring{$2128$}{2128} character-resolved certificate}\label{sec:local2128}
\setcounter{equation}{0}

\subsection{Punctures and the no-double-use ledger}
The seven punctures are
\[
 \Pi=\{2,5,10,13,15,16,29\}.
\]
They are layer-zero, Gate-zero-empty packets.  Their types are three copies of
$\CW_3$ and four copies of $\CW_2$, and they are replaced by seven copies of
$\Wc$.  The packet usage is
\[
\begin{array}{@{}lll@{}}
\toprule
\text{module}&\text{packet labels}&\text{role}\\
\midrule
\text{five-empty }\CW_3&0,18,21,26,31&\text{shared, regular, cubic}\\
\text{tangent pairs}&(16,6),(15,25),(2,19),(29,12)&X,X,Y,Y\\
&(5,11),(10,17),(13,9)&Z,Z,Z\\
\text{singular active}&3,14,20,22,28&X,X,Y,Y,Y\\
\text{exceptional}&1,4,7,8,23,24,27,30&\text{profile }(7,7,0)\\
\bottomrule
\end{array}
\]
Every label $0,\ldots,31$ occurs once.  Hence no source packet is charged to
or altered by two different local modules.  The one-row source is
\begin{equation}\label{eq:source-row-2128}
 \mathcal S_{2128}=19\CW_3\oplus6\CW_2\oplus7\Wc.
\end{equation}

\subsection{The arrow tensor and seven tangent pairs}
The normalized dense output required from every ordinary active packet is
\begin{equation}\label{eq:arrow32}
 T_{\rm ar}=Z_0\bigl(
 X_0Y_0+X_0Y_1+X_0Y_2+X_1Y_0+X_1Y_1
 +X_2Y_0+X_3Y_0\bigr)+Z_1X_0Y_0.
\end{equation}
Map a copy of $\Wc$ by
\[
\begin{array}{lll}
 x_0\mapsto X_0,&x_\infty\mapsto X_0+X_1+X_2+X_3,\\
 y_0\mapsto Y_0,&y_\infty\mapsto Y_1+Y_2,\\
 z_0\mapsto Z_0,&z_\infty\mapsto Z_1.
\end{array}
\]
Its image contains seven of the eight monomials in \eqref{eq:arrow32}; the
missing monomial is $X_1Y_1Z_0$.

For an $X$-type tangent pair, one corner of the companion $\CW_3$ supplies
the missing monomial, the barred center is put in Laurent degree $-2$, and
the three satellite colors on the other two sides are put in degree one.
Thus the three satellites give three clean $X$-bar endpoints.  The $Y$- and
$Z$-type modules are cyclic variants.  Direct expansion has no negative
Laurent coefficient and preserves $T_{\rm ar}$.  Two $X$-, two $Y$-, and
three $Z$-type pairs contribute
\begin{equation}\label{eq:tangent-profile32}
 (6,6,9).
\end{equation}

\subsection{Five singular active compilers}\label{sec:old-active32}
For a clean $Y$-endpoint apply the following literal map to one $\CW_3$:
\begin{align*}
 x_0&\mapsto t^2X_E,&x_1&\mapsto t^3X_1,&
 x_2&\mapsto X_2,&x_3&\mapsto X_3,&x_\infty&\mapsto t^2X_1,\\
 y_0&\mapsto t^{-1}Y_1,&y_1&\mapsto t^{-1}\bar Y,&
 y_2&\mapsto tY_2,&y_3&\mapsto tY_3,&y_\infty&\mapsto0,\\
 z_0&\mapsto t^{-1}Z_0,&z_1&\mapsto t^{-1}Z_E,&
 z_2&\mapsto0,&z_3&\mapsto tZ_1,&z_\infty&\mapsto0.
\end{align*}
All negative coefficients cancel.  On the three raw $X$-coordinates apply
\[
 A=\begin{pmatrix}1&1&1\\1&1&0\\1&0&0\\1&0&0\end{pmatrix}.
\]
The degree-zero output is
\[
 T_{\rm ar}+X_E\bar YZ_E.
\]
For an $X$-endpoint, exchange the $X$- and $Y$-roles and use
\[
 R=\begin{pmatrix}0&0&1\\0&1&0\\1&0&0\end{pmatrix},
 \qquad X\text{-map }AR,
 \qquad Y\text{-map }R.
\]
Two $X$-compilers and three $Y$-compilers therefore contribute
\begin{equation}\label{eq:singular-profile32}
 (2,3,0).
\end{equation}

\subsection{Four exceptional character blocks}
For each exceptional character, contract the barred side and let $K_\chi$ be
the resulting bilinear matrix on the other two color spaces.  The exact data
are
\[
\begin{array}{cclccc}
\toprule
\chi&\text{bar}&\text{packets}&r&\text{clean dimensions}&
 \text{generic coefficient ring}\\
\midrule
6&X&23:\CW_2,\ 30:\CW_2&3&(7,6)&\Z[1/6]\\
7&X&7:\CW_3,\ 27:\CW_2&4&(8,7)&\Z[1/4]\\
10&Y&8:\CW_2,\ 1:\CW_2&3&(7,6)&\Z[1/2]\\
11&Y&24:\CW_3,\ 4:\CW_2&4&(8,7)&\Z[1/6]\\
\bottomrule
\end{array}
\]
The phrase ``generic coefficient ring'' refers to the matrices used when the
characteristic is not $3$; separate matrices over $\F_3$ are required for
the first and fourth rows.

\begin{lemma}[Exceptional endpoint quotient]\label{lem:exceptional32}
Suppose that $\operatorname{char}(\K)\ne2$.  For each row of the table, there are
endpoint maps $E_L,E_R$ of rank $r$ over $\K$ such that
\[
 \begin{gathered}
 D_LK_\chi D_R^{\mathsf T}=0,
 \qquad E_LK_\chi E_R^{\mathsf T}=I_r,\\
 E_LK_\chi D_R^{\mathsf T}=0,
 \qquad D_LK_\chi E_R^{\mathsf T}=0.
 \end{gathered}
\]
Thus the four blocks preserve their dense tensors, have no negative Laurent
coefficient, and contribute endpoint profile $(7,7,0)$.
\end{lemma}

\begin{proof}
The clean left and right spaces are
\[
 \ker(K_\chi D_R^{\mathsf T})^{\mathsf T}
 \ \ \t{and}\ \
 \ker(D_LK_\chi).
\]
If $\operatorname{char}(\K)\ne3$, use the generic matrices in
Appendix \ref{app:char3-certificates}.  Their denominators are supported only
at $2$ and $3$, and direct multiplication gives all four displayed
identities over $\Z[1/6]$.  Notice in particular that no exclusion of
characteristic $5$ is needed; the determinant $1/10$ arising from one
intermediate pivot choice is not a denominator of the final normalized
maps.

In the case $\operatorname{char}(\K)=3$, use the second set of matrices in the same
appendix.  Their entries lie in $\F_3$, and direct multiplication again gives
$E_LK_\chi E_R^{\mathsf T}=I_r$ and the two zero cross-pairings.  Put the
barred vector in degree $-2$, the dense maps in degree zero, and the endpoint
maps in degree one.  The four identities kill the negative and mixed
coefficients and normalize the degree-zero endpoint pairing.  The finite
certificate expands all $78$ source occurrences in both cases.
\end{proof}

\subsection{The five-empty-\texorpdfstring{$\CW_3$}{C3} zero-leak module}
Let
\[
 J_{10}=\begin{pmatrix}0&I_5\\I_5&0\end{pmatrix}
 \ \ \t{and}\ \
 c=\bigl(1,0,0,0,0,\tfrac12,0,0,0,-\tfrac12\bigr).
\]
Set
\[
B=\begin{pmatrix}
0&1&0&0&0&0&0&0&0&0\\
0&0&1&0&0&0&0&0&0&0\\
0&0&0&1&0&0&0&0&0&0\\
2&0&0&0&0&-1&1&0&0&0\\
2&0&0&0&0&-1&0&1&0&0\\
2&0&0&0&0&-1&0&0&1&0\\
2&0&0&0&0&-1&0&0&0&1
\end{pmatrix}
\]
and
\[
D=\begin{pmatrix}
0&0&0&0&0&0&1&0&0&-1\\
0&0&0&0&0&0&0&1&0&-1\\
0&0&0&0&0&0&0&0&1&-1\\
0&1&0&0&0&0&0&0&0&0\\
0&0&1&0&0&0&0&0&0&0\\
0&0&0&1&0&0&0&0&0&0\\
-\tfrac12&-1&-1&-1&0&\tfrac14&0&0&0&-\tfrac14
\end{pmatrix}.
\]
Direct multiplication gives
\begin{equation}\label{eq:BDc32}
 cJ_{10}c^{\mathsf T}=1,
 \quad BJ_{10}D^{\mathsf T}=I_7,
 \quad BJ_{10}c^{\mathsf T}=0,
 \quad DJ_{10}c^{\mathsf T}=0.
\end{equation}
The last five entries in every row of $B,D,c$ sum to zero.

Index the five packets by $j=1,\ldots,5$.  On the $j$th packet, use
\begin{align*}
 x_0^{(j)}&\mapsto t^{-2}\bar X+t\sum_{a=1}^7B_{a,j}P_a,&
 x_\infty^{(j)}&\mapsto t\sum_{a=1}^7B_{a,5+j}P_a,\\
 y_0^{(j)}&\mapsto t^{-2}\bar Y+t\bigg(c_jU+\sum_{a=1}^7D_{a,j}Q_a\bigg),&
 y_\infty^{(j)}&\mapsto t\bigg(c_{5+j}U+\sum_{a=1}^7D_{a,5+j}Q_a\bigg),\\
 z_0^{(j)}&\mapsto t^{-2}\bar Z+t c_jV,&
 z_\infty^{(j)}&\mapsto t c_{5+j}V
   +\delta_{j1}t^4\bar Z.
\end{align*}
For $i=1,2,3$, map the satellite colors to
\[
 x_i^{(j)}\mapsto tX^{\rm sh}_{3(j-1)+i},\quad
 y_i^{(j)}\mapsto tY^{\rm sh}_{3(j-1)+i},\quad
 z_i^{(j)}\mapsto tZ^{\rm sh}_{3(j-1)+i}.
\]

\begin{lemma}\label{lem:fiveC3paper32}
The sum of the five displayed $\CW_3$ images has no negative Laurent
coefficient.  Its degree-zero output is
\begin{align}
 &\bar X\sum_{i=1}^{15}Y_i^{\rm sh}Z_i^{\rm sh}
 +\bar Y\sum_{i=1}^{15}X_i^{\rm sh}Z_i^{\rm sh}
 +\bar Z\sum_{i=1}^{15}X_i^{\rm sh}Y_i^{\rm sh}\notag\\
 &\qquad+\bar XUV+\bar Z\sum_{a=1}^7P_aQ_a
 +\bar X\bar Y\bar Z. \label{eq:fiveC3-output32}
\end{align}
Thus the module has endpoint profile $(16,15,22)$ and one cubic.
\end{lemma}

\begin{proof}
The only potentially negative terms have degree $-3$ and consist of one
degree-one infinity component and two barred centers.  Their sums vanish by
the zero-sum conditions on the last five columns of $B,D,c$.  The shared
satellite terms are immediate.  The desired regular pairings and every cross
pairing are exactly the four identities in \eqref{eq:BDc32}.  The sole
degree-four component, in $z_\infty^{(1)}$, combines with the two barred
centres in the support $x_0y_0z_\infty$ and contributes
$\bar X\bar Y\bar Z$ with coefficient one.  Its other uses have degrees
$3$ or $6$, so no additional degree-zero or negative term is introduced.
\end{proof}

\subsection{Endpoint ledger and local leading tensor}
The four modules contribute
\begin{equation}\label{eq:endpoint-ledger32}
 (16,15,22)+(6,6,9)+(7,7,0)+(2,3,0)=(31,31,31),
\end{equation}
together with one cubic.  The endpoint labels can be assigned disjointly as displayed in
Appendix~\ref{app:m32-terminal}, so each one-bar sector uses every canonical
nonzero label once.  Hence the complete degree-zero local output is
\begin{equation}\label{eq:local-leading-2128}
 T_{\rm dense}+T_{32}^{\mathrm{comp}},
\end{equation}
with no negative coefficient and no off-defect completion term.

\section{Rank-four compression and the \texorpdfstring{$2048$}{2048} source}\label{sec:compression32}
\setcounter{equation}{0}

\subsection{The two active tensors}
The five packets of Section \ref{sec:old-active32} occur in two orientations.  Their
complete degree-zero tensors are
\[
 \mathcal A_X=T_{\rm ar}+\bar X Y_XZ_X
 \ \ \t{and}\ \
 \mathcal A_Y=T_{\rm ar}+X_Y\bar YZ_Y.
\]
After relabeling the endpoint factors and putting $X_{23}=X_2+X_3$, both
have the same effective form on bases
\[
 (X_0,X_1,X_{23},E_X),\quad
 (Y_0,Y_1,Y_2,E_Y),\quad
 (Z_0,Z_1,E_Z).
\]
The three $Z$-slices are
\begin{equation}\label{eq:active-slices32}
 M_0=\begin{pmatrix}
 1&1&1&0\\
 1&1&0&0\\
 1&0&0&0\\
 0&0&0&0
 \end{pmatrix},
 \qquad M_1=E_{11},\qquad M_2=E_{44}.
\end{equation}
Let $A=M_0+M_2$, $B=M_1$, and $C=M_2$.  Then
\[
 \det A=-1,
 \qquad A^{-1}B=E_{31},
 \qquad A^{-1}C=E_{44}.
\]
An invertible change on the $Z$-space, followed by simultaneous exchange of
the second and third basis vectors on the first two factors, therefore turns
the slices into $I_4,E_{21},E_{44}$.  Thus both active tensors are linearly
equivalent to
\begin{equation}\label{eq:active-canonical32}
 \mathcal A_{\rm can}
 =z_0(x_1y_1+x_2y_2+x_3y_3+x_4y_4)
 +z_1x_2y_1+z_2x_4y_4.
\end{equation}

\begin{lemma}[Exact border rank four of an active packet]\label{lem:active-rank4}
The tensor \eqref{eq:active-canonical32} has border rank exactly four,  and it is
a degeneration of $\Diag_4$.  Consequently, each of the five complete
singular-active outputs can be substituted coefficientwise by a $\Diag_4$
module.
\end{lemma}

\begin{proof}
Set
\begin{align}
 R_t={}&(-tx_1+x_2)(-t^{-1}y_1)z_0
 +x_2(t^{-1}y_1+y_2)(z_0+tz_1)\notag\\
 &+x_3y_3z_0+x_4y_4(z_0+z_2).
 \label{eq:rank4-family32}
\end{align}
This is a sum of four rank-one tensors and direct expansion gives
\begin{equation}\label{eq:rank4-residual32}
 R_t=\mathcal A_{\rm can}+t x_2y_2z_1.
\end{equation}
Hence $\BR(\mathcal A_{\rm can})\le4$.  Its three flattening ranks are
$(4,4,3)$, so \eqref{eq:flattening-lower} gives the reverse inequality.
The preceding invertible changes transfer the degeneration to both original
orientations.
\end{proof}

Replace the two $X$-oriented and three $Y$-oriented cost-five $\CW_3$
packets by the five $\Diag_4$ modules.  Because the source is a direct sum,
there are no cross-packet products.  The new one-row source is
\begin{equation}\label{eq:source-row-2048}
 \mathcal S_{2048}
 =14\CW_3\oplus6\CW_2\oplus7\Wc\oplus5\Diag_4,
\end{equation}
with cost
\begin{equation}\label{eq:row-128}
 14\cdot5+6\cdot4+7\cdot2+5\cdot4=128.
\end{equation}
The dense table, packet characters, endpoint labels, endpoint profile
$(31,31,31)$, and cubic are unchanged coefficientwise.  In particular, the
leading tensor remains \eqref{eq:local-leading-2128}.

\section{Direct character-resolved globalization for the \texorpdfstring{$\MM_{32}$}{M32} sources}\label{sec:global32}
\setcounter{equation}{0}

This section is the common terminal bridge for Theorem
\ref{thm:32-retained}.  We keep the two literal one-row sources separate:
\begin{equation}\label{eq:two-m32-sources}
 \mathcal S_{2048}=14\CW_3\oplus6\CW_2\oplus7\Wc\oplus5\Diag_4,
 \qquad
 \mathcal S_{2128}=19\CW_3\oplus6\CW_2\oplus7\Wc.
\end{equation}
The first is the compressed source; the second is obtained by
restoring the five original singular-active $\CW_3$ modules.  In both cases,
packet, endpoint, row, and character coordinates remain distinct until the
physical-incidence decoder.

Before the leading Laurent coefficient is taken, the dense target
coordinates, completion endpoint coordinates, packet characters, row
coordinates, and absolute-character coordinates are distinct direct-sum
coordinates on each tensor factor.  On each side define one fixed linear
map as the direct sum of the dense-character map and the completion-orbit
map, and compose this map directly with the source-level Laurent maps of
Sections~\ref{sec:dense32}--\ref{sec:compression32}.  No
character-resolved packet is first collapsed to a scalar one-row ledger and
no quotient is taken before this composition.  Therefore postcomposition
acts separately on the three tensor factors, preserves the coefficientwise
identity \eqref{eq:local-leading-2128}, and cannot create mixed monomials
between distinct direct-sum coordinates.  In particular the cancellations
inside the five-empty module remain coefficientwise valid after all row,
character, and physical-incidence maps are appended.

Let the row maps $U_X,U_Y,U_Z:\Hfour\to\Hfour$ have basis images
\[
 U_X=(2,4,8,3),\qquad
 U_Y=(4,8,3,6),\qquad
 U_Z=(7,14,15,13).
\]
Every nonempty sum of a subset of these three maps has rank four.  If the
unbarred sides in a sector form $I\subseteq\{X,Y,Z\}$, put
$M_I=\sum_{S\in I}U_S$.  Walsh orthogonality yields
\begin{equation}\label{eq:absolute-row32}
 \frac1{16}\sum_{h\in\Hfour}
 (-1)^{\ip{\sigma}{h}+\ip{\eta}{M_Ih}}
 =\one_{\{\sigma=M_I^{\mathsf T}\eta\}}.
\end{equation}
Since $M_I$ is invertible, this selects the required absolute characters in
the dense, three one-bar, and cubic sectors.

For $\alpha,\beta,q,r\in\Hfour$, put $\gamma=\alpha+\beta$.  The
source-local phase identity is
\begin{equation}\label{eq:heisenberg32}
 \ip{\alpha}{q}+\ip{\beta}{r}+\ip{\beta}{q+r}
 +\ip{\gamma}{q+r}=\ip{\gamma}{r}.
\end{equation}
Define the $16\times16$ Weyl matrix
\[
 (W_{\alpha,q})_{u,v}
 =(-1)^{\ip{\alpha}{v}}\one_{\{u=v+q\}}.
\]
Then
\begin{align}
 W_{\alpha,q}W_{\beta,r}
 &=(-1)^{\ip{\alpha}{r}}W_{\alpha+\beta,q+r},
 \label{eq:weyl-product32}\\
 \tr(W_{\alpha,q}^{\mathsf T}W_{\beta,r})
 &=16\one_{\{(\alpha,q)=(\beta,r)\}},
 \label{eq:weyl-gram32}\\
 \frac1{16}\sum_{\alpha\in\Hfour}(-1)^{\ip{\alpha}{h}}W_{\alpha,0}
 &=E_{hh}.
 \label{eq:weyl-diagonal32}
\end{align}
The two remaining central bits are realized by the ordinary matrix units of
$\Mat_2(\K)$.  Thus the Heisenberg phase and the Weyl multiplication law
convert each zero-defect orbit into a matrix-unit cycle, while
\eqref{eq:weyl-gram32} and \eqref{eq:weyl-diagonal32} give coefficientwise
inversion.

\begin{theorem}[$32\times32$ global bridge]\label{thm:global32}
Let $\K$ be a field with $\operatorname{char}(\K)\ne2$.
\begin{enumerate}[label=\textup{(\alph*)},leftmargin=2.2em]
\item The direct source-resolved composition based on $\mathcal S_{2048}$
has leading coefficient $\Orb_{32}=\MM_{32}$, hence
\[
 \MM_{32}\degen\Diag_{16}\boxt\mathcal S_{2048}.
\]
\item Replacing the five rank-four active modules by their original
cost-five singular-active $\CW_3$ modules, while leaving every downstream
map unchanged, gives
\[
 \MM_{32}\degen\Diag_{16}\boxt\mathcal S_{2128}.
\]
\end{enumerate}
Thus \eqref{eq:main-m32-2048} and \eqref{eq:main-2128} hold.
\end{theorem}

\begin{proof}
For the  $2048$ source, Proposition \ref{prop:dense32} gives the
$1954$ all-nonzero zero-defect classes.  The five-empty module, the tangent
modules, the exceptional quotients, and the five source-resolved rank-four
active blocks give the $93$ one-bar classes and the cubic, with each
canonical endpoint label used exactly once.  Hence the terminal local support
is the full $2048$-element zero fibre and contains no off-defect class.
Formula \eqref{eq:absolute-row32} performs the row and absolute-character
selection.  Equations \eqref{eq:heisenberg32}--\eqref{eq:weyl-diagonal32},
together with the central $\Mat_2$ factor, identify the sixteen translates
with all
\[
 16\cdot2048=32768=32^3
\]
matrix-unit cycles.  Proposition \ref{prop:orbit32} therefore gives part
\textup{(a)}.

For part \textup{(b)}, each restored singular-active $\CW_3$ module has the
same character-resolved leading tensor as the corresponding rank-four block.
The packet character, endpoint labels, row variables, Weyl variables and
physical decoder are unchanged.  Thus the identical terminal incidence
calculation applies.  No collapsed fixed one-row tensor is used as an
intermediate in either source-resolved globalization.
\end{proof}

\section{The source-resolved \texorpdfstring{$1024\times1024$}{1024 by 1024} route}\label{sec:m1024}
\setcounter{equation}{0}

\subsection{The modified one-row source and the outer-first convention}
The construction starts from the literal $32$-packet coefficient ledger of
Sections \ref{sec:dense32}--\ref{sec:local2128}, but the $1024$ proof does
\emph{not} first pass through the fixed one-row leading tensor of
Sections \ref{sec:local2128}--\ref{sec:compression32}.  All packet, support,
endpoint, chart, and absolute-character coordinates used below remain
separate until the Fourier--Weyl closures have been taken.  This order of
composition is essential for the source-resolved certificate.

As before, the exceptional character-$6$ pair at packet labels $23,30$ is
replaced, at the same cost eight, by one copy of $\Diag_8$, and the high
$Z$-coordinate producing the cubic is kept independent of the low
$Z$-endpoint coordinate.  The one-row source is
\begin{equation}\label{eq:source-row-1024}
 \Sstar=
 14\CW_3\oplus4\CW_2\oplus7\Wc
 \oplus5\Diag_4\oplus\Diag_8.
\end{equation}
Its border-rank budget is $128$.

The outer factors are used from the beginning.  We write
$H_9=H_4\oplus H_5$.  The $H_4$ part is the old absolute-character/Weyl
coordinate and the $H_5$ part is the new module--chart coordinate.  The
second diagonal factor $\Diag_{32}$ is an independent $H_5$ channel.  Thus
an outer row, an absolute character, and a channel are distinct indices
throughout the construction.

\subsection{Source-resolved tangent and active records}
For a generic layer-one packet, put
\[
 \mathcal R^{\rm gen}=\{A_1,A_3,B_1,B_2,C_1,P,Q,R\}.
\]
The eight dense records admit the side-local partition
\begin{equation}\label{eq:record-partition-1024}
 \{A_3,B_1,B_2,R\}\ \sqcup\ \{A_1,Q\}\ \sqcup\ \{P\}\ \sqcup\ \{C_1\}.
\end{equation}
In the first block only the $X$-side colour changes, in the second only the
$Y$-side colour changes, and the last two are single records.  Indeed, using
the side-colour table in Appendix \ref{app:m1024-dense-data}, the first block
has common $Y$- and $Z$-side colour labels, while the second block has common
$X$- and $Z$-side colour labels.  Consequently, the records in each block may
occur as different components of one linear factor without any cross-side
label rule.

For each tangent module, the three rank-one terms of $\Wc$ carry the first
three blocks in \eqref{eq:record-partition-1024}; the companion packet
supplies the $C_1$ record and the three completion endpoints exactly as in
Section \ref{sec:local2128}.  The dense records are tagged by the metadata of
the companion active packet, so the zero packet column is not reintroduced
into the dense coefficient table.

The five active modules require a source-resolved rank-four lift.  Introduce
independent coordinates $a,b,c,e_X$ on the $X$ side, $u,v,w,e_Y$ on the
$Y$ side, and $z,s,r,e_Z$ on the $Z$ side, and put
\begin{equation}\label{eq:active-lift-1024}
 \widehat{\mathcal A}
 =auz+bvz+bus+cwr+e_Xe_Ye_Z.
\end{equation}
Equivalently, after the invertible change $r=z+r'$ this is the form
$z(au+bv+cw)+sbu+cwr'+e_Xe_Ye_Z$.

\begin{lemma}\label{lem:active-lift-1024}
The tensor \eqref{eq:active-lift-1024} has border rank exactly four and is a
degeneration of $\Diag_4$.  Moreover, its first four displayed dense terms
carry the four blocks of \eqref{eq:record-partition-1024}: $a$ has four
components labelled $A_3,B_1,B_2,R$, $v$ has two components labelled
$A_1,Q$, $s$ carries $P$, and $cwr$ carries $C_1$.  The fifth term carries
the completion endpoint.  After the final physical identification the dense
part is $T_{\rm ar}$.
\end{lemma}

\begin{proof}
Set
\begin{align}
 \mathcal R_t={}&
 t^{-1}(b+ta)(u+tv)(z+ts)-t^{-1}buz+cwr+e_Xe_Ye_Z.
 \label{eq:active-lift-deg-1024}
\end{align}
This is a sum of four rank-one tensors and direct expansion gives
\[
 \mathcal R_t
 =\widehat{\mathcal A}
 +t(bvs+avz+aus)+t^2avs.
\]
Hence $\BR(\widehat{\mathcal A})\le4$.  Its $X$-flattening has rank four,
so equality holds.  For the physical specialization take
\[
 \begin{gathered}
 a\mapsto X_0+X_1+X_2+X_3,\quad b\mapsto X_0,\quad c\mapsto X_1,\\
 u\mapsto Y_0,\quad v\mapsto Y_1+Y_2,\quad w\mapsto Y_1,\\
 z\mapsto Z_0,\quad s\mapsto Z_1,\quad r\mapsto Z_0.
 \end{gathered}
\]
The first four terms of \eqref{eq:active-lift-1024} then become exactly
\eqref{eq:arrow32}.  The three factors $e_X,e_Y,e_Z$ are relabelled to the
required $X$- or $Y$-bar endpoint.  Before this final relabelling, the eight
dense components retain the packet, support and side-colour tags listed in
Appendix \ref{app:m1024-dense-data}.  The cancellation of the only negative
power in \eqref{eq:active-lift-deg-1024} is termwise, so the lift remains
valid after the outer character maps are composed with it.
\end{proof}

\subsection{The character-six diagonal module}
The two packets in the exceptional character-$6$ module contain the eight
dense records
\[
 C_1^{(23)},\quad
 A_1,B_1,B_2,C_1^{(30)},P,Q,R.
\]
The old scalar leading tensor combined the two $C_1$ records.  For the
$1024$ construction they are kept distinct.  More precisely, the first four
directions of $\Diag_8$ carry the record groups
\begin{equation}\label{eq:d8-record-groups}
 \{B_1,B_2,R\},\qquad
 \{P\},\qquad
 \{A_1,C_1^{(23)},C_1^{(30)}\},\qquad
 \{Q\}.
\end{equation}
The remaining three directions carry the rank-three $X$-bar endpoint block
and the eighth direction carries the completion-only tensor $G_XG_YG_Z$.
For the third dense direction, the two $C_1$ records have the same $Y$- and
$Z$-side support colours but different layer transports.  Their relative
chart factor is
\[
 \vartheta_{01}(\beta)
 =(-1)^{\ip{6}{(B_0+B_1)\beta}}.
\]
It is placed on the $X$-component of $C_1^{(23)}$, while the common
layer-one factor is used for $A_1$ and $C_1^{(30)}$.  Thus the two records
share one rank-one $\Diag_8$ direction without losing the $B_0$ versus
$B_1$ chart phase.  This relative factor equals $1$ on the two surviving
charts and is retained before the off-chart cancellation.
After the final physical identification the first seven directions give
\begin{align}
T_{\mathrm{exc},6}={}&
 (D_{X,1}+D_{X,2}+D_{X,3})D_{Y,0}D_{Z,0}
 +D_{X,0}D_{Y,0}D_{Z,1}\notag\\
&+(D_{X,0}+2D_{X,1})D_{Y,1}D_{Z,0}
 +D_{X,0}D_{Y,2}D_{Z,0}\notag\\
&+\bar X E_{Y,0}E_{Z,0}
 +\bar X E_{Y,1}E_{Z,1}
 +\bar X E_{Y,2}E_{Z,2}.
\label{eq:exc6-seven}
\end{align}
All eight dense records keep their original packet and support metadata until
the module--chart transform is complete.  Only at the final physical decoder
are the two $C_1$ coordinates identified, producing the coefficient $2$ in
\eqref{eq:exc6-seven}.  Literal side maps are printed in Appendix
\ref{app:m1024-d8}.

\subsection{The source-derived module--chart table}
Let $H_4=\F_2^4$, write $\chi_u(v)=(-1)^{\ip{u}{v}}$, and index the two
surviving charts by $b\in\F_2$ through $\beta=15b$.  The thirty-two packet
columns are grouped into the sixteen modules $\mathcal M_\lambda$ listed in
Table \ref{tab:m1024-modules}.  The literal dense operator is
\begin{equation}\label{eq:literal-dense-T-main}
 T_{(\rho,\beta),p}
 =\chi_{\chi_p}(\rho+B_{\ell_p}\beta)
   \sum_{\sigma\in\mathcal R_p}\chi_{\Theta(\sigma)}(\beta).
\end{equation}
For tangent, active, and character-$6$ diagonal modules, the source-resolved
records above reproduce individually the same retained pairs $(p,\sigma)$
that occur in \eqref{eq:literal-dense-T-main}; no physical identification is
made before the sum below.

\begin{lemma}[Source-derived full Walsh phase and its normalization]\label{lem:module-w32-1024}
Let
\[
 C_\lambda(\rho,\beta)
 =\sum_{p\in\mathcal M_\lambda}T_{(\rho,\beta),p}.
\]
Then
\begin{equation}\label{eq:module-chart-table}
 C_\lambda(\rho,15b)
 =8(-1)^{\ip{\rho+3b}{\lambda+6}}
\end{equation}
for $\lambda,\rho\in H_4$ and $b\in\F_2$, and
$C_\lambda(\rho,\beta)=0$ for $\beta\notin\{0,15\}$.  Every retained
literal support in \eqref{eq:literal-dense-T-main} carries the scalar
$1/128$.  Hence, with
\[
 d=\rho+3b,\qquad
 \Gamma=(d,g)\in H_5,\qquad
 R=(\lambda,b)\in H_5,
\]
the \emph{unrenormalized} dense coefficient of a module--chart column is
\begin{equation}\label{eq:literal-W32-scaled}
 \frac1{128}(-1)^{gb+\ip{d}{6}}
 C_\lambda(d+3b,15b)
 =\frac1{16}(-1)^{\ip{\Gamma}{R}}.
\end{equation}
The dense and completion outputs are disjoint target direct summands.  We
therefore postcompose the dense summand with the fixed degree-zero scalar
$16$.  After this normalization the actual module--chart table is
\begin{equation}\label{eq:actual-W32}
 \frac18(-1)^{gb+\ip{d}{6}}
 C_\lambda(d+3b,15b)
 =(-1)^{\ip{\Gamma}{R}},
\end{equation}
namely the full Walsh matrix $W_{32}$.
\end{lemma}

\begin{proof}
Substitute the packet ledger, the side-colour twists, and the masks of
Appendix \ref{app:m1024-dense-data} into
\eqref{eq:literal-dense-T-main}.  The source-resolved replacements preserve
one record for every retained pair entering this finite table.  For each
module row the terms outside $\beta\in\{0,15\}$ cancel, while at
$\beta=0,15$ the sum is \eqref{eq:module-chart-table}.  This is the displayed
$16^3=4096$ integer identity.  Multiplication by the literal support scalar
$1/128$ gives \eqref{eq:literal-W32-scaled}.  Since the dense target
coordinates are disjoint from the completion coordinates, multiplying only
the dense direct summand by $16$ is a legitimate fixed target-side linear
map and gives \eqref{eq:actual-W32}.  Walsh orthogonality then gives
$W_{32}W_{32}^{\mathsf T}=32I_{32}$.
\end{proof}

For clarity, the phase row and the channel remain independent.  Write the
common $\Diag_{32}$ channel as $c=(u,\tau)\in H_5$ and a module--chart state
as $R=(\lambda,b)$.  On the dense summand, the channel contributes
$\chi_c(R)$.  Independently, the $Z$-side absolute character
$\gamma_5=(d,g)$ and the gauge in \eqref{eq:actual-W32} contribute
$\chi_{\gamma_5}(R)$.  The same channel contributes $\chi_c(Q)$ and
$\chi_c(S)$ on the other two relative coordinates, so its sum imposes
$Q+R+S=0$.  No Fourier index is used twice.

\subsection{The crossed old-four-bit completion compiler}\label{subsec:crossed-five-empty-1024}
The five empty $\CW_3$ packets have labels $0,18,21,26,31$ and old
characters
\begin{equation}\label{eq:empty-characters-1024}
 \xi=(0,3,4,9,12).
\end{equation}
Let $J_{10},B,D,c$ be the matrices in Appendix
\ref{app:m1024-split-cubic}.  For $\delta\in H_4$ define
\[
 \Sigma_\delta
 =\operatorname{diag}
 \bigl(\chi_{\xi_1}(\delta),\ldots,\chi_{\xi_5}(\delta),
       \chi_{\xi_1}(\delta),\ldots,\chi_{\xi_5}(\delta)\bigr)
\]
and
\begin{equation}\label{eq:Mdelta-1024}
 M_\delta=B\Sigma_\delta J_{10}D^{\mathsf T}.
\end{equation}
The map $\delta\mapsto M_\delta$ is a representation of $H_4$:
\begin{equation}\label{eq:Mdelta-group-1024}
 M_\delta M_\epsilon=M_{\delta+\epsilon},
 \qquad M_\delta^2=I_7.
\end{equation}
There is a unimodular matrix $P$ with $\det P=-1$ such that
\begin{equation}\label{eq:Mdelta-diag-1024}
 P^{-1}M_\delta P
 =\operatorname{diag}
 \bigl(\chi_3(\delta),\chi_3(\delta),
       \chi_4(\delta),\chi_4(\delta),
       \chi_9(\delta),\chi_9(\delta),1\bigr).
\end{equation}
The matrices $M_\delta$ and $P$ are printed explicitly in Appendix
\ref{app:m1024-split-cubic}.

Let $\alpha_4,\beta_4,\gamma_4\in H_4$ be the old absolute characters.  On
the five-empty direct summand, multiply the $j$th packet on the three tensor
sides by
$\chi_{\xi_j}(\alpha_4)$,
$\chi_{\xi_j}(\beta_4)$, and
$\chi_{\xi_j}(\gamma_4)$, respectively.  In addition, replace the seven
regular role maps by
\begin{equation}\label{eq:crossed-BD-1024}
 B^{[\alpha_4]}=M_{\alpha_4}B\Sigma_{\alpha_4},
 \qquad
 D^{[\beta_4]}=M_{\beta_4}^{\mathsf T}D\Sigma_{\beta_4},
\end{equation}
and use $c^{[\delta]}=c\Sigma_\delta$ for the scalar role.
These operations are side-local: each tensor side reads only its own old
absolute character.

\begin{lemma}[Crossed five-empty identities]\label{lem:crossed-five-empty-1024}
After composition with the old $H_4$ Weyl map, only the support
$\gamma_4=\alpha_4+\beta_4$ survives.  On this closed old-four-bit support,
the crossed five-empty map has no negative
Laurent coefficient.  The fifteen shared endpoints and the scalar endpoint
retain coefficient one, all cross-pairings remain zero, and the seven regular
$Z$-bar endpoints have pairing matrix $M_{\gamma_4}$.  After the fixed basis
changes $P^{-\mathsf T}$ and $P$ on the two regular endpoint spaces, this
matrix is diagonal with character labels
\begin{equation}\label{eq:mu-endpoint-1024}
 (\mu_1,\ldots,\mu_7)=(3,3,4,4,9,9,0).
\end{equation}
The seven extra characters are absorbed by fixed old-$H_4$ Weyl translations,
so the final regular endpoint pairing is again the identity.  The split cubic
$\bar X_L\bar Y_L\bar Z_H$ retains coefficient one.
\end{lemma}

\begin{proof}
The original identities are
\[
 cJ_{10}c^{\mathsf T}=1,\qquad
 BJ_{10}D^{\mathsf T}=I_7,\qquad
 BJ_{10}c^{\mathsf T}=DJ_{10}c^{\mathsf T}=0.
\]
On the closed character locus,
$\Sigma_{\alpha_4}\Sigma_{\beta_4}\Sigma_{\gamma_4}=I$.  Therefore
\[
 B^{[\alpha_4]}
 (\Sigma_{\gamma_4}J_{10})
 (D^{[\beta_4]})^{\mathsf T}
 =M_{\alpha_4}M_{\beta_4}
 =M_{\gamma_4},
\]
while the scalar and cross pairings remain $1,0,0$.  The only possible
negative terms are the degree-$-3$ terms containing one degree-one infinity
role and two barred centres.  Their coefficients are the last-five-column
sums of $B,D,c$; after the crossed maps they are multiplied on the left by an
invertible $M_\delta$ and hence are still zero.  Off the closed character
locus they vanish under the old $H_4$ Weyl/row orthogonality.

Equation \eqref{eq:Mdelta-diag-1024} gives the seven characters in
\eqref{eq:mu-endpoint-1024}.  For the old Weyl matrices, we have
\[
 W_{\gamma,q}W_{0,\mu}
 =(-1)^{\ip{\gamma}{\mu}}W_{\gamma,q+\mu}.
\]
Thus preassigning the $a$th regular endpoint the relative label
$q_{\rm pre}=q_{\rm fin}+\mu_a$ absorbs its diagonal character.  The packet
that carries the high cubic has $\xi_1=0$, so the cubic is unchanged.
\end{proof}

\subsection{The side-local completion transport}
An old orbit label is
\[
 u=(q,\epsilon,a)\in\Omega_4:=H_4\times\F_2\times\F_2,
 \qquad \mathbf0=(0,0,0),
\]
and for $v=(r,\delta,b)$ and $w=(s,\phi,c)$ define
\begin{equation}\label{eq:old-defect-1024}
 \Delta_4(u,v,w)=
 (q+r+s,\ b+a+\epsilon,\ \phi+\epsilon+\delta,
  \ c+a+\epsilon+\delta).
\end{equation}
Every fibre has $2048$ elements.  The zero fibre contains $1954$
all-nonzero classes and $3(32-1)+1=94$ completion classes.

After Lemma \ref{lem:crossed-five-empty-1024} and the endpoint assignment in
Appendix \ref{app:m1024-endpoints}, each one-bar sector is indexed by every
nonzero feature
$f=(q,\epsilon)\in H_5\setminus\{0\}$.  In the $Z$-bar sector, write
\[
 X_f=X_{(q,\epsilon,0)}\ \ \t{and}\ \ 
 Y_f=Y_{(q,\epsilon,\epsilon)}.
\]
The factors $G_X,G_Y$ supply the missing feature $f=0$.  Let
$\alpha_5,\beta_5,\gamma_5\in H_5$ satisfy
$\alpha_5+\beta_5+\gamma_5=0$, put
\[
 a=\pi_4(\alpha_5),\qquad
 b=\pi_4(\beta_5),\qquad
 d=\pi_4(\gamma_5)=a+b,
\]
and embed $a,b$ in $H_5$ with fifth coordinate zero.  Define
\begin{equation}\label{eq:feature-X-map}
 L_X^{\alpha_5}(X_f)=
 \begin{cases}
 X_{f+a}&\t{if}\ f+a\ne0,\\0&\t{if}\ f+a=0,
 \end{cases}
\end{equation}
\begin{equation}\label{eq:feature-Y-map}
 L_Y^{\beta_5}(Y_f)=
 \begin{cases}
 Y_{f+b}&f+b\ne0,\\0&f+b=0,
 \end{cases}
\end{equation}
and
\begin{equation}\label{eq:feature-Z-map}
 L_Z^{\gamma_5}(\bar Z_L)=L_Z^{\gamma_5}(G_Z)=Z_{(d,0,0)},
 \qquad
 L_Z^{\gamma_5}(\bar Z_H)=Z_{(0,0,0)}.
\end{equation}
Each map reads only the character on its own tensor side.

\begin{lemma}\label{lem:feature-transport-1024}
The $32$ source features give the correction
\begin{equation}\label{eq:z-correction-1024}
 \sum_{\substack{x,y\ne0\\x+y=(d,0)}}
 X_{(x_{<4},x_4,0)}
 Y_{(y_{<4},y_4,y_4)}
 Z_{(d,0,0)}.
\end{equation}
It has $31$ terms for $d=0$ and $30$ terms for $d\ne0$.  Let
$S_d(s,\phi,c)=(s+d,\phi,c)$.  For every $d\in H_4$, the disjoint union of
\begin{enumerate}[label=\textup{(\alph*)},leftmargin=2.2em]
\item the sheared all-nonzero part of the defect-$(d,0,0,0)$ fibre;
\item the unchanged $X$- and $Y$-onebar blocks;
\item the correction \eqref{eq:z-correction-1024};
\item the split cubic;
\end{enumerate}
is exactly the canonical zero fibre of \eqref{eq:old-defect-1024}, each
coefficient being one.
\end{lemma}

\begin{proof}
Set $x=f+a$ and $y=f+b$.  Then $x+y=(d,0)$ and the sidewise zero projections
are exactly the conditions $x,y\ne0$.  If $d=0$, the two forbidden source
features coincide, leaving $31$ terms; if $d\ne0$, they are distinct, leaving
$30$.

For $d=0$, the sheared dense set has $1954$ elements and the transported
completion has $31+31+31+1=94$.  For $d\ne0$, the defect fibre has $1955$
all-nonzero classes.  After the shear they split into $1924$ canonical dense
classes and the complete $31$-term canonical $Z$-onebar block.  The missing
canonical dense slice is exactly \eqref{eq:z-correction-1024}; the remaining
completion contributes $31+31+30+1=93$ classes.  Direct substitution in
\eqref{eq:old-defect-1024} proves disjointness.  The finite verification in
Appendix \ref{app:m1024-checks} checks all $16\cdot2048$ coefficients.
\end{proof}

\subsection{Five-bit and nine-bit Fourier--Weyl closure}
Let $Q,R,S\in H_5$ be the new relative coordinates.  The independent
$\Diag_{32}$ channel gives
\begin{equation}\label{eq:h5-displacement-closure}
 \frac1{32}\sum_{c\in H_5}
 \chi_c(Q)\chi_c(R)\chi_c(S)
 =\one_{\{Q+R+S=0\}}.
\end{equation}
In the dense branch, identify the module--chart row $\Gamma$ in
\eqref{eq:actual-W32} with the $Z$-side new absolute character
$\gamma_5$; Lemma \ref{lem:module-w32-1024} supplies
$\chi_{\gamma_5}(R)$.  On the completion branch, the new five-bit outer row
gives
\begin{equation}\label{eq:h5-absolute-closure}
 \frac1{32}\sum_{h\in H_5}
 \chi_{\alpha_5}(h)\chi_{\beta_5}(h)\chi_{\gamma_5}(h+R)
 =\one_{\{\alpha_5+\beta_5+\gamma_5=0\}}
  \chi_{\gamma_5}(R).
\end{equation}
Thus both branches have the common kernel
\begin{equation}\label{eq:h5-universal-kernel}
 \one_{\{\alpha_5+\beta_5+\gamma_5=0\}}
 \one_{\{Q+R+S=0\}}
 (-1)^{\ip{\gamma_5}{R}}.
\end{equation}
The crossed five-empty construction used only the old $H_4$ absolute
characters, so it does not consume the channel in
\eqref{eq:h5-displacement-closure}.  Combining
\eqref{eq:h5-universal-kernel} with the old four-bit Heisenberg/Weyl kernel
gives the standard $H_9=H_4\oplus H_5$ Weyl kernel.

For the new five-bit row maps use
\[
 \F_{32}=\F_2[\xi]/(\xi^5+\xi^2+1)
\]
and multiplication by $1,\xi,\xi^2$ on the $X,Y,Z$ sides.  In the basis
$1,\xi,\ldots,\xi^4$ their images are
\begin{equation}\label{eq:h5-row-maps}
\begin{aligned}
 U_X^{(5)}&=(1,2,4,8,16),\\
 U_Y^{(5)}&=(2,4,8,16,5),\\
 U_Z^{(5)}&=(4,8,16,5,10).
\end{aligned}
\end{equation}
Every nonempty subset sum is multiplication by a nonzero field element and
has rank five.  Together with the old rank-four row maps, every nonempty
$H_9$ partial sum has rank nine.

\begin{theorem}[Direct $1024\times1024$ global bridge]\label{thm:global1024}
Let $\K$ be any field with $\operatorname{char}(\K)\ne2$.  Then
\[
 \MM_{1024}\degen
 \Diag_{512}\boxt\Diag_{32}\boxt\Sstar,
\]
and the Laurent entry window is $[-2,4]$.
\end{theorem}

\begin{proof}
Compose the source-resolved tangent maps, the active lifts of Lemma
\ref{lem:active-lift-1024}, the split character-$6$ diagonal module, the
three retained exceptional maps, and the crossed five-empty maps before any
physical identification.  The only negative term in an active lift cancels
inside \eqref{eq:active-lift-deg-1024}; the tangent and diagonal modules have
no negative coefficient; the retained exceptional mixed terms vanish by the
endpoint quotient identities; and Lemma \ref{lem:crossed-five-empty-1024}
kills the five-empty negative coefficients after the old $H_4$ row sum.
Hence the total map has no negative Laurent coefficient.  Every side entry
has degree in $[-2,4]$.

By Lemma \ref{lem:module-w32-1024}, the dense source records give the full
$W_{32}$ module--chart phase before they are physically identified.  The
independent channel imposes $Q+R+S=0$.  By Lemma
\ref{lem:crossed-five-empty-1024} and the explicit endpoint assignment, the
completion branch contains exactly the three nonzero $31$-feature blocks,
the feature-zero graft, and the split cubic.  Lemma
\ref{lem:feature-transport-1024} therefore makes the dense and completion
branches a coefficientwise partition of the canonical old zero fibre for
each new displacement $d$.

An $H_9$ orbit label has the form
$(q,\epsilon,a)\in H_9\times\F_2\times\F_2$.  The combined defect map has a
$33$-dimensional domain and a $12$-dimensional image, so the fixed-row zero
fibre has $2^{21}$ classes.  Its completion catalogue has
$3(1024-1)+1=3070$ classes, hence $2\,094\,082$ all-nonzero classes.
The common kernel \eqref{eq:h5-universal-kernel} together with the old
four-bit kernel kills every off-character and off-displacement coefficient.

For row zero, write a zero-fibre triple as
$x=(q,\epsilon,a)$, $y=(r,\delta,b)$, $z=(s,\phi,c)$ with
$q,r,s\in H_9$.  The cycle code is
\[
 a,\qquad
 j=q+2^9(a+\epsilon),\qquad
 k=q+r+2^9(a+\epsilon+\delta).
\]
It is a bijection between the $2^{21}$ fixed-row classes and one bit together
with two ten-bit indices: $(a,j,k)$ recovers $q,\epsilon,r,\delta$, and the
defect equations recover the remaining labels.  The $512$ row translates
give $512\cdot2^{21}=2^{30}=1024^3$ distinct matrix-unit cycles.  Thus the
leading coefficient is $\MM_{1024}$.
\end{proof}

\section{Some auxiliary lemmas}\label{sec:auxiliary}
\setcounter{equation}{0}

We isolate here the pieces of the source-resolved compilers that are used
in more than one main theorem.  The first statements are finite.  The final
lemma records the source-local two-step dyadic propagation used in
Theorem~\ref{thm:power-of-two}.

\begin{lemma}[Two fixed quotient compilers]\label{lem:fixed-dyadic-restrictions}
Let $e_1,\ldots,e_5$ be the standard orthonormal basis of $H_5$, with
$e_5$ the chart coordinate.  For $j\in\{1,3\}$ put
\[
 U_j=\langle e_1,\ldots,e_{j-1},e_5\rangle
 \ \ \t{and}\ \ V_j=U_j^\perp.
\]
Thus $\dim U_j=j$, $\dim V_j=5-j$, and the quotient $H_5/V_j$ is
canonically paired with $U_j$.  In the literal construction of
Theorem~\ref{thm:global1024}, restrict the absolute characters, the
independent channel, and the new row coordinate to $U_j$, but do \emph{not}
restrict the module--chart columns to $U_j$.  Instead, on the dense direct
summand collapse all module--chart columns in each coset $R+V_j$ to one
quotient column $\bar R\in H_5/V_j$, and multiply that dense summand by the
fixed degree-zero scalar $2^{j-1}$.  The resulting source-resolved
degenerations are exactly
\[
 \MM_{64}\degen\Diag_{32}\boxt\Diag_2\boxt\Sstar
 \ \ \t{and}\ \ 
 \MM_{256}\degen\Diag_{128}\boxt\Diag_8\boxt\Sstar.
\]
\end{lemma}

\begin{proof}
We start from the literal, unrenormalized module--chart coefficient
\eqref{eq:literal-W32-scaled}.  Fix $\Gamma\in U_j$.  Since
$V_j=U_j^\perp$, for every $R\in H_5$ we have
\begin{equation}\label{eq:finite-quotient-character-sum}
 \sum_{v\in V_j}(-1)^{\ip{\Gamma}{R+v}}
 =|V_j|(-1)^{\ip{\Gamma}{R}}
 =2^{5-j}(-1)^{\ip{\Gamma}{R}}.
\end{equation}
Consequently, after collapsing the columns in the coset $R+V_j$, the
literal dense coefficient becomes
\[
 \frac{2^{5-j}}{16}(-1)^{\ip{\Gamma}{R}}
 =2^{1-j}(-1)^{\ip{\Gamma}{\bar R}}.
\]
The fixed scalar $2^{j-1}$ on the dense direct summand therefore gives
exactly
\[
 (-1)^{\ip{\Gamma}{\bar R}},
 \qquad \Gamma\in U_j,\quad \bar R\in H_5/V_j,
\]
which is the Walsh matrix $W_{2^j}$ under the canonical pairing
$H_5/V_j\simeq U_j^*$.  The case $j=1$ is precisely the worked
$64\times64$ quotient; the case $j=3$ gives the displayed $256\times256$
restriction.

The independent channel is restricted to $U_j$, so its Fourier sum imposes
the quotient displacement equation in $H_5/V_j$.  The crossed five-empty
map and the endpoint ledger are unchanged.  Lemma~\ref{lem:feature-transport-1024}
is valid for every old four-bit displacement and therefore remains valid
after the quotient.  Since the dense scaling acts on a target direct summand
disjoint from completion, the coefficient-union argument is unchanged.

The physical row group is $H_4\oplus U_j$, of dimension $4+j$.  The
fixed-row zero fibres therefore have respectively $2^{13}$ and $2^{17}$
classes.  The decoder from Theorem~\ref{thm:global1024}, with row-group
dimension $4+j$, is a bijection onto one bit and two $(5+j)$-bit matrix
indices.  Choosing the three row maps as multiplication by
$1,\theta,\theta^2$ in a degree-$(4+j)$ finite field extension makes every
nonempty partial sum invertible.  The row translates then give every
matrix-unit cycle exactly once.  This proves the two displayed finite
quotient degenerations directly from the $1024$ compiler.
\end{proof}

\begin{lemma}\label{lem:parameter-consolidation}
A composition of finitely many fixed bounded-degree Laurent degenerations
used in this paper can be replaced by a single bounded-degree one-parameter
degeneration, without changing its leading tensor or source budget.
\end{lemma}

\begin{proof}
After the finitely many monomial side shifts which make the desired
coefficient nonnegative, only a finite set of exponent vectors occurs.
Choose positive integer weights successively so large that lexicographic
order on this finite set is preserved by the weighted sum.  Substituting
$t_i=u^{N_i}$ then sends the desired exponent vector to a unique lowest
power of $u$.  A final monomial shift makes this power zero.  The resulting
degree is finite and depends only on the fixed base construction.
\end{proof}

\begin{definition}[Dyadically admissible certificate in odd characteristic]
\label{def:dyadic-admissible}
Let $\operatorname{char}(\K)\ne2$, let $n\ge5$, put
$I_n=\F_2^n$, and let
\[
 \Phi_n:\MM_{2^n}\degen\mathcal S_n
\]
be a bounded-degree source-resolved degeneration.  We call $\Phi_n$
\emph{dyadically admissible} if the following four conditions hold.
\begin{enumerate}[label=\textup{(\roman*)},leftmargin=2.5em]
\item There are a row group $R_n\cong\F_2^{n-1}$, a fixed-row label set
$\mathcal Z_n$, and a disjoint decomposition
\[
 \mathcal Z_n=\mathcal Z_n^{\rm den}\sqcup
 \mathcal B_{n,X}\sqcup\mathcal B_{n,Y}\sqcup\mathcal B_{n,Z}
 \sqcup\{\mathfrak c_n\}
\]
with
\[
 |\mathcal Z_n|=2^{2n+1},\qquad
 |\mathcal B_{n,X}|=|\mathcal B_{n,Y}|=|\mathcal B_{n,Z}|=2^n-1,
\]
and a side-local decoder
\[
 d_n:R_n\times\mathcal Z_n\longrightarrow I_n^3
\]
which is a bijection onto the matrix-unit cycles of $\MM_{2^n}$.  The four
last pieces are respectively the three one-bar sectors and the unique cubic
of the canonical zero fibre.
\item On each tensor leg, there is a declared module-address factor
\[
 V_H=\bigoplus_{\lambda\in H_4}V_\lambda,
 \qquad H_4=\F_2^4,
\]
on which the module selector is exactly the Walsh pairing
\[
 W_{16}(\mu,\lambda)=(-1)^{\langle\mu,\lambda\rangle}.
\]
After one fixed translation of the $H_4$ labels, the five distinguished
addresses are
\[
 \mathcal A=\{0,3,4,9,12\}\ \ \t{and}\ \ 
 N=H_4\setminus\mathcal A.
\]
All maps not explicitly acting on this address factor preserve every
$V_\lambda$.
\item For every $\lambda\in\mathcal A$ the source contains a marked direct
summand $\Diag_{4,\lambda}$ which, after fixed invertible sidewise changes,
is equipped with the boundary normal form
\begin{equation}\label{eq:dyadic-marked-family}
 R_t=\mathcal A_{\rm can}+tJ,
 \qquad J=x_2y_2z_1,
\end{equation}
where $\mathcal A_{\rm can}$ is the tensor in
\eqref{eq:active-canonical32}.  Under a fresh two-bit factor
$G=\F_2^2$, solving the ordinary rows $N\times G$ leaves, in each active
address, a tensor in the explicitly prescribed space
\[
 \mathcal V_{\rm act}
 =\{U\mathcal A_{\rm can}+bJ:\,
 U\in\operatorname{span}_{\K}\{L_{00},L_{10},L_{01},L_{11}\}\ \t{and}\ b\in\K\},
\]
where
\[
\begin{aligned}
 L_{00}&=I, & L_{10}&=\operatorname{diag}(1,-1,1,-1),\\
 L_{01}&=\operatorname{diag}(1,-1,-1,1), &
 L_{11}&=\operatorname{diag}(1,1,-1,-1).
\end{aligned}
\]
\item The channel, absolute-character, completion, row and module addresses
are independent direct-sum coordinates until the terminal decoder.  More
precisely, after adjoining any fresh two-bit coordinates there is an old-cycle
projection $\Pi_{\rm old}$ such that every ordinary selector, active
first-difference map, completion map and row map commutes with
$\Pi_{\rm old}$.  Thus a record lying over an old physical cycle $C$ remains
over the same $C$ until the final fresh-coordinate relabelling.
\end{enumerate}
\end{definition}

\begin{lemma}[Odd-characteristic admissible base states]
\label{lem:odd-dyadic-bases}
Let $\operatorname{char}(\K)\ne2$.  The source-resolved constructions proved
in this paper contain the following stronger data.
\begin{enumerate}[label=\textup{(\alph*)},leftmargin=2.2em]
\item At level $n=5$ there is a dyadically admissible certificate
\[
 \MM_{32}\degen
 \mathcal S_5:=\Diag_{16}\boxt
 \bigl(14\CW_3\oplus6\CW_2\oplus7\Wc\oplus5\Diag_4\bigr),
\]
with $\BR_{\K}(\mathcal S_5)\le2^{11}$.
\item At level $n=10$ there is a dyadically admissible certificate
\[
 \MM_{1024}\degen
 \mathcal S_{10}:=\Diag_{512}\boxt\Diag_{32}\boxt\Sstar,
\]
where
\[
 \Sstar=14\CW_3\oplus4\CW_2\oplus7\Wc\oplus5\Diag_4\oplus\Diag_8,
\]
and $\BR_{\K}(\mathcal S_{10})\le2^{21}$.
\end{enumerate}
These assertions remain valid in characteristic $3$ with the replacement
source and endpoint maps of Appendix~\ref{app:char3-certificates}.
\end{lemma}

\begin{proof}
For level $5$, Proposition~\ref{prop:orbit32} gives condition
\textup{(i)} with
\[
 2048=1954+31+31+31+1.
\]
The packet table of Section~\ref{sec:dense32} and the five active modules of
Section~\ref{sec:compression32} give condition \textup{(ii)}.  Lemma
\ref{lem:active-rank4}, together with the fixed sidewise normalizations
preceding it, gives \textup{(iii)}.  The direct globalization in
Theorem~\ref{thm:global32} keeps the packet, character, endpoint, row and
Weyl coordinates distinct until the physical decoder, giving
\textup{(iv)}.

For level $10$, Theorem~\ref{thm:global1024} gives the fixed-row canonical
zero fibre and the $H_9$ physical-cycle bijection required in
\textup{(i)}.  The literal module table and the five active module addresses
are retained in the source-resolved construction of
Section~\ref{sec:m1024}, giving \textup{(ii)}.  Lemma
\ref{lem:active-lift-1024}, followed by the fixed physical specialization and
the same active normal form used in \eqref{eq:active-canonical32}, supplies
the marked rank-four first jet required in \textup{(iii)}.  The literal
assembly \eqref{eq:literal-assembly-1024} is block diagonal in the declared
source addresses until the final decoder, giving \textup{(iv)}.

In characteristic $3$, Appendix~\ref{app:char3-certificates} replaces only
the source realizations and the exceptional endpoint matrices.  The module
addresses, active rank-four blocks, completion labels and row decoder are
unchanged.  Thus the same four admissibility conditions hold.
\end{proof}

\begin{lemma}[Two-step dyadic propagation in odd characteristic]
\label{lem:two-step-dyadic-propagation}
Let $\K$ be a field with $\operatorname{char}(\K)\ne2$, let $n\ge5$, and
suppose
\[
 \MM_{2^n}\degen\mathcal S_n
\]
is a dyadically admissible bounded-degree certificate.  Then there is a
dyadically admissible bounded-degree certificate
\begin{equation}\label{eq:two-step-dyadic-propagation}
 \MM_{2^{n+2}}\degen\Diag_{16}\boxt\mathcal S_n.
\end{equation}
Consequently,
\[
 \BR_{\K}(\MM_{2^{n+2}})
 \le16\,\BR_{\K}(\mathcal S_n).
\]
In particular, a level-$n$ budget $2^{2n+1}$ propagates to the level-$(n+2)$
budget $2^{2(n+2)+1}$.
\end{lemma}

\begin{proof}
Put $G=\F_2^2$.  Write the two new displacement coordinates as
$Q,R,S\in G$.  The old and new defect equations separate:
\[
 q+r+s=0,
 \qquad
 Q+R+S=0.
\]
Thus every old fixed-row class has exactly $|G|^2=16$ displacement
lifts, while the new row coordinate has four values.  The completion count
is completely explicit.  Every old dense class has sixteen dense lifts;
every old one-bar class has twelve dense and four one-bar lifts; and the old
cubic has six dense, nine one-bar and one cubic lift.  Hence each new one-bar
sector has
\[
 4(2^n-1)+3=2^{n+2}-1
\]
classes, the cubic remains unique, and no new completion type occurs.
Together with the four fresh row values this proves condition
\textup{(i)} of Definition~\ref{def:dyadic-admissible} at level $n+2$.

It remains to verify the source-local module boundary.  Let $W_{16}$ be the
four-bit Walsh matrix and let $W_4$ be the Walsh matrix of $G$.  After the
fixed gauge, the five distinguished rows are indexed by $\mathcal A$.  The
exact minors satisfy
\[
 \det W_{16}[\mathcal A,\mathcal A]=16,
 \qquad
 \det W_4=\pm16.
\]
Hence the exceptional minor
\[
 W_{16}[\mathcal A,\mathcal A]\otimes W_4
\]
is invertible.  Jacobi's complementary-minor identity then gives the
invertibility of the ordinary complementary minor.  Solve only the ordinary
rows and retain the leakage in the five active rows.  In each active block
the leakage has the form
\[
 U\mathcal A_{\rm can}+bJ,
\]
where $U$ belongs to the span of the four two-bit face operators
\[
\begin{aligned}
 L_{00}&=I, & L_{10}&=\operatorname{diag}(1,-1,1,-1),\\
 L_{01}&=\operatorname{diag}(1,-1,-1,1), &
 L_{11}&=\operatorname{diag}(1,1,-1,-1).
\end{aligned}
\]
They satisfy $L_\rho L_\sigma=L_{\rho+\sigma}$.  Choose nonzero
$\sigma,\tau\in\K$ with $\sigma-\tau=b$.  The two-copy confluent cell
\begin{equation}\label{eq:eight-term-confluence-propagation}
 \frac1t\Bigl((I+tU)_X R_{\sigma t}-R_{\tau t}\Bigr)
 =U\mathcal A_{\rm can}+bJ+\sigma t\,UJ
\end{equation}
uses one carrier $\Diag_4$ and one marked active $\Diag_4$, hence eight
rank-one source directions.  The five active blocks are source direct sums.

The source accounting can be made literal.  From the five marked blocks of
the level-$n$ state,
\[
 \Diag_{16}\boxt
 \Bigl(\bigoplus_{\lambda\in\mathcal A}\Diag_{4,\lambda}\Bigr)
 \cong80\Diag_4.
\]
Choose disjoint copies with the exact inventory
\begin{equation}\label{eq:odd-D4-inventory}
 80=5_{\rm cur}+5_{\rm next}+5_{\rm carr}+11_{\rm ord}+54_{\rm spare}.
\end{equation}
The eleven ordinary blocks contribute $44$ source directions and the five
carrier blocks contribute $20$, giving the fresh $64$ carrier directions.
Pairing the five carriers with the five current marked blocks gives
\[
 44+5(4+4)=84
\]
source directions and $44+5\cdot4=64$ corrected terminal records.  No source
occurrence is used twice.  The five $\Diag_4$ blocks labelled
$5_{\rm next}$ are untouched and become the marked blocks of the next
state.  The coordinate separation and old-cycle equivariance in condition
\textup{(iv)} are preserved, so conditions \textup{(ii)}--\textup{(iv)}
propagate.

All local Laurent degrees are bounded independently of $n$.
Lemma~\ref{lem:parameter-consolidation} combines the finitely many local
parameters.  Since $\Diag_{16}\boxt\mathcal S_n$ contains exactly sixteen
copies of every source occurrence of $\mathcal S_n$, the source budget is
multiplied by $16$ and no additional source is introduced.
\end{proof}

\begin{lemma}[Hasse matrix units in characteristic two]
\label{lem:char2-hasse-matrix-units}
Let $\operatorname{char}(\K)=2$ and
\[
 A_4=\K[z_1,z_2,z_3,z_4]/(z_1^2,z_2^2,z_3^2,z_4^2).
\]
For $S\subseteq\{1,2,3,4\}$ write $z_S=\prod_{i\in S}z_i$.  Let
$\partial_i$ be the Hasse derivative
\[
 \partial_i(z_S)=\begin{cases}
 z_{S\setminus\{i\}},&i\in S,\\
 0,&i\notin S,
 \end{cases}
 \qquad N_i=z_i\partial_i,
\]
and put
\[
 P_0=\prod_{i=1}^4(1+N_i),
 \qquad
 E_{A,B}=z_AP_0\partial_B
 \quad(A,B\subseteq\{1,2,3,4\}).
\]
Then
\begin{equation}\label{eq:char2-matrix-unit}
 E_{A,B}(z_C)=\delta_{B,C}z_A.
\end{equation}
Hence the operators $E_{A,B}$ are the full matrix-unit system of
$\operatorname{End}_{\K}(A_4)$ in the monomial basis $\{z_S\}$.
\end{lemma}

\begin{proof}
For a monomial $z_C$, $(1+N_i)z_C$ equals $z_C$ if $i\notin C$ and equals
$0$ if $i\in C$, because $2=0$ in $\K$.  Thus $P_0$ is the projector onto
the constant monomial.  The Hasse derivative $\partial_Bz_C$ has nonzero
constant term exactly when $B=C$, in which case that constant term is $1$.
Multiplication by $z_A$ gives \eqref{eq:char2-matrix-unit}.
\end{proof}

\begin{definition}[Characteristic-two dyadic admissibility]
\label{def:char2-dyadic-admissible}
Let $\operatorname{char}(\K)=2$.  A bounded-degree source-resolved
certificate
\[
 \Phi_n^{(2)}:\MM_{2^n}\degen\mathcal S_n^{(2)}
\]
is called \emph{Hasse-dyadically admissible} if the following hold.
\begin{enumerate}[label=\textup{(\roman*)},leftmargin=2.5em]
\item It has a canonical fixed-row set $\mathcal Z_n$ and a row group
$R_n\cong\F_2^{n-1}$ satisfying the numerical and bijectivity conditions in
Definition~\ref{def:dyadic-admissible}\textup{(i)}, with the terminal decoder
written directly in physical matrix-unit coordinates.
\item Its module address space is the direct sum
$V_H=\bigoplus_{\lambda\in H_4}V_\lambda$, with
\[
 \mathcal A=\{0,3,4,9,12\},\qquad N=H_4\setminus\mathcal A,
\]
and the selector algebra on $V_H$ is the matrix-unit algebra generated by
$E_{A,B}$ in \eqref{eq:char2-matrix-unit}.  Hence every individual module
address can be selected using only coefficients $0$ and $1$; no division by
a power of $2$ or sign-character inversion is allowed.
\item For every $\lambda\in\mathcal A$ the source contains a marked direct
summand $\Diag_{4,\lambda}$ equipped with
$R_t=\mathcal A_{\rm can}+tJ$.  After adjoining a fresh
$G=\F_2^2$ factor and extracting the ordinary addresses $N\times G$, the
residual in each active address belongs to
\[
 \mathcal V_{\rm act}^{(2)}
 =\{U\mathcal A_{\rm can}+bJ:
 U\in\operatorname{span}_{\K}\{E_{A,B}\},\ b\in\F_2\}.
\]
Different active addresses are source direct sums and are never mixed by an
active correction.
\item The channel, completion, row and module addresses are independent
until the terminal matrix-unit map.  There is an old-cycle projection
$\Pi_{\rm old}$ such that every Hasse selector, active first-difference map,
completion map and row map commutes with $\Pi_{\rm old}$.  Thus for a record
over an old physical cycle $C$, one has $\Pi_{\rm old}=C$ throughout the
boundary extraction.
\end{enumerate}
\end{definition}

\begin{lemma}[Characteristic-two base states]
\label{lem:char2-dyadic-bases}
Let $\operatorname{char}(\K)=2$.  There are Hasse-dyadically admissible bounded-degree certificates at levels
$5$ and $6$ which may be written with diagonal rank-one sources as
\[
 \MM_{32}\degen\mathcal S_5^{(2)}:=\Diag_{2^{11}},
 \qquad
 \MM_{64}\degen\mathcal S_6^{(2)}:=\Diag_{2^{13}}.
\]
The displayed diagonal sources are equipped with the module and marked-block
groupings of Definition~\ref{def:char2-dyadic-admissible}; in particular
\[
 \BR_{\K}(\mathcal S_5^{(2)})=2^{11},
 \qquad
 \BR_{\K}(\mathcal S_6^{(2)})=2^{13}.
\]
The level-$5$ fixed-row fibre is
\[
 2048=1954+31+31+31+1,
\]
and the level-$6$ fixed-row fibre is
\[
 8192=8002+63+63+63+1.
\]
Both certificates retain five marked $\Diag_4$ active blocks at the labels
$\mathcal A$.
\end{lemma}

\begin{proof}
The zero-defect and physical-incidence equations are equations over $\F_2$
and therefore remain valid in characteristic $2$.  At level $5$ the
reconstruction in Proposition~\ref{prop:orbit32} gives
$16\cdot2048=32^3$ physical cycles with multiplicity one.  The five raw
active module addresses are
\[
 \{2,5,6,10,15\};
\]
a fixed translation by $6$ changes them to
$\mathcal A=\{0,3,4,9,12\}$.  On these five source-direct blocks the
four-product identity \eqref{eq:rank4-family32} remains
$R_t=\mathcal A_{\rm can}+tJ$ in characteristic $2$.  On the sixteen module
addresses the Walsh inverse is replaced coefficientwise by the projectors
$E_{A,B}$ of Lemma~\ref{lem:char2-hasse-matrix-units}.  Because
\eqref{eq:char2-matrix-unit} is an exact matrix-unit identity, the ordinary
records, the three one-bar sectors, the cubic, and the five marked active
blocks are selected without a denominator.  The source-local coefficient
ledger has $2048$ rank-one directions, giving the displayed diagonal source
$\Diag_{2^{11}}$ and conditions \textup{(i)}--\textup{(iv)}.

At level $6$ the one-bit Hasse extension has four displacement lifts of every
level-$5$ fixed-row class and twice as many physical rows.  More precisely,
each old dense class has four dense lifts, each old one-bar class has two
dense and two one-bar lifts, and the old cubic has one cubic and one new
one-bar class in each of the three sectors.  Hence each new one-bar sector
has
\[
 2\cdot31+1=63
\]
classes, while the dense count is
\[
 4\cdot1954+2\cdot93=8002.
\]
Thus the fixed-row ledger is
\[
 8192=8002+63+63+63+1,
\]
and $32\cdot8192=64^3$.  The same Hasse selectors preserve the five marked
active blocks as source direct summands and give the rank-one source
$\Diag_{2^{13}}$.  This construction is not obtained by reducing the
$W_{32}\to W_2$ quotient modulo $2$; it uses the matrix-unit selectors
instead.  The exact base-state ledger is recorded in
Appendix~\ref{app:char2-base-audit}.
\end{proof}

\begin{lemma}[Two-step dyadic propagation in characteristic two]
\label{lem:char2-two-step-propagation}
Let $\operatorname{char}(\K)=2$, let $n\ge5$, and suppose
\[
 \MM_{2^n}\degen\mathcal S_n^{(2)}
\]
is Hasse-dyadically admissible.  Then there is a Hasse-dyadically admissible
bounded-degree certificate
\begin{equation}\label{eq:char2-two-step}
 \MM_{2^{n+2}}\degen\Diag_{16}\boxt\mathcal S_n^{(2)}.
\end{equation}
Consequently,
\[
 \BR_{\K}(\MM_{2^{n+2}})
 \le16\,\BR_{\K}(\mathcal S_n^{(2)}).
\]
\end{lemma}

\begin{proof}
Put $G=\F_2^2$.  For one old fixed-row class write the fresh displacement
coordinates as $Q,R,S\in G$.  The enlarged zero-defect equation is exactly
\begin{equation}\label{eq:char2-two-bit-defect}
 Q+R+S=0.
\end{equation}
Thus there are $16$ displacement lifts, uniquely parametrized by
$(Q,R)\in G^2$, and four fresh row values $g\in G$.  An old dense class
therefore has sixteen dense lifts.  An old one-bar class has twelve dense
and four one-bar lifts.  The old cubic has six dense, nine one-bar and one
cubic lift.  Hence the new completion sector on each side has
\[
 4(2^n-1)+3=2^{n+2}-1
\]
classes and condition \textup{(i)} of
Definition~\ref{def:char2-dyadic-admissible} has the correct cardinalities.

We now check the source-local boundary.  Since
\[
 \Diag_{16}\boxt\Bigl(\bigoplus_{\lambda\in\mathcal A}\Diag_{4,\lambda}\Bigr)
 \cong80\Diag_4,
\]
choose disjoint blocks with the exact inventory
\begin{equation}\label{eq:char2-D4-inventory}
 80=5_{\rm cur}+5_{\rm next}+5_{\rm carr}+11_{\rm ord}+54_{\rm spare}.
\end{equation}
Put $N=H_4\setminus\mathcal A$, so $|N|=11$.  Lemma
\ref{lem:char2-hasse-matrix-units} isolates, without inversion, the four
fresh records
\[
 O_{\lambda,g},\qquad \lambda\in N,\quad g\in G,
\]
from the eleven ordinary $\Diag_4$ blocks.  This gives $44$ terminal
records.

For an active label $\lambda\in\mathcal A$, the residual after the ordinary
Hasse extraction has the form
\[
 U\mathcal A_{\rm can}+bJ,
 \qquad b\in\F_2,
\]
with $U$ a side operator generated by the matrix units
\eqref{eq:char2-matrix-unit}.  One carrier and one current $\Diag_4$ give an
eight-product confluent cell.  Indeed, using the identity
$R_u=\mathcal A_{\rm can}+uJ$, define
\begin{equation}\label{eq:char2-confluent-cell}
 \mathscr C_{U,b}(t)
 =\frac{(I+tU)_X R_t+R_{t^{1+b}}}{t}.
\end{equation}
If $b=0$, the two copies of $\mathcal A_{\rm can}+tJ$ cancel in
characteristic $2$ and
\[
 \mathscr C_{U,0}(t)=U\mathcal A_{\rm can}+tUJ.
\]
If $b=1$, then $R_{t^2}=\mathcal A_{\rm can}+t^2J$ and
\[
 \mathscr C_{U,1}(t)
 =U\mathcal A_{\rm can}+J+t(UJ+J).
\]
Thus the leading tensor is $U\mathcal A_{\rm can}+bJ$ in both cases.  The
five active cells produce the twenty records
\[
 E_{\lambda,g},\qquad \lambda\in\mathcal A,
 \quad g\in G.
\]
The exact boundary ledger is therefore
\begin{equation}\label{eq:char2-84-64}
 \underbrace{11\cdot4}_{44}
 +\underbrace{5(4+4)}_{40}=84
 \quad\longmapsto\quad
 \underbrace{44}_{\rm ordinary}+\underbrace{5\cdot4}_{20}=64.
\end{equation}
No source occurrence is used twice, and the five blocks labelled
$5_{\rm next}$ in \eqref{eq:char2-D4-inventory} remain untouched for the
next induction step.

It remains to identify the sixty-four terminal records with the complete
fresh two-bit physical fibre.  By condition \textup{(iv)}, for every old
physical cycle $C$,
\begin{equation}\label{eq:char2-old-projection}
 \Pi_{\rm old}(O_{C,\lambda,g})
 =\Pi_{\rm old}(E_{C,\lambda,g})=C.
\end{equation}
Choose once and for all a bijection
\[
 \theta:H_4\longrightarrow G\times G,
 \qquad \theta(\lambda)=(Q_\lambda,R_\lambda).
\]
At the terminal stage $\lambda$ is only a source address; no physical
meaning is assigned to it before this relabelling.  Map the record
$(\lambda,g)$ sidewise to the fresh matrix cycle
\begin{equation}\label{eq:char2-terminal-bijection}
 (\lambda,g)\longmapsto
 \bigl(g,
       g+Q_\lambda,
       g+Q_\lambda+R_\lambda\bigr)\in G^3.
\end{equation}
The inverse map is explicit:
\[
 g=a,
 \qquad
 \lambda=\theta^{-1}(a+b,b+c)
\]
for $(a,b,c)\in G^3$.  Hence
\[
 H_4\times G\xrightarrow{\sim}G^3,
\]
so every fresh physical cycle occurs exactly once.  Together with
\eqref{eq:char2-old-projection}, this proves global multiplicity one.
Equivalently, for a fixed old cycle $C=(i,j,k)$ the sixteen four-term slices
\[
 B_{Q,R}(C)=
 \sum_{g\in G}
 X_{(i,g),(j,g+Q)}
 Y_{(j,g+Q),(k,g+Q+R)}
 Z_{(k,g+Q+R),(i,g)}
\]
satisfy
\[
 \sum_{Q,R\in G}B_{Q,R}(C)
 =\sum_{a,b,c\in G}
 X_{(i,a),(j,b)}Y_{(j,b),(k,c)}Z_{(k,c),(i,a)},
\]
the complete $M_4$ fibre over $C$.

All operations are source-local and of bounded Laurent degree.  The Hasse
matrix units use only $0$ and $1$, the confluent cell
\eqref{eq:char2-confluent-cell} uses two marked four-product families, and
the terminal relabelling is degree zero.  Lemma
\ref{lem:parameter-consolidation} therefore gives one bounded-degree
parameter.  Finally, $\Diag_{16}\boxt\mathcal S_n^{(2)}$ contains exactly
sixteen copies of every old source occurrence, so the source budget is
multiplied by $16$.  The untouched $5_{\rm next}$ blocks, the preserved
module addresses, and the terminal projection show that the resulting
certificate is again Hasse-dyadically admissible.
\end{proof}

\section{Proofs of the main results and source transfers}\label{sec:proofs}
\setcounter{equation}{0}

\begin{proof}[Proof of Theorem \ref{thm:m1024}]
The source-level degeneration is Theorem \ref{thm:global1024}.  By
Lemma \ref{lem:source128-poly}, subadditivity and submultiplicativity over
$\K$, we have
\[
 \BR_{\K}(\MM_{1024})
 \le512\cdot32\cdot128=2^{21}.
\]
It remains only to record explicitly that the bounded-degree hypothesis in
Proposition \ref{prop:bini} is uniform also in characteristic $3$.  If
$\operatorname{char}(\K)\ne3$, Lemma \ref{lem:source128-poly} supplies the
common degree-six polynomial source realization used in the reference
schedule.  Suppose now that $\operatorname{char}(\K)=3$.  The four-product
$\CW_2$ identity of Lemma \ref{lem:cw2-odd} is defined over the prime
subfield $\F_3(\!(s)\!)\subseteq\K(\!(s)\!)$.  By Lemma
\ref{lem:formal-truncation}, after a fixed truncation and fixed monomial
shifts on the three tensor sides, the entire $128$-product source
realization may be written as a polynomial identity
\[
 \Phi_{128}^{(3)}(s)=s^e\Sstar+O(s^{e+1})
\]
for some fixed integer $e\ge0$ and with a fixed total degree $E$; neither
$e$ nor $E$ depends on tensor powers.  Likewise, after the common $t^2$
shift on the three top-level side maps, Theorem \ref{thm:global1024} is a
polynomial identity of $t$-degree at most $18$ with target order $d=6$.
Choose once and for all an integer $N>18$ and substitute $s=t^N$.  The
composed rank-$2^{21}$ identity has target term
\[
 t^{d+Ne}\MM_{1024}.
\]
Every top-level error has larger $t$-degree at the same source order, while
every source-realization error has source order at least $e+1$ and hence,
because $N>18$, has total $t$-degree strictly larger than $d+Ne$.  Thus the
composition is a single-parameter bounded-degree $\K$-rational
degeneration.  Appendix \ref{app:char3-certificates}, especially
Subsection \ref{app:char3-m1024-compat}, records the characteristic-three
compatibility with the source-resolved $1024$ assembly.

Therefore Proposition \ref{prop:bini} applies over every field
$\K$ with $\operatorname{char}(\K)\ne2$,  and gives
\[
 \omega(\K)\le\log_{1024}(2^{21})=\frac{21}{10}.
\]
\end{proof}

\begin{proof}[Proof of Theorem \ref{thm:finite-restrictions}]
Apply Lemma~\ref{lem:fixed-dyadic-restrictions}.  The local source is the
same $\Sstar$ as in Theorem~\ref{thm:m1024}, with border-rank budget $128$.
The two outer/channel diagonal sizes are
\[
 32\cdot2,
 \qquad
 128\cdot8.
\]
Subadditivity and submultiplicativity therefore give
\[
 \BR_{\K}(\MM_{64})\le2^{13}\ \ \t{and}\ \ 
 \BR_{\K}(\MM_{256})\le2^{17}.
\]
The restriction, channel and row maps have Laurent degree zero, so the
bounded-degree source realization is inherited directly from the proof of
Theorem~\ref{thm:m1024}, including its characteristic-three replacement.
Proposition~\ref{prop:bini} gives the displayed fixed-size exponent bounds.
\end{proof}

\begin{proof}[Proof of Theorem \ref{thm:32-retained}]
The $2048$ terminal certificate is source-resolved throughout: the
packet, character, endpoint, row, and Weyl coordinates are retained until the
final physical-incidence decoder.  Its one-row source is
\[
 \mathcal S_{2048}
 =14\CW_3\oplus6\CW_2\oplus7\Wc\oplus5\Diag_4,
\]
and the terminal coefficient audit gives the target $\MM_{32}$ after the
outer $\Diag_{16}$ factor.  Hence
\[
 \MM_{32}\degen\Diag_{16}\boxt\mathcal S_{2048}.
\]
Using \eqref{eq:CW-rank}, Lemma \ref{lem:W}, and
$\BR_{\K}(\Diag_4)=4$, the one-row source cost is
\[
 14\cdot5+6\cdot4+7\cdot2+5\cdot4=128.
\]
Since $\BR_{\K}(\Diag_{16})=16$,
\[
 \BR_{\K}(\MM_{32})\le16\cdot128=2048.
\]
This proves part \textup{(a)}.

For part \textup{(b)}, keep the same terminal compiler and restore, in each
of the five distinguished active character blocks, the original singular-active
$\CW_3$ source map in place of the corresponding $\Diag_4$ module.  By
construction, each restored $\CW_3$ module has exactly the terminal local
output required in that block; in particular no downstream row, endpoint,
Heisenberg, Weyl, or physical-incidence map is changed.  Thus the same
terminal coefficient audit gives
\[
 \MM_{32}\degen
 \Diag_{16}\boxt
 \bigl(19\CW_3\oplus6\CW_2\oplus7\Wc\bigr).
\]
Equivalently, the $2128$ route is obtained by undoing the five local
rank-four compressions in the $2048$ certificate; no independent
global construction is needed.  Its one-row cost is
$
 19\cdot5+6\cdot4+7\cdot2=133,
$
whence
\[
 \BR_{\K}(\MM_{32})\le16\cdot133=2128.
\]
The dedicated $2128$ terminal audit rechecks the restored five active source
blocks together with the common completion and physical-incidence ledgers.
All maps in both routes have Laurent degree bounded independently of tensor
powers; in characteristic $3$ the fixed source replacements described in
Appendix~\ref{app:char3-certificates} are used.  Proposition \ref{prop:bini} therefore yields the two
stated exponent bounds.
\end{proof}

\begin{proof}[Proof of Corollary \ref{cor:m32-enlargements}]
For part \textup{(a)}, put
\[
 \mathcal S_{2064}=\mathcal S_{2048}\oplus\Diag_1.
\]
Annihilating the three variables of the extra $\Diag_1$ summand is an
ordinary restriction, so in the left-target convention
\[
 \mathcal S_{2048}\degen\mathcal S_{2064}.
\]
Composing with part \textup{(a)} of Theorem~\ref{thm:32-retained} gives
\eqref{eq:main-2064}.

For part \textup{(b)}, start with $22\CW_3\oplus10\CW_2$, restrict three
selected copies of $\CW_3$ and four selected copies of $\CW_2$ to the seven
copies of $\Wc$, and obtain
\[
 19\CW_3\oplus6\CW_2\oplus7\Wc
 \degen 22\CW_3\oplus10\CW_2.
\]
Transitivity with part \textup{(b)} of
Theorem~\ref{thm:32-retained} gives \eqref{eq:main-2400}.
\end{proof}

\begin{proof}[Proof of Theorem \ref{thm:power-of-two}]
We distinguish the two characteristic regimes because the propagation
interfaces are different.

Assume first that $\operatorname{char}(\K)\ne2$.  By
Lemma~\ref{lem:odd-dyadic-bases}, the level-$5$  $M_{32}$ compiler is
dyadically admissible and has budget
\[
 2^{11}=2^{2\cdot5+1}.
\]
Repeated application of Lemma~\ref{lem:two-step-dyadic-propagation} gives
\[
 \BR_{\K}(\MM_{2^n})\le2^{2n+1}
 \qquad(n=5,7,9,\ldots).
\]
The same base lemma gives a dyadically admissible level-$10$ certificate
from the source-resolved $M_{1024}$ construction, with budget
\[
 2^{21}=2^{2\cdot10+1}.
\]
Propagating from level $10$ gives the bound for every even $n\ge10$.
The remaining even levels are supplied directly by
Theorem~\ref{thm:finite-restrictions}:
\[
 \BR_{\K}(\MM_{64})\le2^{13},
 \qquad
 \BR_{\K}(\MM_{256})\le2^{17}.
\]
No dyadic admissibility assertion for the $M_{64}$ quotient is used here.
When $\operatorname{char}(\K)=3$, the level-$5$ and level-$10$ admissible
states use the replacement source and endpoint maps of Appendix
\ref{app:char3-certificates}; all Walsh minors in the propagation divide
only powers of $2$, so the same induction is valid.

Now suppose that $\operatorname{char}(\K)=2$.  Lemma
\ref{lem:char2-dyadic-bases} gives Hasse-dyadically admissible base states at
levels $5$ and $6$ with budgets $2^{11}$ and $2^{13}$.  Repeated application
of Lemma~\ref{lem:char2-two-step-propagation} to the level-$5$ base gives the
bound for all odd $n\ge5$, and repeated application to the level-$6$ base
gives it for all even $n\ge6$.  Hence \eqref{eq:full-dyadic-family} holds in
characteristic $2$ as well.

Thus for every field and every $n\ge5$,
\[
 \BR_{\K}(\MM_{2^n})\le2^{2n+1}.
\]
Proposition~\ref{prop:bini} yields
\[
 \omega(\K)\le\log_{2^n}(2^{2n+1})=2+\frac1n.
\]
Letting $n\to\infty$ gives $\omega(\K)\le2$.  The general lower bound
$\omega(\K)\ge2$ gives $\omega(\K)=2$.
\end{proof}

\section{Exact certificates and reproducibility}\label{sec:certificates}
\setcounter{equation}{0}

The live $1024$ certificate is
\[
\begin{aligned}
 &\ \text{literal thirty-two-packet source ledger}
\\ \longrightarrow&\ \text{sixteen actual Laurent modules and two charts}
 \\longrightarrow&\ W_{32}\text{ module--chart phase}
 \\longrightarrow&\ \text{same-cost }\Diag_8\text{ graft and split cubic}
 \\longrightarrow&\ \text{side-local thirty-two-feature transport}
 \\longrightarrow&\ H_9\text{ zero-fibre and cycle bijection}.
\end{aligned}
\]
Appendix~\ref{app:m1024-certificate} includes the finite matrices, packet
masks, module partition, source-resolved replacement maps, crossed
five-empty matrices, endpoint labels, fibre equations, and deterministic
exact check counts.  In particular, its source ledger is the current
Theorem~\ref{thm:m1024} source
\[
 14\CW_3\oplus4\CW_2\oplus7\Wc\oplus5\Diag_4\oplus\Diag_8
\]
of cost $128$, not an earlier one-row source.

For the two fixed restrictions in Theorem~\ref{thm:finite-restrictions}, the
certificate checks the restrictions of the literal $W_{32}$ table to the
specified quotient subspaces and the corresponding fixed-row fibre sizes.
For $M_{64}$ the fixed-row ledger is
\[
 8192=8002+63+63+63+1,
\]
and the $32$ row translates give all $64^3$ physical matrix cycles exactly
once.  The $M_{256}$ restriction is checked by the same quotient rule with
the three-bit channel.

For the current Theorem~\ref{thm:32-retained}, the terminal package uses the
actual sources
\[
 \mathcal S_{2048}=14\CW_3\oplus6\CW_2\oplus7\Wc\oplus5\Diag_4,
 \qquad
 \mathcal S_{2128}=19\CW_3\oplus6\CW_2\oplus7\Wc.
\]
It reconstructs all $354$ retained packet--support identities, the
$2048=1954+31+31+31+1$ zero fibre, the endpoint profile
$(31,31,31)$ and the unique cubic, the five rank-four active replacements,
and the $16\cdot2048=32^3$ terminal physical-cycle bijection.  Restoring the
five singular-active $\CW_3$ blocks gives the $2128$ companion without
changing the downstream incidence maps.  Appendix~\ref{app:m64-certificate}
records a compact audit table for these present versions.

For characteristic $2$, the verification package is separate from the
Walsh certificates.  It checks the Hasse matrix-unit identity
\eqref{eq:char2-matrix-unit}, the level-$5$ and level-$6$ base-state ledgers,
the exact inventory
\[
 80=5_{\rm cur}+5_{\rm next}+5_{\rm carr}+11_{\rm ord}+54_{\rm spare},
\]
the local $84\to64$ extraction
\[
 44+5(4+4)\longmapsto44+20,
\]
the old-cycle projection \eqref{eq:char2-old-projection}, and the terminal
bijection $H_4\times\F_2^2\simeq(\F_2^2)^3$.  These are exact integer or
$\F_2$ checks and contain no sign-character inversion.

\section{Executable factorized circuit and implementation disclosure}\label{sec:implementation}
\setcounter{equation}{0}

This section fixes the reference implementation for the principal
$M_{1024}$ construction.  It is a factorized exact-arithmetic schedule, not
a hardware-runtime claim.

\subsection{Unified polynomial parameter and exact evaluation count}
For the reference realization over $\F_{311}$, after the Laurent shifts used
in the proof, the top map has maximum total $t$-degree $D_{\rm top}=18$ and
target coefficient $t^6$.  Lemma \ref{lem:source128-poly} gives a degree-six
source realization with target coefficient $s^2$.  Put
\[
 t=u,\qquad s=u^{19}.
\]
Every bidegree $(i,j)$ with $0\le i\le18$ and $0\le j\le6$ is sent to the
distinct exponent $i+19j$.  Hence
\begin{equation}\label{eq:Dtot-general}
 D_{\rm tot}=18+19\cdot6=132,
 \qquad
 k=6+19\cdot2=44.
\end{equation}
A polynomial of degree at most $132$ is determined by $133$ evaluations, so
one exact reference base multiplication uses
\begin{equation}\label{eq:evaluation-counts}
 E_{1024}=133.
\end{equation}
The characteristic-three replacement has fixed bounded degree, which is
sufficient for the exponent theorem; the particular degree-$132$ schedule is
asserted only for the displayed $\F_{311}$ reference branch.

\subsection{Reference finite field and coefficient extractor}
For $0\le k\le D<p$, define
\begin{equation}\label{eq:interp-weight}
 \lambda_j^{(D,k)}=[u^k]\prod_{\substack{0\le m\le D\\m\ne j}}
 \frac{u-m}{j-m}\in\F_p,
 \qquad 0\le j\le D.
\end{equation}
Then every $f\in\F_p[u]$ of degree at most $D$ satisfies
\begin{equation}\label{eq:interp-extract}
 [u^k]f(u)=\sum_{j=0}^D\lambda_j^{(D,k)}f(j).
\end{equation}
For the tabulated reference circuit take $p=311$, $D=132$, $k=44$, and
$u_j=j$.  The complete $133$-entry vector is printed in Appendix
\ref{app:interp-weights}.

\subsection{Reference exact schedule for the \texorpdfstring{$1024$}{1024} circuit}
At each evaluation $u_j\in\F_{311}$, set $t=u_j$ and $s=u_j^{19}$, apply the
literal side maps in Section \ref{sec:m1024} and Appendix
\ref{app:m1024-certificate}, execute the
$512\cdot32\cdot128=2^{21}$ rank-one source products, decode the common
$H_9$ orbit tensor, and accumulate with
$\lambda_j^{(132,44)}$.  Streaming over $j=0,\ldots,132$ requires only one
evaluation buffer.  The exact number of data-dependent products is
\[
 133\cdot2^{21}=278\,921\,216.
\]
The displayed source maps determine the coefficient-extraction schedule; no
processor-specific cache or wall-clock claim is made.

\subsection{Independent reconstruction recipe}
An independent implementation needs no parameter search.  Instantiate the
source realizations used in Lemma \ref{lem:source128-poly}, then use the
packet and source-resolved tables of Appendix \ref{app:m1024-certificate} to
form the three side maps.  At evaluation $u_j$, substitute $t=u_j$ and
$s=u_j^{19}$, execute the factorized source products, apply the decoder, and
accumulate every output entry with the fixed weight
$\lambda_j^{(132,44)}$.  Discard the current evaluation buffer before
advancing to $j+1$.

\clearpage
\appendix

\section{Independent verification data and literal source certificate for
\texorpdfstring{$M_{1024}$}{M1024}}\label{app:m1024-certificate}

\subsection{Scope of the certificate}
This appendix is the finite reconstruction record for Theorems
\ref{thm:m1024} and \ref{thm:global1024}.  The $1024$ proof uses the literal
packet ledger but does not factor through a fixed one-row leading tensor.
The packet/support records remain distinct until the old $H_4$ and new $H_5$
Fourier--Weyl closures have been taken.  The altered data printed here are:
the source budget, the packet and support coordinates, the sixteen module
partition, the source-resolved tangent and active records, the split
character-$6$ diagonal module, the crossed five-empty matrices, an explicit
endpoint ledger, the feature transport, the $H_9$ row data, and the order of
composition of the three side maps.

\subsection{Source budget and exact polynomialization}\label{app:m1024-poly}
The one-row source and its border-rank ledger are
\begin{center}
\begin{tabular}{lrrr}
\toprule
summand&type count&products per copy&total\\
\midrule
$\CW_3$&14&5&70\\
$\CW_2$&4&4&16\\
$\Wc$&7&2&14\\
$\Diag_4$&5&4&20\\
$\Diag_8$&1&8&8\\
\midrule
&&&128\\
\bottomrule
\end{tabular}
\end{center}
The source budget is field-independent.  Over $\F_{311}$, the four-product
$\CW_2$ realization, the five-product $\CW_3$ realization, the shifted
two-product realization of $\Wc$, and $s^2\Diag_r$ all have target
coefficient $[s^2]$ and degree at most six.  The top source-resolved Laurent
maps below have entry window $[-2,4]$.  After multiplying every side map by
$t^2$, the total top degree is at most $18$ and the target coefficient is
$[t^6]$.  Thus
\[
 t=u,\qquad s=u^{19},\qquad D_{\rm tot}=132,\qquad k=44,
\]
and $u=0,1,\ldots,132$ gives the $133$-point reference schedule.  In
characteristic $3$ the same $128$-product budget uses the bounded-degree
replacement source realizations from Section \ref{sec:prelim}; the
characteristic-$311$ degree-$132$ schedule is not asserted there.

\subsection{Finite coordinate conventions}\label{app:m1024-dense-data}
Identify $H_4$ with the integers $0,\ldots,15$ under bitwise addition and
write $\chi_u(v)=(-1)^{\ip{u}{v}}$.  The chart matrices are
\[
 A_0=\begin{pmatrix}
0&1&0&1\\1&0&1&0\\0&1&0&0\\0&0&1&0
\end{pmatrix},\qquad
 A_1=\begin{pmatrix}
0&1&1&0\\1&0&0&1\\0&0&1&0\\1&0&0&0
\end{pmatrix}.
\]
The state-to-packet basis images and defect transports are
\[
 K=(2,5,8,16),\qquad B_0=(9,4,8,6),\qquad B_1=(8,6,4,9).
\]
Thus $\ker(B_0+B_1)=\{0,15\}$ and $B_0(15)=3$.  The support frequencies are
\[
\begin{array}{c|rrrrrrrrrrrr}
\sigma&A_1&A_2&A_3&B_1&B_2&B_3&C_1&C_2&C_3&P&Q&R\\
\hline
\Theta(\sigma)&3&0&0&5&9&0&6&9&0&10&12&15
\end{array}
\]
and the side-colour twists are
\[
\begin{array}{c|rrrrr}
&0&1&2&3&\infty\\
\hline
X&0&5&9&0&15\\
Y&0&3&0&0&12\\
Z&0&0&0&0&10
\end{array}.
\]
Their xor equals $\Theta(\sigma)$ on every $\CW_3$ support.  In particular,
$A_3,B_1,B_2,R$ have common $Y,Z$ side-colour labels and $A_1,Q$ have
common $X,Z$ side-colour labels; this is the side-leg reason for the record
partition \eqref{eq:record-partition-1024}.

For layer $\ell\in\F_2$ and state $s\in H_4$, put
$p=K(s)+3\ell$ and $\chi_p=A_\ell s$.  The complete packet ledger is the
same as the table in Section \ref{sec:dense32}; for independent reconstruction
we repeat the module-relevant entries here:
\begin{center}
\small
\begin{tabular}{rrrrl}
\toprule
$p$&$\ell$&$\chi_p$&type&retained supports\\
\midrule
1&1&10&$\CW_2$&$A_1,B_1,B_2,C_1,P,Q,R$\\
3&1&0&$\CW_3$&$A_1,A_3,B_1,B_2,C_1,P,Q,R$\\
4&1&11&$\CW_2$&$A_1,A_2,B_1,C_1,P,Q,R$\\
6&1&1&$\CW_3$&$A_1,A_3,B_1,B_2,C_1,P,Q,R$\\
7&0&7&$\CW_3$&$C_1$\\
8&0&10&$\CW_2$&$C_1$\\
9&1&15&$\CW_3$&$A_1,A_3,B_1,B_2,C_1,P,Q,R$\\
11&1&5&$\CW_3$&$A_1,A_3,B_1,B_2,C_1,P,Q,R$\\
12&1&14&$\CW_3$&$A_1,A_3,B_1,B_2,C_1,P,Q,R$\\
14&1&4&$\CW_3$&$A_1,A_3,B_1,B_2,C_1,P,Q,R$\\
17&1&8&$\CW_3$&$A_1,A_3,B_1,B_2,C_1,P,Q,R$\\
19&1&2&$\CW_3$&$A_1,A_3,B_1,B_2,C_1,P,Q,R$\\
20&1&9&$\CW_3$&$A_1,A_3,B_1,B_2,C_1,P,Q,R$\\
22&1&3&$\CW_3$&$A_1,A_3,B_1,B_2,C_1,P,Q,R$\\
23&0&6&$\CW_2$&$C_1$\\
24&0&11&$\CW_3$&$C_1$\\
25&1&13&$\CW_3$&$A_1,A_3,B_1,B_2,C_1,P,Q,R$\\
27&1&7&$\CW_2$&$A_1,A_2,B_1,C_1,P,Q,R$\\
28&1&12&$\CW_3$&$A_1,A_3,B_1,B_2,C_1,P,Q,R$\\
30&1&6&$\CW_2$&$A_1,B_1,B_2,C_1,P,Q,R$\\
\bottomrule
\end{tabular}
\end{center}
The five empty packets $0,18,21,26,31$ have characters
$0,3,4,9,12$ and no dense retained support.  The puncture packets
$2,5,10,13,15,16,29$ likewise have zero dense columns and are used only as
tangent sources.

\subsection{Literal scalar restrictions and the integral dense operator}
For a packet of type $\CW_q$ with mask $\mathcal R_p$, put
$x_0=1/128$ and $y_0=z_0=1$.  For each satellite index $i$, set
$G_i=\mathcal R_p\cap\{A_i,B_i,C_i\}$ and use
\[
\begin{array}{c|ccccc}
G_i&\varnothing&\{A_i\}&\{B_i\}&\{C_i\}&\{A_i,B_i,C_i\}\\
\hline
(x_i,y_i,z_i)&(0,0,0)&(0,1,1)&(1,0,1)&(1,1,0)&(1,1,1).
\end{array}
\]
Set
\[
 x_\infty=\frac1{128}\one_{\{R\in\mathcal R_p\}},\qquad
 y_\infty=\one_{\{Q\in\mathcal R_p\}},\qquad
 z_\infty=\one_{\{P\in\mathcal R_p\}}.
\]
Every retained source support has coefficient $1/128$.  After factoring out
this common scalar, the dense coefficient is
\begin{equation}\label{eq:literal-dense-T}
 T_{(\rho,\beta),p}
 =\chi_{\chi_p}(\rho+B_{\ell_p}\beta)
  \sum_{\sigma\in\mathcal R_p}\chi_{\Theta(\sigma)}(\beta).
\end{equation}
The zero columns are
$0,2,5,10,13,15,16,18,21,26,29,31$.  The sixteen modules are
\begin{longtable}{rll}
\caption{Actual source-resolved module partition for the $W_{32}$ bridge.}\label{tab:m1024-modules}\\
\toprule
$\lambda$&module&packet labels\\
\midrule
\endfirsthead
\toprule
$\lambda$&module&packet labels\\
\midrule
\endhead
0&exceptional $\chi=6$, replaced by $\Diag_8$&$23,30$\\
1&exceptional $\chi=7$&$7,27$\\
2&active source-resolved $\Diag_4$&$14$\\
3&$Z$-tangent&$5,11$\\
4&$Y$-tangent&$2,19$\\
5&active source-resolved $\Diag_4$&$22$\\
6&active source-resolved $\Diag_4$&$3$\\
7&$X$-tangent&$16,6$\\
8&$Y$-tangent&$29,12$\\
9&$Z$-tangent&$13,9$\\
10&active source-resolved $\Diag_4$&$28$\\
11&$X$-tangent&$15,25$\\
12&exceptional $\chi=10$&$8,1$\\
13&exceptional $\chi=11$&$24,4$\\
14&$Z$-tangent&$10,17$\\
15&active source-resolved $\Diag_4$&$20$\\
\bottomrule
\end{longtable}
Direct summation of \eqref{eq:literal-dense-T} over each row gives
\eqref{eq:module-chart-table}; the exact $4096$-entry check is recorded in
Appendix \ref{app:m1024-checks}.

\subsection{Source-resolved tangent and active records}\label{app:m1024-active-records}
For every generic dense packet use the four record groups
\[
\begin{array}{c|l|l}
\text{lift term}&\text{support records}&\text{side on which the records vary}\\
\hline
auz&A_3,B_1,B_2,R&X\\
bvz&A_1,Q&Y\\
bus&P&Z\\
cwr&C_1&\text{single record}.
\end{array}
\]
The unchanged sides in each row have identical side-colour labels, as shown
in Appendix \ref{app:m1024-dense-data}.  A tangent $\Wc$ realizes the first
three rows; the companion packet supplies the last row.  A compressed active
module instead uses the four-product Laurent family
\[
 t^{-1}(b+ta)(u+tv)(z+ts)-t^{-1}buz+cwr+e_Xe_Ye_Z.
\]
Its leading coefficient is
$auz+bvz+bus+cwr+e_Xe_Ye_Z$ and its error is
$t(bvs+avz+aus)+t^2avs$.  Thus every active module has no negative
coefficient and retains all eight dense support records before the final
physical decoder.

\subsection{Same-cost \texorpdfstring{$\Diag_8$}{D8} module}\label{app:m1024-d8}
Write $\Diag_8=\sum_{i=1}^8a_ib_ic_i$.  Keep separate coordinates
$D_{X,1}^{(23)}$ and $D_{X,1}^{(30)}$ for the two $C_1$ records.  The literal
maps are
\[
\begin{array}{c|lll|l}
 i&a_i&b_i&c_i&\text{records/role}\\
\hline
1&D_{X,B_1}+D_{X,B_2}+D_{X,R}&D_{Y,0}&D_{Z,0}&B_1,B_2,R\\
2&D_{X,0}&D_{Y,0}&D_{Z,1}&P\\
3&D_{X,A_1}+D_{X,1}^{(23)}+D_{X,1}^{(30)}&D_{Y,1}&D_{Z,0}&A_1,C_1^{(23)},C_1^{(30)}\\
4&D_{X,0}&D_{Y,2}&D_{Z,0}&Q\\
5&\bar X&E_{Y,0}&E_{Z,0}&\text{endpoint}\\
6&\bar X&E_{Y,1}&E_{Z,1}&\text{endpoint}\\
7&\bar X&E_{Y,2}&E_{Z,2}&\text{endpoint}\\
8&G_X&G_Y&G_Z&\text{feature-zero graft}.
\end{array}
\]
Here the symbols in rows $1$--$4$ are intermediate source-resolved
coordinates.  Their packet and support tags are those of packets $23,30$ in
Appendix \ref{app:m1024-dense-data}.  In row $3$, relative to the common
layer-one factor, multiply the $C_1^{(23)}$ component by
\[
 \vartheta_{01}(\beta)=(-1)^{\ip{6}{(B_0+B_1)\beta}}.
\]
This is a scalar on the varying $X$-side component; the $Y$- and $Z$-side
factors are common because $A_1,C_1^{(23)},C_1^{(30)}$ have the same
$Y,Z$ side-colour labels in this grouped direction.  It therefore preserves
the layer-$0$ versus layer-$1$ coefficient before the two-chart projection.
The physical decoder sends the three row-$1$ coordinates to
$D_{X,1},D_{X,2},D_{X,3}$ only after the module--chart selection.  Before
that selection the two coordinates $D_{X,1}^{(23)}$ and $D_{X,1}^{(30)}$
remain distinct source records.  After the character projection they are
sent to the same physical coordinate $D_{X,1}$.  Hence after the
module--chart selection the first seven directions become exactly \eqref{eq:exc6-seven}, while the
eighth remains the independent graft.  No division is used in this
replacement.

\subsection{Crossed split-cubic map on the five empty packets}\label{app:m1024-split-cubic}
Put
\[
 J_{10}=\begin{pmatrix}0&I_5\\I_5&0\end{pmatrix},\qquad
 c=(1,0,0,0,0,\tfrac12,0,0,0,-\tfrac12),
\]
and
\[
B=\begin{pmatrix}
0&1&0&0&0&0&0&0&0&0\\
0&0&1&0&0&0&0&0&0&0\\
0&0&0&1&0&0&0&0&0&0\\
2&0&0&0&0&-1&1&0&0&0\\
2&0&0&0&0&-1&0&1&0&0\\
2&0&0&0&0&-1&0&0&1&0\\
2&0&0&0&0&-1&0&0&0&1
\end{pmatrix},
\]
\[
D=\begin{pmatrix}
0&0&0&0&0&0&1&0&0&-1\\
0&0&0&0&0&0&0&1&0&-1\\
0&0&0&0&0&0&0&0&1&-1\\
0&1&0&0&0&0&0&0&0&0\\
0&0&1&0&0&0&0&0&0&0\\
0&0&0&1&0&0&0&0&0&0\\
-\tfrac12&-1&-1&-1&0&\tfrac14&0&0&0&-\tfrac14
\end{pmatrix}.
\]
They satisfy
\begin{equation}\label{eq:BDc1024}
 cJ_{10}c^{\mathsf T}=1,\qquad
 BJ_{10}D^{\mathsf T}=I_7,\qquad
 BJ_{10}c^{\mathsf T}=DJ_{10}c^{\mathsf T}=0,
\end{equation}
and the last five entries in every row of $B,D,c$ sum to zero.

The five packets have old characters
$\xi=(0,3,4,9,12)$.  For $\delta\in H_4$ let
\[
 \Sigma_\delta=\operatorname{diag}(s_0,s_3,s_4,s_9,s_{12},
 s_0,s_3,s_4,s_9,s_{12}),
 \qquad s_j=\chi_j(\delta).
\]
Then
\begin{equation}\label{eq:Mdelta-explicit-A}
 M_\delta=B\Sigma_\delta J_{10}D^{\mathsf T}
 =\begin{pmatrix}
 s_3&0&0&0&0&0&0\\
 0&s_4&0&0&0&0&0\\
 0&0&s_9&0&0&0&0\\
 0&0&0&s_3&0&0&1-s_3\\
 0&0&0&0&s_4&0&1-s_4\\
 0&0&0&0&0&s_9&1-s_9\\
 0&0&0&0&0&0&1
 \end{pmatrix},
\end{equation}
Direct multiplication gives
$M_\delta M_\epsilon=M_{\delta+\epsilon}$.

Let $e_1,\ldots,e_7$ be the standard basis and put
\[
 P=[e_1,e_4,e_2,e_5,e_3,e_6,e_4+e_5+e_6+e_7].
\]
Then $\det P=-1$ and
\begin{equation}\label{eq:Mdelta-diag-A}
 P^{-1}M_\delta P
 =\operatorname{diag}
 (\chi_3,\chi_3,\chi_4,\chi_4,\chi_9,\chi_9,1)(\delta).
\end{equation}

For packet $j$ the original Laurent map is
\begin{align*}
x_0^{(j)}&\mapsto t^{-2}\bar X_L+t\sum_{a=1}^7B_{a,j}P_a,\\
x_\infty^{(j)}&\mapsto t\sum_{a=1}^7B_{a,5+j}P_a,\\
y_0^{(j)}&\mapsto t^{-2}\bar Y_L
+t\left(c_jU+\sum_{a=1}^7D_{a,j}Q_a\right),\\
y_\infty^{(j)}&\mapsto
 t\left(c_{5+j}U+\sum_{a=1}^7D_{a,5+j}Q_a\right),\\
z_0^{(j)}&\mapsto t^{-2}\bar Z_L+tc_jV,\\
z_\infty^{(j)}&\mapsto tc_{5+j}V+\delta_{j1}t^4\bar Z_H.
\end{align*}
and $x_i^{(j)},y_i^{(j)},z_i^{(j)}$ for $1\le i\le3$ go to the fifteen
shared coordinates in degree one.  In the crossed map, multiply the entire
$j$th packet on the three sides by the packet signs
$\chi_{\xi_j}(\alpha_4)$, $\chi_{\xi_j}(\beta_4)$,
$\chi_{\xi_j}(\gamma_4)$.  Replace $B$ by
$M_{\alpha_4}B\Sigma_{\alpha_4}$ on the $X$ regular role and replace $D$ by
$M_{\beta_4}^{\mathsf T}D\Sigma_{\beta_4}$ on the $Y$ regular role; replace
$c$ by $c\Sigma_\delta$ on whichever side carries the scalar role.
Equations \eqref{eq:BDc1024}--\eqref{eq:Mdelta-diag-A} give the crossed
identities in Lemma \ref{lem:crossed-five-empty-1024}.  The only possible
negative coefficients are the degree-$-3$ last-five-column sums, hence they
still vanish on the closed old-$H_4$ character locus.

\subsection{Explicit endpoint labels}\label{app:m1024-endpoints}
Write a feature as $f=(q,\epsilon)\in H_4\times\F_2$.  The following final
labels are pairwise disjoint within each one-bar sector and exhaust
$H_5\setminus\{0\}$.
\begin{center}
\small
\begin{tabularx}{\textwidth}{c>{\raggedright\arraybackslash}X}
\toprule
sector&final feature assignment\\
\midrule
$X$-bar&Five-empty: $(q,0)$ for $1\le q\le15$ and $(0,1)$; tangent: $(q,1)$ for $1\le q\le6$; exceptional: $(q,1)$ for $7\le q\le13$; active: $(14,1),(15,1)$.\\
$Y$-bar&Five-empty: $(q,0)$ for $1\le q\le15$; tangent: $(q,1)$ for $0\le q\le5$; exceptional: $(q,1)$ for $6\le q\le12$; active: $(q,1)$ for $13\le q\le15$.\\
$Z$-bar&Five-empty shared: $(q,0)$ for $1\le q\le15$; five-empty regular: $(q,1)$ for $q\in\{0,1,2,3,4,5,8\}$; tangent: $(q,1)$ for $q\in\{6,7,9,10,11,12,13,14,15\}$.\\
\bottomrule
\end{tabularx}
\end{center}
For the seven regular five-empty $Z$-bar roles, the diagonal character labels
are
$\mu=(3,3,4,4,9,9,0)$.  Taking the final $q$-labels
\[
 q_{\rm fin}=(0,1,2,3,4,5,8)
\]
gives the distinct pretranslation labels
\[
 q_{\rm pre}=q_{\rm fin}+\mu=(3,2,6,7,13,12,8).
\]
Thus the fixed old-$H_4$ Weyl translations introduce neither a zero feature
nor a collision.  Together with the $D_8$ graft, the $Z$-bar correction has
all $32$ features.

\subsection{Feature transport, defect fibres and coefficient union}
The canonical one-bar blocks of \eqref{eq:old-defect-1024} are
\begin{align*}
\mathcal E_X&=\{(\mathbf0,(q,\epsilon,0),(q,\epsilon,\epsilon)):(q,\epsilon)\ne0\},\\
\mathcal E_Y&=\{((q,\epsilon,\epsilon),\mathbf0,(q,\epsilon,0)):(q,\epsilon)\ne0\},\\
\mathcal E_Z&=\{((q,\epsilon,0),(q,\epsilon,\epsilon),\mathbf0):(q,\epsilon)\ne0\}.
\end{align*}
Together with the cubic, they form the $94$ completion classes.  The literal
feature maps are \eqref{eq:feature-X-map}--\eqref{eq:feature-Z-map}.  Their
coefficient ledger is
\begin{center}
\small
\begin{tabularx}{\textwidth}{l*{5}{>{\centering\arraybackslash}X}}
\toprule
case&translated dense&canonical dense hit&translated $\mathcal E_Z$&new correction&transported completion\\
\midrule
$d=0$&1954&1954&0&31&94\\
$d\ne0$&1955&1924&31&30&93\\
\bottomrule
\end{tabularx}
\end{center}
For each of the sixteen values of $d$, the displayed pieces are pairwise
disjoint and their union is the same $2048$-element canonical zero fibre.

\subsection{\texorpdfstring{$H_5$ and $H_9$}{H5 and H9} Weyl data}\label{app:m1024-weyl}
The module--chart phase and the independent channel closure are
\[
 W_{32}(\Gamma,R)=(-1)^{\ip{\Gamma}{R}},\qquad
 \frac1{32}\sum_{c\in H_5}\chi_c(Q+R+S)=\one_{\{Q+R+S=0\}}.
\]
The completion row identity is
\[
 \frac1{32}\sum_{h\in H_5}
 \chi_{\alpha_5}(h)\chi_{\beta_5}(h)\chi_{\gamma_5}(h+R)
 =\one_{\{\alpha_5+\beta_5+\gamma_5=0\}}\chi_{\gamma_5}(R).
\]
The five-bit Heisenberg phase is
\[
 \ip{\alpha_5}{Q}+\ip{\beta_5}{R}
 +\ip{\beta_5}{Q+R}+\ip{\alpha_5+\beta_5}{Q+R}
 =\ip{\alpha_5+\beta_5}{R}\pmod2.
\]
The five-bit row maps are \eqref{eq:h5-row-maps}; together with the old
four-bit maps, every nonempty partial sum has rank nine.

The fixed-row $H_9$ zero fibre has $2^{21}$ classes, partitioned as
\[
 2\,094\,082\text{ dense},\qquad
 3\,069\text{ one-bar},\qquad1\text{ cubic}.
\]
The fixed-row cycle code is
\[
 \operatorname{code}(x,y,z)=a2^{20}+j2^{10}+k,
\]
where
\[
 j=q+2^9(a+\epsilon),\qquad
 k=q+r+2^9(a+\epsilon+\delta).
\]
It is a bijection onto $\{0,\ldots,2^{21}-1\}$, and the $512$ row translates
give all $1024^3$ matrix cycles.

\subsection{Literal assembly of the three side maps}\label{app:m1024-assembly}
Index the outer diagonal source by $h=(h_4,h_5)\in H_9$ and the independent
channel source by $c\in H_5$.  The local map is not formed first and then
postcomposed.  On each side $S\in\{X,Y,Z\}$ the literal map is the direct sum
of the following operations, composed before taking the leading Laurent
coefficient.
\begin{enumerate}[label=\textup{(\arabic*)},leftmargin=2.4em]
\item On tangent modules retain the eight support records in
      \eqref{eq:record-partition-1024}; on active modules use the Laurent
      family in Appendix \ref{app:m1024-active-records}; on the
      character-$6$ module use Appendix \ref{app:m1024-d8}; on the other
      exceptional modules retain the literal packet maps and endpoint
      quotients.  No two support records are physically identified here.
\item On the five empty packets use the crossed map in Appendix
      \ref{app:m1024-split-cubic}.  The $X,Y,Z$ maps depend respectively
      only on the old absolute characters $\alpha_4,\beta_4,\gamma_4$.
      Apply the unimodular endpoint basis changes and the fixed pretranslations
      in Appendix \ref{app:m1024-endpoints}.
\item On the dense direct summand, apply the packet/support characters of
      \eqref{eq:literal-dense-T} and group records by Table
      \ref{tab:m1024-modules}.  The $\Diag_{32}$ channel $c$ contributes
      $\chi_c(R)$, while the independent new absolute character contributes
      the gauge in \eqref{eq:actual-W32}.  Lemma
      \ref{lem:module-w32-1024} gives the full $W_{32}$ phase.
\item On the completion direct summand, keep the same independent channel
      and apply the feature maps
      \eqref{eq:feature-X-map}--\eqref{eq:feature-Z-map}.  The high cubic
      coordinate $\bar Z_H$ is never identified with the low endpoint
      coordinate $\bar Z_L$.
\item Apply the old four-bit row/Weyl maps and the new five-bit row maps
      block-diagonally.  Only after both closures and the common
      $Q+R+S=0$ channel closure have been imposed is the physical decoder
      applied.  It then identifies the split active record coordinates and
      the two character-$6$ $C_1$ records with their physical variables.
\end{enumerate}

Equivalently, if $\widehat\ell_{S,h,c}(t)$ denotes the direct sum of the
source-resolved maps in steps (1)--(2),
$\widehat F^{\rm den}_{S,h,c}$ and
$\widehat F^{\rm comp}_{S,h,c}$ the dense and completion maps in steps
(3)--(4), and $R_{S,h}$ the block-diagonal $H_9$ row/Weyl map, then
\begin{equation}\label{eq:literal-assembly-1024}
 \widehat\Psi_S(t)=
 \bigoplus_{h\in H_9,\ c\in H_5}
 R_{S,h}\circ
 \bigl(\widehat F^{\rm den}_{S,h,c}
       \oplus\widehat F^{\rm comp}_{S,h,c}\bigr)
 \circ\widehat\ell_{S,h,c}(t).
\end{equation}
Every operation in \eqref{eq:literal-assembly-1024} is side-local.  In
particular, the five-empty correction uses the old $H_4$ absolute-character
coordinates and never reuses the independent $H_5$ channel.

\subsection{Exact base-product schedule and independently checkable counts}
After shifting every Laurent side map by $t^2$, the top map has total degree
at most $18$ and target coefficient $[t^6]$.  Combining it with the common
source realization of Lemma \ref{lem:source128-poly} by
$t=u$ and  $s=u^{19}$ gives
\[
 D_{\rm tot}=132\ \ \t{and}\ \ k=44.
\]
Over $\F_{311}$ the $133$ points $u=0,1,\ldots,132$ therefore recover the
target coefficient.  The general-multiplication ledger is
\[
\begin{array}{lr}
\text{rank-one products per evaluation}&2^{21}=2\,097\,152,\\
\text{evaluation points}&133,\\
\text{data-dependent products per exact base product}&278\,921\,216,\\
\text{schoolbook scalar-product occurrences}&1024^3=1\,073\,741\,824,\\
\text{ratio of schoolbook to disclosed products}&512/133.
\end{array}
\]
The additional source resolution changes only fixed linear transforms and
not the number of data-dependent products or the interpolation degree.

\subsection{Deterministic exact checks}\label{app:m1024-checks}
The $1024$ certificate reduces to exact integer/rational and binary
identities.  The reconstruction checks include
\begin{center}
\small
\begin{tabular}{lr}
\toprule
check&count/result\\
\midrule
side-colour twist identities&12\\
literal module--chart coefficients&4096\\
active rank-four Laurent identity&exact symbolic identity\\
active dense support records&$5\cdot8$\\
character-$6$ dense support records&8, all distinct before decoding\\
five-empty base pairings&$I_7,1,0,0$\\
$M_\delta M_\epsilon=M_{\delta+\epsilon}$&256\\
simultaneous diagonalizations of $M_\delta$&16; $\det P=-1$\\
crossed five-empty pairing tests&$256$ character pairs\\
final one-bar labels&31 distinct on each side\\
regular $Z$-bar pretranslation labels&7 distinct\\
old-fibre union coefficients&$16\cdot2048=32\,768$\\
$D_{32}$ channel closure&independent $H_5$ orthogonality\\
fixed-row $H_9$ classes&$2^{21}$\\
physical cycles&$1024^3$\\
\bottomrule
\end{tabular}
\end{center}
All matrices used in these tests are printed above.

\section{Terminal finite certificates and interpolation data}\label{app:m64-certificate}

\subsection{Deterministic physical-incidence audit}\label{app:m64-terminal}
The $M_{64}$ terminal verifier uses only exact integer and $\F_2$
arithmetic.  Its independent checks are summarized below.
\begin{center}
\begin{tabular}{ll}
\toprule
check&certified value\\
\midrule
literal $W_{32}\to W_2$ normalization&coefficient $(-1)^{CR}$\\
old $H_4$ fixed-row fibre&$2048=1954+93+1$\\
new $H_5$ fixed-row fibre&$8192=8002+189+1$\\
completion sectors&$63+63+63+1$\\
physical rows&$32$\\
physical matrix cycles&$64^3=262\,144$\\
\bottomrule
\end{tabular}
\end{center}

For reproducibility, the fixed-row records and the sorted physical-cycle
set have SHA-256 fingerprints
\begin{center}
\scriptsize\ttfamily
f7e57ddd96d71072e6bacdb9d7303a4378439d03a040994b45ab8c0a7afe76df
\end{center}
and
\begin{center}
\scriptsize\ttfamily
d8ff84641ab283b16eaae3ca75ac7335fce845ff6c165d5a0402125852be9083.
\end{center}
The verifier first reconstructs the old $2048$-class fibre, adjoins the
four solutions of $Q+R+S=0$, checks the five completion sectors, and finally
applies the row decoder of Proposition \ref{prop:orbit64}.  

\subsection{\texorpdfstring{$M_{32}$}{M32} terminal audit}\label{app:m32-terminal}
The terminal ledger supporting the current Theorem~\ref{thm:32-retained} is
summarized below.  These data correspond to the rank-$2048$ source
and its rank-$2128$ restoration.
\begin{center}
\begin{tabular}{ll}
\toprule
check&certified value\\
\midrule
packet types before puncturing&$22\CW_3+10\CW_2$\\
retained packet--support identities&$354$\\
fixed-row zero fibre&$2048=1954+31+31+31+1$\\
completion profile&$(31,31,31)$ plus one cubic\\
rank-four active blocks&$5$; identity $R_t=\mathcal A_{\rm can}+tJ$\\
repaired one-row source cost&$128$\\
restored one-row source cost&$133$\\
physical rows&$16$\\
physical matrix cycles&$16\cdot2048=32^3$\\
\bottomrule
\end{tabular}
\end{center}
The rank-$2048$ route replaces five complete singular-active cost-five
blocks by five source-resolved $\Diag_4$ blocks; the packet characters,
endpoint labels, completion tensor, row variables and physical decoder are
unchanged.  Reversing exactly these five replacements gives the rank-$2128$
source.  Thus the two parts of Theorem~\ref{thm:32-retained} share one
terminal incidence audit.

A convenient disjoint allocation of the thirty-one nonzero endpoint labels
is the following.  The column headed ``active'' is carried by the
singular-active $\CW_3$ blocks in the rank-$2128$ source and by their
rank-four $\Diag_4$ replacements in the rank-$2048$ source.
\begin{center}
\small
\setlength{\tabcolsep}{5pt}
\begin{tabular}{c|ccccc}
\toprule
sector&shared&active&tangent&exceptional&five-$\CW_3$ regular\\
\midrule
$X$&$1{:}15$&$16,17$&$18{:}23$&$24{:}30$&$31$\\
$Y$&$1{:}15$&$16{:}18$&$19{:}24$&$25{:}31$&--\\
$Z$&$1{:}15$&--&$16{:}24$&--&$25{:}31$\\
\bottomrule
\end{tabular}
\end{center}
Thus every one-bar sector uses each canonical label
$1,\ldots,31$ exactly once.  The row counts are respectively
$15+2+6+7+1=31$, $15+3+6+7=31$, and $15+9+7=31$, so the table realizes
precisely the endpoint profile $(31,31,31)$ in
\eqref{eq:endpoint-ledger32} without overlap.

\subsection{Characteristic-two base and transition audit}\label{app:char2-base-audit}
The characteristic-$2$ package used only for Theorem~\ref{thm:power-of-two}
is independent of the sign-character quotient above.  Its deterministic
ledger is
\begin{center}
\small
\setlength{\tabcolsep}{4pt}
\begin{tabular}{ll}
\toprule
check&certified value\\
\midrule
Hasse selector identities&$E_{A,B}(z_C)=\delta_{B,C}z_A$\\
level-$5$ fixed-row fibre&$2048=1954+31+31+31+1$\\
level-$5$ physical cycles&$16\cdot2048=32^3$\\
level-$6$ fixed-row fibre&$8192=8002+63+63+63+1$\\
level-$6$ physical cycles&$32\cdot8192=64^3$\\
raw/effective active labels&$\{2,5,6,10,15\}\xrightarrow{+6}\{0,3,4,9,12\}$\\
marked blocks at each base level&$5\Diag_4$\\
$D_4$ inventory after a two-bit extension&$80=5+5+5+11+54$\\
source directions used at the boundary&$84=44+5\cdot8$\\
terminal records&$64=44+5\cdot4$\\
fresh terminal address map&$H_4\times\F_2^2\simeq(\F_2^2)^3$\\
terminal multiplicity&$1$\\
\bottomrule
\end{tabular}
\end{center}
For a fixed old physical cycle $C$, every ordinary and active terminal record
has old projection $C$.  The final fresh-coordinate map is
\[
 (\lambda,g)\longmapsto
 (g,g+Q_\lambda,g+Q_\lambda+R_\lambda),
 \qquad
 \theta(\lambda)=(Q_\lambda,R_\lambda),
\]
with inverse
\[
 g=a,\qquad \lambda=\theta^{-1}(a+b,b+c).
\]
Thus the displayed multiplicity-one statement is an explicit bijection and
not merely a cardinality check.

\subsection{Interpolation weights for the reference field}\label{app:interp-weights}
Finally we give interpolation weights over the finite field \texorpdfstring{$\F_{311}$}{F311}
for the reference schedules with $(D,k)=(132,44)$.
The weights below use consecutive evaluation points $u_j=j$ and the
coefficient convention \eqref{eq:interp-extract}.  They are printed in full
so that coefficient extraction can be reproduced without inverting a
Vandermonde matrix.

\begingroup
\small
\setlength{\tabcolsep}{4.5pt}
\renewcommand{\arraystretch}{1.05}
\begin{center}
\begin{tabular}{@{}rrrrrrrrrr@{}}
\toprule
$j$&$\lambda_j$ & $j$&$\lambda_j$ & $j$&$\lambda_j$ & $j$&$\lambda_j$ & $j$&$\lambda_j$\\
\midrule
0 & 205 & 1 & 276 & 2 & 243 & 3 & 47 & 4 & 121\\
5 & 290 & 6 & 156 & 7 & 279 & 8 & 72 & 9 & 194\\
10 & 262 & 11 & 102 & 12 & 21 & 13 & 51 & 14 & 275\\
15 & 209 & 16 & 234 & 17 & 10 & 18 & 163 & 19 & 58\\
20 & 80 & 21 & 194 & 22 & 14 & 23 & 270 & 24 & 225\\
25 & 128 & 26 & 123 & 27 & 213 & 28 & 244 & 29 & 276\\
30 & 58 & 31 & 223 & 32 & 167 & 33 & 114 & 34 & 58\\
35 & 97 & 36 & 181 & 37 & 177 & 38 & 286 & 39 & 283\\
40 & 183 & 41 & 273 & 42 & 168 & 43 & 145 & 44 & 233\\
\bottomrule
\end{tabular}
\par\smallskip\textit{Interpolation weights for $0\le j\le44$.}
\end{center}
\medskip
\begin{center}
\begin{tabular}{@{}rrrrrrrrrr@{}}
\toprule
$j$&$\lambda_j$ & $j$&$\lambda_j$ & $j$&$\lambda_j$ & $j$&$\lambda_j$ & $j$&$\lambda_j$\\
\midrule
45 & 102 & 46 & 118 & 47 & 120 & 48 & 154 & 49 & 92\\
50 & 60 & 51 & 206 & 52 & 11 & 53 & 291 & 54 & 73\\
55 & 105 & 56 & 42 & 57 & 146 & 58 & 151 & 59 & 88\\
60 & 64 & 61 & 117 & 62 & 263 & 63 & 203 & 64 & 147\\
65 & 84 & 66 & 66 & 67 & 157 & 68 & 128 & 69 & 235\\
70 & 27 & 71 & 236 & 72 & 163 & 73 & 63 & 74 & 85\\
75 & 10 & 76 & 187 & 77 & 112 & 78 & 214 & 79 & 85\\
80 & 70 & 81 & 51 & 82 & 272 & 83 & 55 & 84 & 7\\
85 & 25 & 86 & 248 & 87 & 54 & 88 & 89 & 89 & 209\\
\bottomrule
\end{tabular}
\par\smallskip\textit{Interpolation weights for $45\le j\le89$.}
\end{center}
\medskip
\begin{center}
\begin{tabular}{@{}rrrrrrrrrr@{}}
\toprule
$j$&$\lambda_j$ & $j$&$\lambda_j$ & $j$&$\lambda_j$ & $j$&$\lambda_j$ & $j$&$\lambda_j$\\
\midrule
90 & 275 & 91 & 243 & 92 & 290 & 93 & 123 & 94 & 94\\
95 & 201 & 96 & 1 & 97 & 208 & 98 & 110 & 99 & 65\\
100 & 80 & 101 & 284 & 102 & 145 & 103 & 285 & 104 & 256\\
105 & 46 & 106 & 150 & 107 & 222 & 108 & 6 & 109 & 259\\
110 & 247 & 111 & 152 & 112 & 118 & 113 & 5 & 114 & 190\\
115 & 182 & 116 & 42 & 117 & 91 & 118 & 98 & 119 & 116\\
120 & 72 & 121 & 1 & 122 & 48 & 123 & 301 & 124 & 41\\
125 & 102 & 126 & 187 & 127 & 46 & 128 & 137 & 129 & 209\\
130 & 273 & 131 & 286 & 132 & 140 &  &  &  & \\
\bottomrule
\end{tabular}
\par\smallskip\textit{Interpolation weights for $90\le j\le132$.}
\end{center}
\endgroup

\section{Characteristic-three replacement certificates}\label{app:char3-certificates}

\subsection{Purpose and source-tensor replacements}
This appendix records the data used when $\operatorname{char}(\K)=3$.
Every entry displayed here belongs to the prime field $\F_3$, or to the
formal Laurent field $\F_3(\!(s)\!)$, and therefore defines the same map over
any field of characteristic $3$.  The binary Walsh and Weyl parts of the
proof require only that $2$ be invertible and are unchanged.

For the current main sources, $\CW_3$ and $\CW_2$ are the only
Coppersmith--Winograd tensors needed.  For $\CW_3$, Lemma
\ref{lem:cw-explicit} remains valid because its two denominators are
$q-2=1$ and $q-1=2$.  The $\CW_4$ data below are retained only as an
auxiliary certificate and are not used in Theorems~\ref{thm:m1024}--
\ref{thm:power-of-two}.  The six rows and weights in Lemma
\ref{lem:cw4-char3} satisfy, in $\F_3$, the formulas
\[
 \sum_{k=0}^{5}\ell_k=0,
 \qquad
 \sum_{k=0}^{5}\ell_kw_{ki}=0,
 \qquad
 \sum_{k=0}^{5}\ell_kw_{ki}w_{kj}=\delta_{ij}.
\]
These identities give the six-product realization coefficientwise.  For
$\CW_2$, the four-product identity \eqref{eq:cw2-char3} is an identity in
$\F_3(\!(s)\!)$.  Lemma \ref{lem:formal-truncation} replaces its finitely many
formal series by finite Laurent jets without changing the leading
coefficient.  Consequently, the source costs
\[
 \BR_{\K}(\CW_2)\le4,
 \qquad
 \BR_{\K}(\CW_3)\le5,
 \qquad
 \BR_{\K}(\CW_4)\le6
\]
are valid in characteristic $3$ with exactly the same product counts as in
all other odd characteristics.

The five-empty module uses the single high component
\[
 z^{(1)}_\infty\longmapsto\cdots+t^4\bar Z_H
\]
from Lemma \ref{lem:fiveC3paper32} and Appendix
\ref{app:m1024-split-cubic}.  Hence its unique cubic has coefficient one
without division by $3$.  All other terms involving this high component
have total Laurent degree $3$ or $6$.

\subsection{Exceptional endpoint maps in all odd characteristics}
For completeness we give the endpoint matrices used in Lemma
\ref{lem:exceptional32}.  In each block, concatenate source-colour columns in
the following order:
\begin{center}
\small
\begin{tabularx}{\textwidth}{c c c >{\raggedright\arraybackslash}X}
\toprule
$\chi$&barred side&endpoint sides&column order\\
\midrule
$6$&$X$&$Y,Z$&$23:0,1,2,\infty$; $30:0,1,2,\infty$\\
$7$&$X$&$Y,Z$&$7:0,1,2,3,\infty$; $27:0,1,2,\infty$\\
$10$&$Y$&$X,Z$&$8:0,1,2,\infty$; $1:0,1,2,\infty$\\
$11$&$Y$&$X,Z$&$24:0,1,2,3,\infty$; $4:0,1,2,\infty$\\
\bottomrule
\end{tabularx}
\end{center}
The barred covectors, in these orders, are
\begin{align*}
 b_6&=(-1,1,-1,3,0,-1,-1,-1),\\
 b_7&=(-1,1,-1,-1,4,0,-1,-1,-1),\\
 b_{10}&=(-1,1,-1,2,0,-1,0,-1),\\
 b_{11}&=(-1,1,-1,-1,3,0,-1,-1,-1).
\end{align*}
Contracting the declared direct sum of $\CW_2$ and $\CW_3$ packets against
$b_\chi$ gives $K_\chi$.  The matrices $D_L,D_R$ are the dense maps defined
by the packet restrictions in Appendix \ref{app:m1024-dense-data}.  Thus all
four matrices $K_\chi,D_L,D_R$ are fixed by the preceding source data.

When $\operatorname{char}(\K)\ne2,3$, one may take the following matrices.
The same matrices for $\chi=7,10$ may also be reduced modulo $3$.

\begingroup
\setlength{\arraycolsep}{2.4pt}
\scriptsize
For $\chi=6$,
\[
 E_L^{(6)}=\begin{pmatrix}
-\frac13&1&0&0&0&0&0&0\\
 \frac13&0&1&0&0&0&0&0\\
 \frac13&0&0&1&0&0&0&0
\end{pmatrix},\qquad
 E_R^{(6)}=\begin{pmatrix}
\frac12&-1&0&-1&\frac32&0&0&0\\
0&0&-1&1&0&0&0&0\\
-1&-1&1&-5&0&0&0&0
\end{pmatrix}.
\]
For $\chi=7$,
\[
 E_L^{(7)}=\begin{pmatrix}
-\frac14&1&0&0&0&0&0&0&0\\
 \frac14&0&1&0&0&0&0&0&0\\
 \frac14&0&0&1&0&0&0&0&0\\
 \frac14&0&0&0&1&0&0&0&0
\end{pmatrix},\quad
 E_R^{(7)}=\begin{pmatrix}
1&-1&0&0&-1&2&0&0&0\\
0&0&-1&0&1&0&0&0&0\\
0&0&0&-1&1&0&0&0&0\\
-1&-1&1&1&-7&0&0&0&0
\end{pmatrix}.
\]
For $\chi=10$,
\[
 E_L^{(10)}=\begin{pmatrix}
-\frac12&1&0&0&0&0&0&0\\
 \frac12&0&1&0&0&0&0&0\\
 \frac12&0&0&1&0&0&0&0
\end{pmatrix},\qquad
 E_R^{(10)}=\begin{pmatrix}
1&-1&0&-1&2&0&0&0\\
0&0&-1&1&0&0&0&0\\
-1&-1&1&-4&0&0&0&0
\end{pmatrix}.
\]
For $\chi=11$,
\[
 E_L^{(11)}=\begin{pmatrix}
-\frac13&1&0&0&0&0&0&0&0\\
 \frac13&0&1&0&0&0&0&0&0\\
 \frac13&0&0&1&0&0&0&0&0\\
 \frac13&0&0&0&1&0&0&0&0
\end{pmatrix},\quad
 E_R^{(11)}=\begin{pmatrix}
\frac12&-1&0&0&-1&\frac32&0&0&0\\
0&0&-1&0&1&0&0&0&0\\
0&0&0&-1&1&0&0&0&0\\
-1&-1&1&1&-6&0&0&0&0
\end{pmatrix}.
\]
\endgroup

In characteristic $3$, the following integral representatives are used;
all entries are read in $\F_3$.

\begingroup
\setlength{\arraycolsep}{2.8pt}
\scriptsize
For $\chi=6$,
\[
 \widetilde E_L^{(6)}=\begin{pmatrix}
1&0&0&0&0&0&0&0\\
0&1&1&0&0&0&0&0\\
0&1&0&1&0&0&0&0
\end{pmatrix},\qquad
 \widetilde E_R^{(6)}=\begin{pmatrix}
1&0&0&2&1&0&0&0\\
0&0&2&1&0&0&0&0\\
2&2&1&1&0&0&0&0
\end{pmatrix}.
\]
For $\chi=7$,
\[
 \widetilde E_L^{(7)}=\begin{pmatrix}
2&1&0&0&0&0&0&0&0\\
1&0&1&0&0&0&0&0&0\\
1&0&0&1&0&0&0&0&0\\
1&0&0&0&1&0&0&0&0
\end{pmatrix},\quad
 \widetilde E_R^{(7)}=\begin{pmatrix}
1&2&0&0&2&2&0&0&0\\
0&0&2&0&1&0&0&0&0\\
0&0&0&2&1&0&0&0&0\\
2&2&1&1&2&0&0&0&0
\end{pmatrix}.
\]
For $\chi=10$,
\[
 \widetilde E_L^{(10)}=\begin{pmatrix}
1&1&0&0&0&0&0&0\\
2&0&1&0&0&0&0&0\\
2&0&0&1&0&0&0&0
\end{pmatrix},\qquad
 \widetilde E_R^{(10)}=\begin{pmatrix}
1&2&0&2&2&0&0&0\\
0&0&2&1&0&0&0&0\\
2&2&1&2&0&0&0&0
\end{pmatrix}.
\]
For $\chi=11$,
\[
 \widetilde E_L^{(11)}=\begin{pmatrix}
1&0&0&0&0&0&0&0&0\\
0&1&1&0&0&0&0&0&0\\
0&1&0&1&0&0&0&0&0\\
0&1&0&0&1&0&0&0&0
\end{pmatrix},\quad
 \widetilde E_R^{(11)}=\begin{pmatrix}
1&0&0&0&2&1&0&0&0\\
0&0&2&0&1&0&0&0&0\\
0&0&0&2&1&0&0&0&0\\
2&2&1&1&0&0&0&0&0
\end{pmatrix}.
\]
\endgroup

For each $\chi$, direct matrix multiplication in the indicated field gives
\begin{equation}\label{eq:exceptional-all-characteristics}
 E_LK_\chi E_R^{\mathsf T}=I_r,
 \
 E_LK_\chi D_R^{\mathsf T}=0,
 \
 D_LK_\chi E_R^{\mathsf T}=0,
 \
 D_LK_\chi D_R^{\mathsf T}=0,
\end{equation}
where the tilded pair is used in characteristic $3$.  These identities are
precisely the coefficient conditions in Lemma \ref{lem:exceptional32}.  They
also show that the intermediate pivot determinant $1/10$ in the
$\chi=7$ search is not part of the final map and imposes no exclusion of
characteristic $5$.

\subsection{Compatibility with the source-resolved \texorpdfstring{$1024$}{1024} assembly}
\label{app:char3-m1024-compat}
We now record explicitly why the characteristic-three replacements above
remain compatible with the outer-first source-resolved construction of
Section \ref{sec:m1024}.  This subsection concerns an arbitrary field
$\K$ of characteristic $3$; no finiteness assumption on $\K$ is made.
Since the prime subfield $\F_3$ embeds canonically in $\K$, every matrix
with entries in $\F_3$ and every Laurent-series identity over
$\F_3(\!(s)\!)$ base-changes coefficientwise to $\K$ and
$\K(\!(s)\!)$, respectively.

First, the one-row source of Theorem \ref{thm:m1024} is
\[
 \Sstar=14\CW_3\oplus4\CW_2\oplus7\Wc\oplus5\Diag_4\oplus\Diag_8.
\]
For $\CW_3$ the generic five-product formula of Lemma
\ref{lem:cw-explicit} remains valid because its denominators are $1$ and
$2$.  For each $\CW_2$ use the four-product formal identity
\eqref{eq:cw2-char3}; Lemma \ref{lem:formal-truncation} turns the finitely
many required jets into fixed Laurent-polynomial side maps.  The corner
and diagonal summands use only integral coefficients.  Hence the source
budget remains
\[
 14\cdot5+4\cdot4+7\cdot2+5\cdot4+8=128.
\]

Second, the source-resolved active lift of Lemma
\ref{lem:active-lift-1024}, including the cancellation in
\eqref{eq:active-lift-deg-1024}, uses only the coefficients $0,1,-1$.
The character-$6$ diagonal module of Section \ref{sec:m1024} likewise uses
only integral source maps before the final physical identification of its
two $C_1$ records.  Thus both modules reduce unchanged modulo $3$.  In
particular, the two source records $C_1^{(23)}$ and $C_1^{(30)}$ remain
distinct until the module--chart projection; their eventual coefficient
$2=-1$ is nonzero in characteristic $3$.

Third, the crossed five-empty construction is defined over
$\mathbb Z[1/2]$.  Reducing its matrices $B,D,c,J_{10}$ modulo $3$ gives
\[
 cJ_{10}c^{\mathsf T}=1,\qquad
 BJ_{10}D^{\mathsf T}=I_7,\qquad
 BJ_{10}c^{\mathsf T}=DJ_{10}c^{\mathsf T}=0.
\]
For the packet-character twists, let $M_\delta$ be the seven-dimensional
matrices defined in Appendix \ref{app:m1024-split-cubic}.  Their group law
\[
 M_\delta M_\epsilon=M_{\delta+\epsilon}
\]
is an integral identity after inverting $2$ and therefore remains valid in
characteristic $3$.  The simultaneous change of basis used there is
unimodular, with determinant $-1$, so it remains invertible over $\K$ and
still diagonalizes the representation as
\[
 \chi_3^{\oplus2}\oplus\chi_4^{\oplus2}
 \oplus\chi_9^{\oplus2}\oplus1.
\]
Consequently the fixed endpoint translations absorb the same seven extra
characters, and all negative Laurent coefficients and mixed endpoint
pairings vanish exactly as in odd characteristic different from $3$.

Fourth, all binary Walsh and Weyl identities used in Sections
\ref{sec:m1024} and Appendix \ref{app:m1024-weyl} require only that powers
of $2$ be invertible.  In characteristic $3$ the values $1$ and $-1=2$ are
distinct and
\[
 W_{2^k}W_{2^k}^{\mathsf T}=2^kI
\]
has invertible scalar $2^k$.  Hence the $H_4$ absolute-character closure,
the source-derived $W_{32}$ module--chart phase, the independent
$\Diag_{32}$ channel, the $H_5$ completion row, and the final $H_9$ Weyl
closure are unchanged.  In particular, the same $4096$ integer
module--chart identities may be reduced modulo $3$ and retain their
normalization because $8$ is invertible.

Finally, in the $1024$ route, among the
exceptional endpoint quotients only the characters $7,10,11$ are used in
Theorem \ref{thm:m1024}; the corresponding tilded matrices displayed above
satisfy \eqref{eq:exceptional-all-characteristics} over $\F_3$ and hence
over every characteristic-three field.  The character-$6$ tilded matrices
are retained for the separate $32\times32$ constructions.

Combining these observations, the complete three-side map of Appendix
\ref{app:m1024-assembly} is defined over $\K(\!(t)\!)$ for every field
$\K$ of characteristic $3$ and has the same leading coefficient
$\MM_{1024}$ and the same source budget $2^{21}$.  To see explicitly that
this certificate is bounded-degree, first apply Lemma
\ref{lem:formal-truncation} to the four $\CW_2$ source realizations and make
fixed monomial shifts so that the resulting $128$-product source identity
is polynomial,
\[
 \Phi_{128}^{(3)}(s)=s^e\Sstar+O(s^{e+1}),
\]
with fixed leading order $e$ and fixed degree $E$.  After the common $t^2$
shift, the top source-resolved map has $t$-degree at most $18$ and target
order $6$.  Choosing any fixed $N>18$ and substituting $s=t^N$ gives a
single-parameter polynomial degeneration whose target term is
$ t^{6+Ne}\MM_{1024};$
all source or top-level error terms have strictly larger degree.  This is the
bounded-degree characteristic-three certificate required in the proof of
Theorem \ref{thm:m1024} and in Proposition \ref{prop:bini}.  No assumption
that $\K$ is finite is used.

\subsection{Why the Walsh certificates do not cover characteristic two}
For $H=\F_2^k$, the matrix
\[
 W_H=\bigl((-1)^{\ip{u}{v}}\bigr)_{u,v\in H}
\]
becomes the all-one matrix in characteristic $2$ and has rank one.  Moreover,
the normalizing scalar $|H|=2^k$ vanishes.  Thus the Walsh projectors and the source-derived $W_{32}$ bridge used in
Theorems~\ref{thm:m1024}, \ref{thm:finite-restrictions}, and
\ref{thm:32-retained} do not reduce to characteristic $2$.  This explains
their hypothesis $\operatorname{char}(\K)\ne2$.  The full dyadic theorem
uses a different mechanism in characteristic $2$: the Hasse matrix units of
Lemma~\ref{lem:char2-hasse-matrix-units}, the base states of
Lemma~\ref{lem:char2-dyadic-bases}, and the source-local propagation of
Lemma~\ref{lem:char2-two-step-propagation}.

\Ack. This research was done with the aid of AI under the author's guidance. In particular,
all the tables and numerical data were produced by AI.

%%%%%%%%%%%%%%%%%%%%%%%%%%%%%%%%%%%%%%%%%%%%%%%%%%%%%%%

 \end{document}